\documentclass[12pt,a4paper]{article}
\usepackage{amsmath}
\usepackage{graphicx}
\usepackage{epsfig}%
\usepackage{amsfonts}%
\usepackage{amssymb}
\usepackage{algpseudocode}						% algorithm envoronment, pseudocode
\usepackage{algorithm}	
\usepackage{hhline}

\usepackage{xcolor}
\usepackage{mathtools}
\usepackage{diagbox}						% table cells with diagonal subdivision
\usepackage{subfigure}
\usepackage{nicefrac}

\newcommand{\cn}{\color{black}}

\newtheorem{theorem}{Theorem}[section]

\newtheorem{remark}[theorem]{Remark}

\newenvironment{proof}[1][Proof]{\noindent \emph{#1.} }{\hfill \
\rule{0.5em}{0.5em}}

\makeatletter\@addtoreset{equation}{section}\makeatother
\makeatletter\@addtoreset{figure}{section}\makeatother
\makeatletter\@addtoreset{table}{section}\makeatother
\newcommand{\dotprod}[2]{\langle#1,#2\rangle} 					% dot product
\newcommand{\R}{{\mathbb{R}}}       							% real numbers
\newcommand{\nell}{{n_{(\ell)}}}
\newcommand{\aplus}{a_{i+\frac{1}{2}}}
\newcommand{\aminus}{a_{i-\frac{1}{2}}}
\newcommand{\yplus}{y_{i+1}}
\newcommand{\yminus}{y_{i-1}}
\newcommand{\pminus}{p_{i-1}}
\newcommand{\A}{{\cal A}}
\newcommand{\uu}{\mathbf{u}}
\newcommand{\y}{\mathbf{y}}
\newcommand{\yom}{\mathbf{y}_\varOmega}

\DeclareMathOperator{\diag}{diag}								% diag as an operator

\makeatletter
\let\@@citation@@=\citation
\renewcommand{\citation}[1]{\@@citation@@{#1}%
\@for\@tempa:=#1\do{\@ifundefined{cit@\@tempa}%
  {\global\@namedef{cit@\@tempa}{}}{}}%
}

\def\@lbibitem[#1]#2#3\par{%
  \@ifundefined{cit@#2}{}{\item[\@biblabel{#1}\hfill]}%
  \if@filesw
      {\let\protect\noexpand
       \immediate
       \write\@auxout{\string\bibcite{#2}{#1}}}\fi\ignorespaces
  \@ifundefined{cit@#2}{}{#3}}
\def\@bibitem#1#2\par{%
  \@ifundefined{cit@#1}{}{\item}%
  \if@filesw \immediate\write\@auxout
    {\string\bibcite{#1}{\the\value{\@listctr}}}\fi\ignorespaces
  \@ifundefined{cit@#1}{}{#2}}
\makeatother
 
\begin{document}

\title{Tensor Method for Optimal Control Problems Constrained by Fractional  3D Elliptic
Operator with Variable Coefficients}
 
\author{Britta Schmitt\thanks{University of Trier,
        FB 4 - Department of Mathematics, D-54296, Trier,
        Germany ({\tt schmittb@uni-trier.de}).}
        \and
        Boris N. Khoromskij\thanks{Max-Planck-Institute for
        Mathematics in the Sciences, Inselstr.~22-26, D-04103 Leipzig,
        Germany ({\tt bokh@mis.mpg.de}); University of Trier,
        FB 4 - Department of Mathematics, D-54296, Trier.} 
        \and
        Venera Khoromskaia\thanks{Max Planck Institute for
        Mathematics in the Sciences, Leipzig, Germany ({\tt vekh@mis.mpg.de}).}
        \and        
        Volker Schulz\thanks{University of Trier,
        FB 4 - Department of Mathematics, D-54296, Trier,
        Germany ({\tt volker.schulz@uni-trier.de}).}
        }

\date{}

\maketitle

\begin{abstract}
We introduce the tensor numerical method for solving optimal control problems that are 
constrained by fractional  2D and 3D elliptic operators 
with variable coefficients. We solve the governing equation for the control function 
which includes a sum of the fractional operator and its inverse, both discretized 
over large 3D $n\times n \times n$ spacial grids. 
 Using the diagonalization of the arising matrix valued functions in the eigenbasis of the 1D Sturm-Liouville 
operators, we construct the rank-structured tensor approximation with controllable precision 
for the discretized fractional elliptic operators and the respective preconditioner. 
The right-hand side in the constraining equation (the optimal design function) is supposed to be represented 
in a form of a low-rank canonical tensor. Then the equation for the control function is solved  
in a tensor structured format by 
using preconditioned CG iteration with the adaptive rank truncation procedure that also ensures the accuracy 
of calculations, given an $\varepsilon$-threshold. 
This method reduces the numerical cost for solving the control problem to $O(n \log n)$ 
(plus the quadratic term $O(n^2)$ with a small weight), which is superior to the 
approaches based on the traditional linear algebra tools that yield at least $O(n^3 \log n)$ 
complexity in the 3D case. 
The storage for the representation of all 3D nonlocal operators and functions involved 
is  also estimated by $O(n \log n)$.  
This essentially outperforms the  
traditional methods operating with  fully populated  $n^3 \times n^3$ matrices and vectors 
in $\mathbb{R}^{n^3}$. 
Numerical tests for 2D/3D control problems indicate the almost linear complexity  
scaling of the rank truncated PCG iteration in the univariate grid size $n$.
\end{abstract}



%\hspace*{-6mm}\textbf{Keywords} Fractional operators, low-rank tensor approximations, optimal control problems.

\section{Introduction}\label{Int:SH} 
 
 Optimization problems that are constrained by partial differential equations (PDEs) have a long 
  history in mathematical literature since they allow a huge number of applications  in 
  different fields of natural science,  see   
    \cite{Troeltzsch:2005,Allaire:07,HerKun:2010} for some comprehensive examples. 
  Being studied for many years, 
  tracking-type problems which trace the discrepancy between the solution of the PDE and a given 
  target state represent a very important class of optimal control problems \cite{Reyes:2015}. 
  In such problems the discretization and numerical treatments of the elliptic PDE in constraints that 
  determines the relation between the optimal design and control functions,
  $$\mathcal{L} y = u,
  $$
  can be performed by the traditional FEM methods dealing with sparse matrices.
  Multigrid methods for elliptic equations are shown to be efficient since their 
computational complexity is linear in the number of grid points in the computational 
domain in $\mathbb{R}^d$, see \cite{BoSch:2009,BoSch:2012}.
  
  The construction of suitable discretization and solution schemes become a challenging task
  whenever the operator $\mathcal{L}$ in the constraint equation inherits nonlocal structures, 
  such as for instance in the case of a fractional differential operator, that is $\mathcal{L}^\alpha$. 
     In recent years, considering fractional PDEs  
  is gaining more attention due to an higher accuracy within the numerical simulation of the real 
  world problems, for example in the subject areas concerning heat diffusion in special 
  materials \cite{BenAbdallah:2013}, image processing \cite{GilOsher:2008}, material 
  science \cite{Bates:2006}, optimization \cite{Duvaut:1976}. For more application areas 
  see \cite{AntilOtarola:2015} and references therein. 
  
  However, as the pay-off for higher modeling accuracy in 
  applications, fractional operators in PDEs also imply nonlocality to the given equation, 
  which after discretization leads to dense problem structures resulting in quadratic complexity 
  in the number of degrees of freedom in $\R^d$, so that they are hard to handle especially when 
  large grids are considered in many dimensions. As a result, common numerical solution approaches lead to severe 
  problems, such that a number of special techniques has been advocated 
  \cite{HaHighTrefeth:08,Vabich:2015,DuLazPas:18,HLMMV:2018,Schwab:18}.
  
    Furthermore, the possible presence of variable coefficients within the PDE has a 
  huge impact on the numerical complexity of a suitable solution technique. 
  In this case, assuming separability for the involved 
  coefficients in  $\R^d$ leads to favorable, highly structured matrices after discretization that 
  allow using  efficient numerical representations such as Kronecker product structures 
  \cite{KhorCA:09,KreTob:2010, KreTob:2011,DolOs:2011} and the respective multilinear algebra.

  In this article, we consider a tracking-type optimal control problem constrained by 
  %PDE operator $\mathcal{L}$ to be an elliptic Laplacian type operator, more precisely we   consider 
  a fractional Laplace type elliptic operator with variable, separable coefficients 
   discretized on a tensor grid.   
  For an overview about several characterizations of the fractional elliptic operators 
  and the respective algebra
  see \cite{GHK:05,Higham_MatrFunc:08,HaHighTrefeth:08,Kwasn:17,FracLapl:2018}.
   A number of application fields motivating the use of fractional power of elliptic operators,
 for example in biophysics, mechanics, nonlocal electrostatics and image processing 
 have been discussed in the literature
 \cite{AtPSZo:14,AntilOtarola:2015,DPSS:2014,Vabich:2015,DuLazPas:18,HLMMV:2018,KaMelenk:19}. 
 In such applications control problems arise naturally.

   An application of standard numerical methods for the solution of PDEs in $\mathbb{R}^d$ 
  is essentially limited by 
  the so-called \textit{curse of dimensionality} \cite{bellman-dyn-program-1957},   
  that is the effect of an exponential growth of storage and computational complexity, $O(n^d)$, 
  in the dimension of the problem $d$, where $n$ is the univariate grid size of the discretization.   
   This phenomenon effects   all  basic procedures such as matrix-vector calculus 
  and  full format matrix arithmetics.  
    Some special numerical techniques like adaptive $\mbox{h-p}$ mesh refinement, 
    sparse grids \cite{BunGri:2004}, 
   hierarchical matrices \cite{Hack_Book:09} and fast multipole methods \cite{RochGreen:87}
    only partially relax the curse of dimensionality.

The modern tensor numerical   methods, based on low-rank separable approximation of 
   operators and functions in $\mathbb{R}^d$,
  are capable to reduce the numerical cost in higher dimensions to the linear scaling in $d$,
  $O(dn)$, thus making possible the efficient numerical modeling in higher dimensions.
%   The recent tensor numerical methods circumvent the curse of dimensionality. 
  Meanwhile, tensor decompositions using canonical and Tucker tensor formats 
  and the respective multilinear algebra techniques have been used since long in the computer science
  for data analysis and signal processing \cite{Tuck:1966,DMV-SIAM2:00,Cichocki:2002,KoldaB:07}. 
  The Tucker approximation tools are based on the principal generalization of the singular value 
  decomposition (SVD) called the higher order SVD (HOSVD) \cite{DMV-SIAM2:00}.
  However, the analytical methods of separable low-rank representation of multivariate 
  functions and operators \cite{Stenger,HaKhtens:04I,GHK:05,Khor1:06}
  appeared to be the main prerequisite to tensor numerical calculus.
  Tensor numerical methods in scientific computing have been first introduced for calculation of the
  3D convolution integral operators with Green's kernels in quantum chemistry, 
  see  \cite{ Khor2Book2:18} for the detailed discussion.
  
  In the recent decade tensor formats created specifically for the solution of multidimensional 
  problems have been introduced in the form of tensor train  
  and quantized tensor train representations, hierarchical tensors,   
  and the recent range-separated tensor decomposition.  
  Finally, the main feature of the tensor numerical techniques in scientific computing
   is the  arrangement  and usage of separable data structures in order to reduce the solution of given 
multidimensional equations to essentially one-dimensional operations, 
see \cite{Hack_Book12,KhorBook:18,Khor2Book2:18} for the detailed discussion and comprehensive references.

In the present paper, we proceed with the development of efficient tensor numerical techniques for the solution of 
optimal control problems  which were initiated in \cite{HKKS:18} for solving  
control problem with 2D and 3D classical  fractional Laplace operator in constraints.
%classical  fractional Laplace operator $(-\Delta)^\alpha$ in constraints 
In the latter case one could use the FFT-based diagonalization of the fractional 
Laplacian to construct the rank-structured representation of the governing nonlocal operator.
 In our problem setting we essentially generalize the previous approach to 
  a more general class of tracking-type optimal control problems in $\mathbb{R}^d$, $d=2,3$,
constrained by a PDE containing a fractional divergence type elliptic operator 
with variable coefficients. 

Since the FFT-based factorizations may no longer be applied, 
we use the diagonalization of the arising fully populated matrix valued functions 
in the eigen-basis of the 1D Sturm-Liouville operators. 
%?? First, the problem is discretized on $n\times n \times n$ Cartesian grid in $\mathbb{R}^3$. ??  
The diagonal of the coefficients matrix in the factorized 
representation of the resulting governing operator (in the product of 1D eigen-bases)
is reshaped to a third-order tensor which undergoes the rank-structured decomposition 
by using the multigrid Tucker-to-canonical tensor transform  \cite{KhKh3:08}.  
The spectrally equivalent preconditioner with small Kronecker rank ($K$-rank) is constructed by using the 
 decomposition of the fractional anisotropic Laplace operator in the Fourier basis, 
 and by subsequent low-rank approximation via the multigrid Tucker algorithm.
Then the discrete linear system is solved by the  preconditioned conjugate gradient (PCG) iteration 
with adaptive rank truncation adapting the 
rank-structured tensor representation of all involved quantities: the governing operator, 
the preconditioner, the right hand side as well the solution vector. 
The adaptive rank reduction for the solution vector in the course of PCG iteration 
is performed  by using the canonical-to-Tucker decomposition via the reduced HOSVD (RHOSVD), 
see \cite{KhKh3:08,Khor2Book2:18}. The numerical accuracy is controlled by the 
given $\varepsilon$-threshold in the rank reduction procedure of the algorithm RHOSVD.

% 
% we make use of special tailored tensor techniques in order to  apply above methods. 
% When dealing with two-dimensional equations, we make use of rank-structuring reduced singular 
% value decomposition that provides efficient computations  with rank truncation. \cn 
% However, for three- and high-dimensional 
% problems, special rank-structuring tensor techniques such as the Tucker decomposition that serves 
% as higher order singular value decomposition and its combination with the canonical tensor format 
% provide the efficient decompositions with which the operator rank can be controlled and adjusted 
% whenever needed. The theoretical justification for the usage of aforementioned tensor approximations 
% can be found  in \cite{HKKS:18} and \cite{GHK:05}. 

Finally, we apply the above methods to solve a tracking-type optimal control problem 
constrained by a PDE with a fractional divergence type elliptic operator 
with variable coefficients  discretized on $n\times n \times n$ Cartesian grid in $\mathbb{R}^3$. 
Our approach exhibits the linear complexity scaling in dimension, $O(d n^2)$.  
The theoretical justification for the use of aforementioned tensor approximations 
  can be found  in  \cite{HaKhtens:04I,GHK:05,KhorBook:18}, see also   
    \cite{HKKS:18} where the classical  fractional Laplace operator $(-\Delta)^\alpha$ in constraints 
was discussed and analyzed.
   
  The rest of the paper is structured as follows. Section \ref{Int2:SH} describes the target 
class of optimal control problems and the special definitions that come with 
a separable structure of the involved fractional elliptic operator. 
In section \ref{sec:Larrange_eqn} the discrete optimality equation system is 
derived with the help of the Lagrangian multiplier approach. Afterwards, the discretization 
of the control problem is developed and the Laplace-type operator with variable coefficients is 
decomposed with the help of factorized low-rank structures. 
We also recapitulate the low-rank structures used for the preconditioner. 
The presented techniques are then used in a special tailored  
PCG algorithm  with adaptive rank truncation.   
Finally, in section \ref{sec:NumTests} we present 
and analyze our numerical results for 2D and 3D examples and discuss aspects 
related to computation and storage complexities when solving the considered  
class of optimal control problems by using tensor numerical method combined with
the introduced PCG iterative algorithm.

 \section{Problem setting } \label{Int2:SH}  
 
 Our goal is the construction of fast tensor numerical solution schemes for 
solving the control problems
constrained by fractional $d$-dimensional elliptic operators with variable coefficients. 
For this reason we restrict ourself to the case of rectangular domains and to the class
of elliptic operators with diagonal separable coefficients. 

Given the design function $y_\Omega \in L^2(\Omega)$ on $\Omega :=(0,1)^d$, $d=1,2,3$,
first, we consider the optimization problem for the cost functional
\begin{equation} \label{eqn:cost_func}
 \min_{y,u} J(y,u):=\int_\Omega (y(x) -y_\Omega(x))^2\, dx + \frac{\gamma}{2} 
 \int_\Omega u^2(x) \,dx, % \; \mapsto \; \min,
\end{equation}
constrained by the elliptic boundary value problem  
in $\Omega$ for the state variable  $y \in H_0^1(\Omega)$,
\begin{equation} \label{eqn:basic_setting}
  {\cal A} y  := -\nabla^T \cdot \mathbb{A}(x)\nabla y = \beta u, \quad x\in \Omega,\;
  u\in L_2(\Omega),\quad \beta >0,
  \end{equation}
endorsed with the homogeneous Dirichlet (or periodic) boundary conditions 
on $\Gamma = \partial \Omega $, i.e., $y_{|\Gamma}=0$.
The coefficient matrix $\mathbb{A}(x)\in \mathbb{R}^{d\times d}$ is supposed to 
be symmetric, positive definite and uniformly bounded in $\Omega $ 
with positive constants $c>0$ and $C>0$, i.e.,
\[
 c\, I_{d\times d}   \leq \mathbb{A}(x)\leq C\, I_{d\times d}.
\]
Under above assumptions the associated bilinear form
\[
 A(u,v)=\int_\Omega \mathbb{A}(x)\nabla u(x) \cdot \nabla v(x) \, dx
\]
defined on $V\times V $, $V:=\{v\in H_0^1(\Omega) \}$ is symmetric, coercive and bounded on $V$
with the same constants $c$ and $C$. 

In what follows, we consider for $0< \alpha \leq 1$ the control problems constrained by the fractional 
elliptic operator 
\begin{equation}
\label{eqn:frac_setting}
 {\cal A}^\alpha y= \beta u, \quad u\in L_2(\Omega),\quad \beta >0
\end{equation}
with the spectral definition of the fractional power of the elliptic operator ${\cal A}$ by
\begin{equation*} \label{eqn:frac_Oper}
 {\cal A}^\alpha y(x) = \sum_{i=1}^\infty \lambda_i^\alpha c_i \psi_i(x), \quad 
 y=\sum_{i=1}^\infty c_i \psi_i(x),
\end{equation*}
where $\{\psi_i(x)\}_{i=1}^\infty$ is the set of $L_2$-orthogonal eigenfunctions
of the symmetric, positive definite operator ${\cal A}$, while $\{\lambda_i\}_{i=1}^\infty$ are 
the corresponding (real and positive) eigenvalues. 

The numerical efficiency of this approach is based on the exponentially fast 
convergence of the sinc quadratures on a class of analytic functions. 
This technique is to be applied in the present paper for the theoretical 
analysis of the rank decomposition schemes, for the description of their
constructive representation and for the design of spectrally close preconditioners 
for the governing equations based on the use of fractional Laplacian.

The linear constraint equation (\ref{eqn:frac_setting}) allows to eliminate the state 
variable and then to derive the Lagrange 
equation for the control $u$ in the explicit form as follows 
(see \S\ref{sec:Larrange_eqn} concerning the Lagrange equations)
\begin{equation} \label{eqn:Lagrange_cont2}
   {\cal F}(\mathcal{A}) u :=   \big(\beta\mathcal{A}^{-\alpha} + 
 \tfrac{\gamma}{\beta}\mathcal{A}^{\alpha}\big) u = y_\varOmega,
\end{equation}
for some positive constants $\beta>0$ and $\gamma>0$. The equation  
for the state variable reads
\begin{equation} \label{eqn:State2}
      y=  \beta\mathcal{A}^{-\alpha} u.
%\big(\mathcal{I} + \tfrac{\gamma}{\beta^2}\mathcal{A}^{2\alpha}\big)^{-1} = y_\varOmega.
\end{equation}
The practically interesting range of parameters includes the case $\beta=O(1)$ for small 
values of $\gamma>0$.

The presented numerical scheme is focused on the solution
of discrete versions of equations (\ref{eqn:Lagrange_cont2}) and (\ref{eqn:State2}) that include a nonlocal 
elliptic operator $\mathcal{A}^{\alpha}$ and its inverse $\mathcal{A}^{-\alpha}$. 
The efficiency of the rank-structured tensor approximations presupposes 
that the design function on the right-hand side of these equations,
$y_\varOmega(x_1,x_2, x_3)$, allows the low-rank separable approximation.

Since we aim for the low-rank (approximate) tensor representation of all functions and 
operators involved in the above formulation of the control problem, 
we further assume that the equation coefficients matrix takes a diagonal form
\begin{equation} \label{eqn:DiagCoef2}
 \mathbb{A}(x) = \mbox{diag}\{a_1(x_1),a_2(x_2) \}, \quad a_\ell(x_\ell) >0, 
 \quad \ell=1,2,
\end{equation} 
in 2D case and similar for $d=3$,
\begin{equation} \label{eqn:DiagCoef3}
 \mathbb{A}(x) = \mbox{diag}\{a_1(x_1),a_2(x_2),a_3(x_3) \}, \quad a_\ell(x_\ell) >0, 
 \quad \ell=1,2,3.
\end{equation}
Hence, we consider the optimal control problem constrained by fractional elliptic operator with 
separable coefficients, such that the equation \eqref{eqn:frac_setting} takes form
\begin{equation} \label{eqn:CoefClass}
 {\cal A}^\alpha  y =
 \left(-\sum_{\ell=1}^d \frac{\partial}{d x_\ell} a_\ell(x_\ell) \frac{\partial}{d x_\ell}
 \right)^\alpha y = \beta u, 
 \quad  0\leq \alpha\leq 1.
\end{equation}
Notice that the efficient tensor numerical method for the case of fractional Laplace operator
in constraint, i.e., 
$a_\ell(x_\ell)=1$, $\ell=1,2,3$, was developed in \cite{HKKS:18}.

For a rank structured representation of the elliptic operator inverse (and some other operator 
valued functions $f({\cal A})$) one can apply 
the integral representation based on the Laplace type transform \cite{HaHighTrefeth:08},
\begin{equation} \label{eqn:Laplace_transf}
 {\cal A}^{-\alpha}=\frac{1}{\Gamma(\alpha)} \int_0^\infty t^{\alpha -1} e^{-t {\cal A}}\, dt, 
 \end{equation} 
which suggests the numerical schemes for low-rank canonical tensor representation of the 
operator (matrix) ${\cal A}^{-\alpha}$ by using the
sinc quadrature approximations for the integral on the real axis \cite{GHK:05},
\begin{equation} \label{eqn:sinc_Laplace_tr}
  \int_0^\infty t^{\alpha -1} e^{-t {\cal A}}\, dt \approx 
 \sum\limits_{k=-M}^M \hat{c}_k t_k^{\alpha -1} e^{-t_k {\cal A}}= \sum\limits_{k=-M}^M c_k 
\bigotimes_{\ell=1}^d e^{-t_k {A_\ell}},
\end{equation}
applied to the operators composed by a sum of commutable terms, 
\begin{equation} \label{eqn:Commut_term}
{\cal A}=\sum_{\ell=1}^d A_\ell, \quad [A_\ell,A_k]=0,\quad \mbox{for all} \quad \ell,k=1,\ldots,d,
\end{equation}
which ensures that each summand in (\ref{eqn:sinc_Laplace_tr}) is separable, i.e.
$e^{-t_k {\cal A}}=\bigotimes_{\ell=1}^d e^{-t_k {A_\ell}}$.
For example, in the case of the target class of operators in (\ref{eqn:CoefClass}),   (\ref{eqn:Commut_term}),  
we obtain the $d$-term decomposition
\[
 {\cal A}   = -\sum_{\ell=1}^d \frac{\partial}{d x_\ell} a_\ell(x_\ell) \frac{\partial}{d x_\ell}
\]
with \emph{commutable} 1D operators  $A_\ell=-\frac{\partial}{d x_\ell} a_\ell(x_\ell) \frac{\partial}{d x_\ell}$,
for $\ell=1,\ldots, d$.  

  In this paper we consider the discrete matrix formulation of the  optimal control problem 
  (\ref{eqn:cost_func}) constrained by (\ref{eqn:basic_setting})
 based on a FEM/FDM discretization ${A}_h$ of a $d$-dimensional elliptic operator $\mathcal{A}$ 
 defined on the uniform $n_1\times n_2 \times \ldots \times n_d  $ tensor grid in $\Omega$, 
 where $h=h_\ell=1/n_\ell$ is the univariate mesh parameter.
 The $L_2$ scalar product 
 is substituted by the Euclidean scalar product $(\cdot,\cdot)$ of multi-indexed vectors 
  in $\mathbb{R}^{\bf n}$, ${\bf n}=(n_1, n_2, \ldots, n_d)$.
 
%  
%   We use the collocation grid representation of functions, such that the $L_2$ scalar product 
%   will be substituted by the Euclidean scalar product $(\cdot,\cdot)$ of vectors 
%   in $\mathbb{R}^{\bf n}$, ${\bf n}=(n_1, n_2, \ldots, n_d)$.
 
 The fractional elliptic operators ${\cal A}^{\pm\alpha}$ are approximated
  by their FEM/FDM representation $({A}_h)^{\pm\alpha}$, where the matrix $({ A}_h)^{\pm \alpha}$ is defined
  in terms pf spectral decomposition of ${A}_h$.
  % where the eigenpairs for the corresponding grid discretization 
%   ${A}_h$ of the elliptic operator ${\cal A}$ are used.  
%   Here $h=h_\ell=1/n_\ell$ is the univariate mesh parameter.  
% %   The FEM/FDM approximation theory for fractional powers of elliptic operator was presented in 
%  \cite{DuLazPas:18}, see also literature therein. 

 \section{Rank-structured representation of operator valued functions of interest}
\label{sec:Larrange_eqn}

The numerical treatment of minimization problem \eqref{eqn:cost_func} with constraint 
\eqref{eqn:frac_setting} requires
solving the corresponding Lagrange equation for the necessary optimality conditions. 
In this section, we will derive these conditions
based on a discretize-then-optimize-approach. Afterwards, we will discuss how the involved discretized
operators can be applied efficiently in a low-rank format, and how this can be used to design
a preconditioned conjugate gradient (PCG) scheme for the necessary optimality conditions.

\subsection{Discrete optimality conditions}\label{ssec:Optim_cond}

We present the necessary first order conditions based on their derivation in \cite{HKKS:18}.\\
We consider a version of the control problem \eqref{eqn:cost_func} constrained by \eqref{eqn:frac_setting}, 
discretized on a uniform grid in each dimension, 
\begin{align*}
	\min_{\mathbf{y},\mathbf{u}} =  \ &\frac{1}{2}(\mathbf{y} - \mathbf{y}_\Omega)^TM(\mathbf{y} 
	- \mathbf{y}_\Omega)  + \frac{\gamma}{2} \mathbf{u}^T M \mathbf{u}\\
	\text{s.\,t.}\ ~ A^\alpha \mathbf{y} = \ & \beta M\mathbf{u},
\end{align*}
where vectors $\y, \y_\Omega, \uu \in \R^N$ denote the discretized state $y$, design $y_\Omega$ 
and control $u$, respectively. The discretization 
of the elliptic operator $\mathcal{A}$ by finite elements or finite differences is denoted by $A = A_h$. 
 The matrix $M$ will be a mass matrix in the finite element case and simply the 
identity matrix in the finite difference case. 

Setting up the Lagrangian function with the help of the discrete adjoint variable $\mathbf{p}$, 
\begin{equation*}
	L(\mathbf{y},\mathbf{u},\mathbf{p}) = \frac{1}{2}(\mathbf{y} - 
	\mathbf{y}_\Omega)^T M(\mathbf{y} - \mathbf{y}_\Omega)  
	+ \frac{\gamma}{2} \mathbf{u}^T M \mathbf{u} + \mathbf{p}^T(A^\alpha \mathbf{y} - 
	\beta M \mathbf{u}),
\end{equation*}
and deriving it with respect to all three variables, we end up with the equation system
\begin{equation*}
	\begin{bmatrix}
		M & O &  A^\alpha\\
		O & \gamma M & -\beta M\\
		A^\alpha & -\beta M & O
	\end{bmatrix}
	\begin{pmatrix}
		\mathbf{y}\\
		\mathbf{u}\\
		\mathbf{p}\\
	\end{pmatrix} =
	\begin{pmatrix}
		\mathbf{y}_\varOmega\\
		\mathbf{0}\\
		\mathbf{0}\\
	\end{pmatrix} \
	\begin{matrix}
	\mathrm{(I)} \\
	\mathrm{(II)} \\
	\mathrm{(III)} \\
	\end{matrix}.
\end{equation*}
The system contains the necessary first order optimality conditions belonging to the discussed 
optimal control problem.

Then the state equation $\mathrm{(III)}$ can be solved for $\y$, yielding
\begin{equation}\label{eqn:State}
	\mathbf{y} = \beta A ^{-\alpha}\mathbf{u}.
\end{equation}
Subsequently to solving $\mathrm{(II)}$ for $\mathbf{p}$, the adjoint equation $\mathrm{(I)}$ 
eventually provides an equation for the control $\mathbf{u}$, that is
\begin{equation} \label{eqn:Lagrange_cont}
\big(\beta A^{-\alpha} + \tfrac{\gamma}{\beta}A^\alpha\big) \mathbf{u}  = \mathbf{y}_\varOmega.
\end{equation}

In what follows, we describe how to derive the discretization of the operator $\A$ and afterwards 
present the efficient tensor based numerical method for solving equation \eqref{eqn:Lagrange_cont} and related ones.

\subsection{Finite difference scheme}\label{ssec:FDM}

First, we derive a factorized representation of the discrete elliptic operator $A$
and the related matrix valued functions of $A$, which are compatible with low-rank data.
Let $I_{(\ell)}$ denote the identity matrix, and $A_{(\ell)}$
the discretized one-dimensional Sturm-Liouville operators on the given grid in the $\ell$th mode, then we have
\begin{equation} \label{eqn_low_rank_lap}
    A \equiv A_1 \oplus A_2 \oplus A_3= A_{(1)} \otimes I_2 \otimes I_3 + I_1 \otimes A_{(2)} \otimes I_3 +
    I_1 \otimes I_2 \otimes A_{(3)},
\end{equation}
 where $\otimes$ denotes the Kronecker product of matrices.  

To obtain the symmetric three-diagonal matrices $A_{(\ell)}$, $\ell=1,2,3$, we simply discretize the one-dimensional
Sturm-Liouville eigenvalue problems
\begin{align}\label{eqn:eig1D}
    -\frac{\partial}{d x_\ell} a_\ell(x_\ell) \frac{\partial}{d x_\ell}   y &= \lambda y, \\
     y(0) &= y(1) =0. \nonumber
\end{align}
Therefore, we use the corresponding weak formulation
\begin{equation} \label{eqn:weak1D} \int_0^1 -a_\ell(x_\ell) p'(x_\ell)y'(x_\ell) \text{d} x_\ell = 
\int_0^1 - \lambda y(x_\ell) p(x_\ell) \text{d}x_\ell
 \end{equation} for $\Omega = ]0, 1[$ and 
 $p(x_\ell) \in C^1(]0,1[) \cap C([0,1]),\ p(0) = 0 = p(1)$, \ $\ell = 1, 2$.\\
 For the sake of simplicity we set $a := a_\ell$ and consider a uniform grid with grid size 
 $h := h_\ell$. For discretization, we use the finite difference approximations
 \begin{gather*}
p_i' = \frac{p_i - p_{i-1}}{h} \qquad \text{and} \qquad y_i' = \frac{y_i - y_{i-1}}{h}
%\partial_x y_i = \frac{\yplus-y_i}{h} \qquad \text{and} \qquad \partial_x y_i = \frac{y_i - y_{i-1}}{h},\\
%\partial_{xx}y_i = \frac{\yplus - 2y_i + \yminus}{h^2},
\end{gather*}
as well as the evaluation $a_{i \pm \frac{1}{2}}$ of the equation coefficients $a_\ell (x_\ell)$ 
at the middle point of two grid points.

 Finally, we end up with the following discretization for above weak formulation \eqref{eqn:weak1D}. 
Considering the left side of \eqref{eqn:weak1D}, the discretization reads
\begin{alignat}{1}
& \qquad \qquad \int_0^1 -a_\ell(x_\ell) p'(x_\ell)y'(x_\ell) \text{d} x_\ell\nonumber \\ &= \qquad -h \sum_{i=2}^\nell \aminus \frac{p_i-p_{i-1}}{h}  \frac{y_i - y_{i-1}}{h} \nonumber\\
&= \qquad -h \sum_{i=2}^\nell \aminus \frac{y_i - \yminus}{h^2}(p_i - \pminus) \quad \nonumber  \\
&= \qquad -h \sum_{i=2}^\nell \aminus \frac{y_i - \yminus}{h^2} p_i + h \sum_{i = 1}^{\nell -1} \aplus \frac{\yplus - y_i}{h^2} p_i \nonumber \\
&  = \ h a_{1+\frac{1}{2}} \frac{y_2 - y_1}{h^2} p_1 - h \sum_{i=2}^{\nell -1} \left( \aminus \frac{y_i - \yminus}{h^2} - \aplus \frac{\yplus - y_i}{h^2} \right) p_i \nonumber  - h a_{\nell-\frac{1}{2}} \frac{y_\nell - y_{\nell -1}}{h^2} p_\nell,
\end{alignat}
whereas for the right side of \eqref{eqn:weak1D}, it holds that
$$
\int_0^1 - \lambda y(x_\ell) p(x_\ell) \text{d}x_\ell \ = \ -\lambda h \sum_{i=2}^{\nell-1} y_i p_i - h \frac{y_1p_1}{2} - h \frac{y_\nell p_\nell}{2}. 
$$

%\cn
%\begin{alignat}{1}
%&\quad -h \sum_{i=2}^\nell \aminus \frac{p_i-p_{i-1}}{h} \cdot \frac{y_i - y_{i-1}}{h} \quad = 
%\quad \underbrace{-\lambda h \sum_{i=2}^{\nell-1} y_i p_i - h \frac{y_1p_1}{2} - h \frac{y_\nell p_\nell}{2}}_{\substack{Z}}
%\nonumber \\
%\Leftrightarrow &\quad -h \sum_{i=2}^\nell \aminus \frac{y_i - \yminus}{h^2}(p_i - \pminus) \quad 
%= \quad Z \nonumber \\
%\Leftrightarrow &\quad -h \sum_{i=2}^\nell \aminus \frac{y_i - \yminus}{h^2} p_i + h \sum_{i = 1}^{\nell -1} \aplus \frac{\yplus - y_i}{h^2} p_i \quad 
%= \quad Z  \nonumber \\
%\Leftrightarrow &\quad h a_{1+\frac{1}{2}} \frac{y_2 - y_1}{h^2} p_1 - h \sum_{i=2}^{\nell -1} \left( \aminus \frac{y_i - \yminus}{h^2} - \aplus \frac{\yplus - y_i}{h^2} \right) p_i \nonumber \\ \quad& \quad - h a_{\nell-\frac{1}{2}} \frac{y_\nell - y_{\nell -1}}{h^2} p_\nell  %\label{eq:69}
% \quad = \quad Z. \label{eqn:disc1D}
%\end{alignat}.

Ultimately, the full discretization scheme for the one-dimensional eigenvalue problem \eqref{eqn:eig1D} is 
 \cn
\begin{equation*}
%-\frac{1}{h_\ell^2}
\underbrace{-\frac{1}{h_\ell^2}
 \begin{bmatrix}
    a_{1,1}   & a_{1,2}      &       &       \\
    a_{2,1}   & a_{2,2}     & a_{2,3}     &  \\
    & \ddots & \ddots &   a_{n_\ell -1,n_\ell}     \\
    &        &      a_{n_\ell,n_\ell -1 }  & a_{n_\ell,n_\ell}
 \end{bmatrix} }_{\substack{= A_\ell}}
 \begin{pmatrix}
 y_1 \\  \\ \vdots \\  \\ y_{n_\ell}
\end{pmatrix}
=
\lambda
\begin{pmatrix}
y_1 \\ \\ \vdots  \\ \\ y_{n_\ell}
\end{pmatrix}.
 % \in \mathbb{R}^{n_\ell \times n_\ell},
 \end{equation*}

 with the entries of the three-diagonal matrix $A=[a_{i,j} ]$ given by (we skip the subindex $\ell$)

 {\begin{center}
 $a_{i,j} = \begin{cases}
 \aplus & j-i = 1\\
 \aminus & j-i = -1\\
 -\aplus - \aminus & i = j,\quad i \not\in \{1, \nell\} \\
 -2\aplus & i = j = 1\\
 -2\aminus & i = j = \nell.
 \end{cases}$\end{center}}

\subsection{Stiffness matrix in the low-rank Kronecker form}
\label{ssec:StiffnessMatrixinLR}

Again, consider the Laplace-type operator $A$ in discretized format \eqref{eqn_low_rank_lap}, 
that is $$ A = A_{(1)} \otimes I_2 \otimes I_3 + I_1 \otimes A_{(2)} \otimes I_3 +
    I_1 \otimes I_2 \otimes A_{(3)}.$$
    
Let $G_\ell \in \mathbb{R}^{n_\ell \times n_\ell}$ be the orthogonal matrix composed of the eigenvectors
of the problem
\begin{equation}\label{eqn:eigs_1D}
 A_{(\ell)} g_i = \lambda_i^{(\ell)} g_i, \quad g_i \in \mathbb{R}^{n_\ell}, \; i \in \{1,\dots,n_\ell  \}.
\end{equation}
Then the one-dimensional operator (matrix) $A_{(\ell)}$ admits an eigenvalue decomposition in its eigenbasis,
\begin{equation*}
    A_{(\ell)} = G_\ell^T {\varLambda}_{(\ell)} G_\ell, \quad
{\varLambda}_{(\ell)}=\mbox{diag}\{\lambda^{(\ell)}_1,\dots,\lambda^{(\ell)}_{\nell}\}, \quad \ell = 1,2,3
\end{equation*}
with $G_\ell$ consisting of the (column) eigenvectors $g_i \in \mathbb{R}^{n_\ell}$. 

Following \cite{HKKS:18} and using the properties of the Kronecker product, we can write the first 
summand in \eqref{eqn_low_rank_lap} as
\begin{equation*}
\begin{split}
A_{(1)} \otimes I_2 \otimes I_3 &= (G_1^T {\varLambda}_{(1)} G_1) \otimes (G_2^T I_2 G_2) 
\otimes (G_3^T I_3 G_3) \\
&= (G_1^T\otimes G_2^T \otimes G_3^T) ({\varLambda}_{(1)}\otimes I_2\otimes I_3) (G_1\otimes G_2 \otimes G_3),
\end{split}
\end{equation*}
and similarly for matrices $A_{(2)}$ and $A_{(3)}$. Eventually, this suggests the following rank-$3$ Kronecker 
representation of the full stiffness matrix in \eqref{eqn_low_rank_lap} as
\begin{equation}\label{eqn:DiagGen}
A = 
(G_1^T\otimes G_2^T\otimes G_3^T) 
	\big[ \varLambda_1\otimes I_2\otimes I_3 + I_1\otimes \varLambda_2\otimes I_3 + 
	I_1\otimes I_2\otimes \varLambda_3 \big]
	(G_1\otimes G_2 \otimes G_3).
% \begin{split}
% A =  &(F_1^*\otimes F_2^*\otimes F_3^*) (\varLambda_1\otimes I_2\otimes I_3) (F_1\otimes F_2 \otimes F_3)\\
% 	&+(F_1^*\otimes F_2^*\otimes F_3^*) (I_1\otimes \varLambda_2\otimes I_3) (F_1\otimes F_2 \otimes F_3)\\
% 	&+(F_1^*\otimes F_2^*\otimes F_3^*) (I_1\otimes I_2\otimes \varLambda_3) (F_1\otimes F_2 \otimes F_3)\\
% 	=\; &(F_1^*\otimes F_2^*\otimes F_3^*) 
% 	\big( (\varLambda_1\otimes I_2\otimes I_3) + (I_1\otimes \varLambda_2\otimes I_3) + 
% 	(I_1\otimes I_2\otimes \varLambda_3) \big)
% 	(F_1\otimes F_2 \otimes F_3).
% \end{split}
\end{equation}

Considering $d=2$, expression (\ref{eqn:DiagGen}) simplifies to
\begin{equation*}\label{eqn:DiagGen2D}
\begin{split}
A &=  (G_1^T\otimes G_2^T) (\varLambda_1\otimes I_2) (G_1\otimes G_2)
	+(G_1^T\otimes G_2^T) (I_1\otimes \varLambda_2) (G_1\otimes G_2)\\
	&= (G_1^T\otimes G_2^T) 
	\underbrace{\big( \varLambda_1\otimes I_2 + I_1\otimes \varLambda_2 \big)}_{\eqqcolon \varLambda}
	(G_1\otimes G_2 ),
\end{split}
\end{equation*}
which provides the eigenvalue decomposition for any matrix valued function ${\cal F}(A)$,
\begin{equation}\label{eqn:DiagFGen2D}
\mathcal{F}(A)  = (G_1^T\otimes G_2^T) 
	\mathcal{F}(\varLambda) (G_1\otimes G_2 ).
\end{equation}

Then the eigenvalue decompositions  (\ref{eqn:DiagGen}) and (\ref{eqn:DiagFGen2D}) provide the 
efficiently computation of some matrix valued functions ${\cal F}(A)$ at the low cost of the 
order of $O(d\,n^2)$. \\

Again following \cite{HKKS:18}, let us assume that $\mathcal{F}(A)$ can be expressed approximately 
by a linear combination of Kronecker rank-$1$ operators, so that due to \eqref{eqn:DiagFGen2D}, 
for the approximation of $\mathcal{F}(A)$ it is sufficient to approximate the 
diagonal matrix $\mathcal{F}(\varLambda)$.
Thus, we assume the decomposition
\begin{equation} \label{eqn:CanFormat}
	\mathcal{F}(\varLambda) \approx \sum_{k=1}^R \diag \big(\mathbf{u}_1^{(k)} \otimes \mathbf{u}_2^{(k)}\big),
\end{equation}
with vectors $\mathbf{u}_{(\ell)}^{(k)}\in\R^{n_{(\ell)}}$ and $R \ll \min(n_1,n_2)$, and let
$\mathbf{x}\in \R^N$ be a vector given in a low-rank format, i.\,e.
\begin{equation*}
	\mathbf{x} \approx \sum_{j=1}^S \mathbf{x}_1^{(j)} \otimes \mathbf{x}_2^{(j)},
\end{equation*}
with vectors $\mathbf{x}_{(\ell)}^{(j)}\in\R^{n_{(\ell)}}$ and ${S \ll \min(n_1,n_2)}$. \\
Then, we can compute a matrix-vector product
\begin{equation}\label{eqn:Gen2DtimesX}
\begin{split}
	\mathcal{F}(A) \mathbf{x} &\approx (G_1^T\otimes G_2^T) 
	\bigg( \sum_{k=1}^R \diag \big(\mathbf{u}_1^{(k)} \otimes \mathbf{u}_2^{(k)}\big) \bigg)
	(G_1\otimes G_2 )
	\bigg( \sum_{j=1}^S \mathbf{x}_1^{(j)} \otimes \mathbf{x}_2^{(j)} \bigg)\\
	&= \sum_{k=1}^R \sum_{j=1}^S G_1^T \big( \mathbf{u}_1^{(k)} \odot G_1 \mathbf{x}_1^{(j)} \big) \otimes
	G_2^T\big( \mathbf{u}_2^{(k)} \odot G_2 \mathbf{x}_2^{(j)} \big),
\end{split}
\end{equation}
where $\odot$ denotes the componentwise (Hadamard) product of vectors.

Expression \eqref{eqn:Gen2DtimesX} can be calculated in factored 
form in $\mathcal{O}(RSn^2 )$ flops, where ${n = \max(n_1,n_2)}$.

Considering $d=3$, with completely analogous reasoning, equation \eqref{eqn:Gen2DtimesX} becomes
\begin{equation}\label{eqn:Gen3DtimesX}
	\mathcal{F}(A) \mathbf{x} \approx 
	\sum_{k=1}^R \sum_{j=1}^S G_1^T \big( \mathbf{u}_1^{(k)} \odot G_1 \mathbf{x}_1^{(j)} \big) 
	\otimes
	G_2^T\big( \mathbf{u}_2^{(k)} \odot G_2 \mathbf{x}_2^{(j)} \big)
	\otimes G_3^T\big( \mathbf{u}_3^{(k)} \odot G_3 \mathbf{x}_3^{(j)} \big),
\end{equation}
which can be implemented with the same asymptotic cost as in 2D case, i.e., 
in $\mathcal{O}(RSn^2 )$ operations.%\\ With the help of factorized expressions \eqref{eqn:Gen2DtimesX} and \eqref{eqn:Gen3DtimesX} we now have a computational cost that is independent of the dimension $d$, which means that we successfully managed to circumvent the curse of dimensionality in solving the considered optimal control problem.

\begin{remark}\label{rem:cost1Ddiagonal}
In the $d$-dimensional case we arrive at the 
linear scaling in $d$ of numerical cost, $\mathcal{O}(dRSn^2 )$, 
such that the effect of quadratic scaling in the univariate grid size $n$ becomes negligible 
in comparison with the gain from getting rid of the curse of dimensionality.
Moreover, the cubic cost of solving the eigenvalue  problem for $n\times n$ three-diagonal
matrix can be considered as negligible in the practically interesting range of grid-size $n$
until several thousand, since it only takes few seconds even for rather large matrices, 
see also section \ref{ssec:numerics_tensor_2D} for more details. 
\end{remark}

It is worth to note  that in the case of   non-structured (full format) long   vectors 
$\mathbf{x}\in \mathbb{R}^{n^3}$ we have $S=n^2$ which increases the cost of matrix-vector 
multiplication up to $\mathcal{O}(Rn^4 )$.

In our applications we are interested in the low Kronecker rank ($K$-rank) representations 
(approximations) of the matrix valued functions
\[
 \mathcal{F}_1(A)= A^\alpha, \quad    \mathcal{F}_2(A)= A^\alpha + A^{-\alpha},  
 \quad \mbox{and} \quad  \mathcal{F}_3(A)= (A^\alpha + A^{-\alpha})^{-1},
\]
where $A=A_h$ is the FEM/FDM discretization of the target elliptic operator ${\cal A}$.
  Another important task is the construction of the spectrally close low $K$-rank preconditioners 
for the matrix valued function $\mathcal{F}_2(A)$.
These issues will be discussed in the next sections.

\subsection{Low-rank approximation and the Kronecker rank bounds}\label{ssec:KronRankBound}

First we notice that the orthogonal transformation matrix 
$G_1\otimes G_2 \otimes G_3$ in the factorization (\ref{eqn:DiagGen}) has Kronecker rank $1$.
This means that the low $K$-rank decomposition of the matrix valued function 
${\cal F} (A)$ is equivalent to the low-rank tensor approximation of the $3$-way folding of the 
diagonal matrix ${\cal F} (\Lambda)$.
 If we suppose that the initial system matrix $A$ is spectrally equivalent to the anisotropic 
Laplacian (see details in \S\ref{ssec:Precond} below) then the existence of the low-rank approximation 
for the target matrix valued function ${\cal F} (A)$, in particular ${\cal F}_2 (A)$ could be  justified 
by slightly modified argument of those applied in \cite{HKKS:18} for the case of discrete Laplacain. 

In what follows, we discuss the Tucker/canonical decomposition of the coefficients tensor 
  ${\bf D}_\Lambda = [d_\Lambda(i_1,i_2,i_3)]\in \mathbb{R}^{n_1 \times n_2 \times n_3}, \; i_\ell\in \{1,\ldots,n_\ell\}$, 
  obtained by reshaping the diagonal of the matrix ${\cal F}_2 (\Lambda)\in \mathbb{R}^{n^3\times  n^3}$ 
  to the third order tensor ${\bf D}_\Lambda$. 
 For example, in the case of discrete Laplacian the elements of the corresponding rank-$3$ coefficients 
 tensor take a simple form 
 \[
 d_\Delta(i_1,i_2,i_3)=\lambda_{i_1} + \lambda_{i_2} + \lambda_{i_3},
\]
 with the eigenvalues $\lambda_{i_\ell}$ of univariate Laplacian. 
 
 For the ease of exposition,
 let us suppose that the 1D elliptic operators $A_\ell$ are all the same for three dimensions so that 
 we omit the index $\ell$ in notations for $\lambda_i=\lambda_i^{(\ell)}$.
 Then the elements of the rank-$3$ coefficients tensor 
 ${\bf D}_A = [d_A(i,j,k)]\in \mathbb{R}^{n \times n \times n}, i,j,k=1,\dots,n$, corresponding to the 
 factorization of the target matrix $A$ in the eigenbasis,  take a form
 \[
 d_A(i,j,k)=\lambda_{i} + \lambda_{j} + \lambda_{k},
\]
 where the eigenvalues $\lambda_{i}$, $i=1,\ldots,n$, are given by (\ref{eqn:eigs_1D}).
 Hence, we arrive at the explicit representation for the entries of 
 \begin{equation}\label{eqn:DiagLambyA}
{\bf D}_\Lambda= {\bf D}_A^{-\alpha} + {\bf D}_A^{\alpha}, 
 \end{equation}
 as follows
 \begin{equation}\label{eq:TensCoefG2_3D}
 d(i,j,k) = \frac{1}{(\lambda_i +\lambda_j +\lambda_k)^\alpha} +  (\lambda_i +\lambda_j+\lambda_k)^\alpha.
\end{equation}
Here we point out that the rank-$R$ canonical approximation of the third order tensor  ${\bf D}_\Lambda$
is equivalent to the $R$-term Kronecker representation of the diagonal matrix 
${\cal F}_2 (\Lambda)\in \mathbb{R}^{n^3\times  n^3}$, due to the relation
\begin{equation*} \label{eqn:CanFormat_3D}
\mathcal{F}_2(\varLambda) = 
\sum_{k=1}^R \diag \big(\mathbf{u}_1^{(k)} \otimes \mathbf{u}_2^{(k)}\otimes \mathbf{u}_3^{(k)}  \big),
\end{equation*}
where $\mathbf{u}_{\ell}^{(k)}\in\R^{n_{\ell}}$ are
the skeleton vectors  of the corresponding canonical decomposition of ${\bf D}_\Lambda$.
This is  straightforwardly  translated to the respective $R$-term Kronecker representation of 
the matrix valued function  ${\cal F}_2 (A)$ of interest
\begin{equation}\label{eqn:DiagKron3D}
 \begin{split}
{\cal F}_2(A)  
& = (G_1^T\otimes G_2^T\otimes G_3^T) {\cal F}_2 (\Lambda)(G_1\otimes G_2 \otimes G_3)\\
& = (G_1^T\otimes G_2^T\otimes G_3^T)
\sum_{k=1}^R \diag \big(\mathbf{u}_1^{(k)} \otimes \mathbf{u}_2^{(k)}\otimes \mathbf{u}_3^{(k)}  \big)
(G_1\otimes G_2 \otimes G_3)\\
& = \sum_{k=1}^R (G_1^T \diag \mathbf{u}_1^{(k)}G_1) \otimes
 (G_2^T \diag \mathbf{u}_2^{(k)}G_2) \otimes (G_3^T \diag \mathbf{u}_3^{(k)}G_3).
\end{split}		
\end{equation}
 The representation (\ref{eqn:DiagKron3D}) benefits from the linear scaling in $d$ for both storage size and matrix-vector 
 multiplication cost.
 
 In turn, the low-rank canonical decomposition (approximation) of the tensor ${\bf D}_\Lambda$ given by 
 (\ref{eq:TensCoefG2_3D}) is performed with the robust multigrid full-to-Tucker-to-canonical algorithm 
 \cite{KhKh3:08,Khor2Book2:18} sketched in Appendix 1.

 The existence of the accurate low-rank canonical approximation of the tensor ${\bf D}_\Lambda$
 can be analyzed separately for both summands in (\ref{eqn:DiagLambyA}). 
 For the term with negative fractional power ${\bf D}_A^{-\alpha}$
 the Laplace integral transform representation (\ref{eqn:Laplace_transf}) and the corresponding sinc 
 quadrature approximation for the integral on the real axis (\ref{eqn:sinc_Laplace_tr}) apply to the target tensor 
 pointwise to obtain
\begin{equation} \label{eqn:sinc_Laplace_tr_Dalpha}
{\bf D}_A^{-\alpha}[i_1,i_2,i_3]= \frac{1}{\Gamma(\alpha)}
  \int_0^\infty t^{\alpha -1} e^{-t {\bf D}_A }\, dt \approx 
 \sum\limits_{k=-M}^M \hat{c}_k t_k^{\alpha -1} e^{-t_k {\bf D}_A}= 
 \sum\limits_{k=-M}^M c_k \bigotimes_{\ell=1}^d e^{-t_k \lambda_{i_\ell} }.
\end{equation}
Assume that the argument in the exponential in (\ref{eqn:sinc_Laplace_tr_Dalpha}) varies in the fixed 
interval on the 
positive semi-axis, i.e. $0< a_0 \leq \lambda_{i} + \lambda_{j} + \lambda_{k}\leq a_1 < \infty$, which is the case 
in our construction, then there is the quasi-optimal choice of the quadrature points $t_k$ and weights $c_k$ that
ensures the exponentially fast convergence of the quadrature rule in the number of terms, $2M+1$
 \cite{GHK:05,HaKhtens:04I,KhorBook:18}. Hence the number of separable terms, $R=2M+1$, 
 that is the respective canonical rank, can be estimated by 
 $$ R \leq C |\log \varepsilon |, 
 $$
 where $\varepsilon >0$  is the accuracy threshold.
 The rank bound for the positive power of ${\bf D}_A^{\alpha}$, $0<\alpha <1$, follows from the pointwise 
 (Hadamard product) factorization
 \[
  {\bf D}_A^{\alpha} = {\bf D}_A \cdot {\bf D}_A^{\alpha-1},\quad \alpha-1 <0,
 \]
where ${\bf D}_A$ is the rank-$3$ tensor, implying $\mbox{rank}({\bf D}_A^{\alpha})\leq 3\, R$.

\subsection{Preconditioner in the low-rank Kronecker form}\label{ssec:Precond}

We propose and analyze the two different candidates for the efficient preconditioning of 
the matrix valued function 
${\cal F}_2 (A)$: \\
(A) The preconditioning matrix is constructed by using the tensor decomposition of ${\cal F}_3 (\Delta)$
by using the Fourier based diagonalization of the 
discrete anisotropic Laplacian $\Delta$ in the similar way as in \cite{HKKS:18};\\
(B) Making use of the direct low $K$-rank approximation to the reciprocal matrix valued function 
\[
 \mathcal{F}_3(A)= (A^\alpha + A^{-\alpha})^{-1},
\]
where the target discrete elliptic operator (matrix) $A$ is factorized in the eigenbasis 
of the univariate elliptic operators $A_\ell$, $\ell=1,\ldots,3$, with variable coefficients.

First, we recall some basic  rank-structured decompositions for functions of the discrete Laplacian  
presented in \cite{HKKS:18}. 
The one-dimensional Laplace operator $\Delta_{(\ell)}$ has an eigenvalue decomposition in the Fourier basis, i. e.
\begin{equation*}
	\Delta_{(\ell)} = F_\ell^* {\varLambda}_{(\ell)} F_\ell.
\end{equation*}
In the case of homogeneous Dirichlet boundary conditions the matrix $F_\ell \in \mathbb{R}^{n\times n}$ defines the 
$\sin$-Fourier transform  and the diagonal matrix
${\varLambda}_{(\ell)}=\mbox{diag}\{\lambda^{(\ell)}_1,\dots,\lambda^{(\ell)}_n\}$ is composed of  
 the eigenvalues of the univariate discrete Laplacian, $\lambda_k^{(\ell)}$,   which are given by
\begin{equation*}\label{eqn:eig_lap}
	\lambda_k = -\frac{4}{h_\ell^2}\sin^2 \bigg( \frac{\pi k}{2(n_\ell+1)} \bigg)
	 = -\frac{4}{h_\ell^2}\sin^2 \bigg( \frac{\pi k h_\ell}{2} \bigg).
\end{equation*}

Analogously to section \ref{ssec:StiffnessMatrixinLR}, we can use the properties of 
the Kronecker product and rewrite the first summand in \eqref{eqn_low_rank_lap} as
\begin{equation*}
\Delta_{(1)} \otimes I_2 \otimes I_3 = 
(F_1^*\otimes F_2^* \otimes F_3^*) ({\varLambda}_{(1)}\otimes I_2\otimes I_3) (F_1\otimes F_2 \otimes F_3).
\end{equation*}
Rewriting the other summands in the same style, we finally can write equation
\eqref{eqn_low_rank_lap} as
\begin{equation}\label{eqn:DiagLaplace}
\begin{split}
\Delta = \ &(F_1^*\otimes F_2^*\otimes F_3^*) 
	\big( \varLambda_1\otimes I_2\otimes I_3 + I_1\otimes \varLambda_2\otimes I_3 + 
	I_1\otimes I_2\otimes \varLambda_3 \big)
	(F_1\otimes F_2 \otimes F_3).
\end{split}
\end{equation}
% The above expression gives us the eigenvalue decomposition, which can be used to efficiently compute functions of $A$. 

For $d=2$, the expression simplifies to
\begin{equation}\label{eqn:DiagLaplace2D}
\begin{split}
\Delta &= (F_1^*\otimes F_2^*)
	\underbrace{\big( \varLambda_1\otimes I_2 + I_1\otimes \varLambda_2 \big)}_{\eqqcolon \varLambda}
	(F_1\otimes F_2 ).
\end{split}
\end{equation}

With the help of the eigenvalue decomposition of $\Delta$ \eqref{eqn:DiagLaplace2D}, 
we can compute a function $\mathcal{F}=\mathcal{F}_3$ applied to $\Delta$ as
\begin{equation*}\label{eqn:DiagFLaplace2D}
\mathcal{F}(\Delta)  = (F_1^*\otimes F_2^*) \mathcal{F}(\varLambda) (F_1\otimes F_2 ).
\end{equation*}

Supposing low-rank decompositions for both $\mathcal{F}(A)$ and a vector $\mathbf{x} \in \R^N$ 
in the same style as in to section \ref{ssec:StiffnessMatrixinLR}, we can compute a matrix-vector product as
\begin{equation}\label{eqn:Lap2DtimesX}
\begin{split}
	\mathcal{F}(\Delta) \mathbf{x} &\approx (F_1^*\otimes F_2^*) 
	\bigg( \sum_{k=1}^R \diag \big(\mathbf{u}_1^{(k)} \otimes \mathbf{u}_2^{(k)}\big) \bigg)
	(F_1\otimes F_2 )
	\bigg( \sum_{j=1}^S \mathbf{x}_1^{(j)} \otimes \mathbf{x}_2^{(j)} \bigg)\\
	&= \sum_{k=1}^R \sum_{j=1}^S F_1^* \big( \mathbf{u}_1^{(k)} \odot F_1 \mathbf{x}_1^{(j)} \big) 
	\otimes F_2^*\big( \mathbf{u}_2^{(k)} \odot F_2 \mathbf{x}_2^{(j)} \big),
\end{split}
\end{equation}
where $\odot$ denotes the componentwise (Hadamard) product.

Using the sin-FFT, expression \eqref{eqn:Lap2DtimesX} can be  evaluated   in the factored 
form in $\mathcal{O}(RSn\log n)$ flops, where
$n = \max(n_1,n_2)$,  and $S$ is the Kronecker rank of vector $\mathbf{x}$.  

For $d=3$, with completely analogous reasoning, equation \eqref{eqn:Lap2DtimesX} becomes
\begin{equation}\label{eqn:Lap3DtimesX}
	\mathcal{F}(A) \mathbf{x} = 
	\sum_{k=1}^R \sum_{j=1}^S F_1^* \big( \mathbf{u}_1^{(k)} \odot F_1 \mathbf{x}_1^{(j)} \big) \otimes
	F_2^*\big( \mathbf{u}_2^{(k)} \odot F_2 \mathbf{x}_2^{(j)} \big)
	\otimes F_3^*\big( \mathbf{u}_3^{(k)} \odot F_3 \mathbf{x}_3^{(j)} \big),
\end{equation}
and similar in the case of $d>3$.  It can be evaluated in $\mathcal{O}(d RSn\log n)$ flops for 
 $K$-rank structured vectors represented on $n^{\otimes d}$ Cartesian grid.

It is worth to note that the previous constructions remain valid also in the case of anisotropic Laplacian 
\begin{equation}\label{eqn:anisotropicL}
 {\cal B} = -\sum_{\ell=1}^d \frac{\partial}{d x_\ell} b^\ell \frac{\partial}{d x_\ell}, 
 \quad \mbox{with some constants} \quad b^\ell >0.
\end{equation}

 In case (A), for the sake of preconditioning, we need the low $K$-rank 
approximation\footnote{Note that the numerical algorithm for the Tucker and canonical tensor decomposition of 
the $d$th-order tensors has been introduced in \cite{KhKh3:08} and it was adapted 
to the case of fractional Laplacian in \cite{HKKS:18}.} 
to the matrix valued function of the form
\[
 \mathcal{F}_3({ B})= ({ B}^\alpha + { B}^{-\alpha})^{-1},
\]
where the anisotropic Laplacian stiffness matrix ${B}$ corresponding to the operator 
in (\ref{eqn:anisotropicL})
can be factorized in the Fourier basis as in (\ref{eqn:DiagLaplace}).
To that end, we define the average coefficients
$$
{b}_0^{\ell}=\frac{1}{2}(\max   a^+_\ell(x)   + \min   a^-_\ell(x)   ) ,\quad \ell=1,2,3,
$$ 
where   $a^+_\ell(x) > 0$ and $a^-_\ell(x) > 0$   are chosen as {\it majorants and minorants} 
of the equation coefficient $a_\ell(x_\ell)$, respectively. 
Then we introduce the fractional anisotropic Laplacian type operator 
${\cal B}$ generated by the constant coefficients ${b}_0^{\ell}$, $\ell=1,2,3,$ as follows
\begin{equation} \label{eqn:PrecClassA}
  {\cal B}^\alpha  :=  \left(-\sum_{\ell=1}^d \frac{\partial}{d x_\ell} b_0^\ell \frac{\partial}{d x_\ell}
 \right)^\alpha, 
 \quad  0\leq \alpha\leq 1,
\end{equation}
and define the desired preconditioning matrix by using the discrete versions of the nonlocal 
operators ${\cal B}^{\alpha}$ and ${\cal B}^{-\alpha}$ 
in a form of a weighted sum of the matrix ${B}^{\alpha}$ and its inverse
\begin{equation} \label{eqn:Precond_var}
{B}_0= \beta {B}^{-\alpha} + \tfrac{\gamma}{\beta}{B}^{\alpha}.
\end{equation}
It can be proven that the condition number of the preconditioned matrix ${B}_0^{-1} {\cal F}_2(A)$, where
${\cal F}_2({A})= \beta {A}^{-\alpha} + \tfrac{\gamma}{\beta} {A}^{\alpha}$,
is uniformly bounded in $n$. 
%\cb
% This proves the following Lemma on the spectral equivalence of the preconditioner ${B}_0$ by summing up the 
% above estimates with proper weights.
\begin{theorem}\label{lem:cond_bound_var}
Let the matrix $B_0$ be given by (\ref{eqn:Precond_var}).
Under the above assumptions the condition number of the preconditioned matrix for 
the target Lagrange equation with the system matrix ${\cal F}_2({A})$
is uniformly bounded in $n$, specifically 
\[
 \mbox{cond}\{{B}_0^{-1} {\cal F}_2({A}) \} \leq C 
 \frac{\max\{(1+q )^\alpha,(1-q )^{-\alpha}\}}{\min\{(1+q )^{-\alpha},(1-q )^{\alpha}\}}
 %\frac{\beta(1-q )^{-\alpha}+\tfrac{\gamma}{\beta}(1+q )^\alpha}{\beta(1+q )^{-\alpha} +\tfrac{\gamma}{\beta}(1-q )^\alpha},
\] 
with some constant $0< q < 1$ independent on the grid size $n$.
\end{theorem}
\begin{proof}
The matrix $ B$ generated by the coefficients
${b}_0^{\ell}$ corresponding to (\ref{eqn:PrecClassA}) allows the condition number estimate % 
\[
 cond \{{B}^{-1} {A} \} \leq C \frac{1+q}{1-q},  \quad\mbox{with}\quad
 q:=\max_{\ell}\max_x \frac{a_\ell^+(x) - {b}_0^\ell}{b_0^\ell}<1,
\]
which is the consequence of the spectral equivalence estimate
\begin{equation} \label{eqn:specequivBA}
 C_0(1-q )   {B}   \leq  {A} \leq C_1 (1+q )   {B} ,
\end{equation}
where the latter follows from the simple bounds
\[
   a^+_\ell(x) \leq (1+q)  b_0^\ell, \quad (1-q)  b_0^\ell \leq a^-_\ell(x), \quad \ell = 1,2,3.
\]
As a result of (\ref{eqn:specequivBA}), we readily derive the spectral bounds for the fractional elliptic 
operators of interest,
\[
 C_0 (1-q )^\alpha   {B}^\alpha   \leq  {A}^\alpha \leq C_1(1+q )^\alpha   {B}^\alpha  , 
\]
and
\[
 C_0(1+q )^{-\alpha}   {B}^{-\alpha}  \leq  {A}^{-\alpha} \leq C_1(1-q )^{-\alpha}   {B}^{-\alpha} . 
\]
This proves the Theorem on the spectral equivalence of the preconditioner ${B}_0$ by summing up the 
above estimates with the proper weights. Indeed, we obtain for 
${\cal F}_2({A})= \beta {A}^{-\alpha} + \tfrac{\gamma}{\beta} {A}^{\alpha}$
\[
 {\cal F}_2({A}) \leq C_1 (\beta (1-q)^{-\alpha}{B}^{-\alpha}+
 \tfrac{\gamma}{\beta} (1+q)^{\alpha}{B}^{\alpha}) \leq C_1 \max\{(1+q )^\alpha,(1-q )^{-\alpha}\} B_0,
\]
and likewise
\[
  C_0 \min\{(1+q )^{-\alpha},(1-q )^{\alpha}\}  B_0 \leq
 C_0(\beta (1+q)^{-\alpha}{B}^{-\alpha}+ \tfrac{\gamma}{\beta} (1-q)^{\alpha}{B}^{\alpha})\leq {\cal F}_2({A}), 
 \]
which completes the proof.  
\end{proof}\\

\vspace*{5mm}
The practical application of this Theorem in our numerical computations 
presupposes the low $K$-rank approximation of the matrix valued 
function of anisotropic Laplacian, ${B}_0^{-1}$, which is performed by using the multigrid Tucker 
approximation with the consequent Tucker-2-Canonical transform, applied to the respective diagonal 
core matrix ${\cal F}_2({\varLambda})$, see also \S\ref{ssec:KronRankBound} and Appendix 1.

To conclude the discussion of case (A) we notice that the presented preconditioning technique 
also applies to the case of degenerate elliptic operator with the non-negative 
equation coefficients $a_\ell(x_\ell) \geq 0$ for some spacial directions
because the function $ \lambda^\alpha + \lambda^{-\alpha} \geq costant >0$ for all $\lambda \geq 0$.

In case (B), we perform the direct low $K$-rank tensor approximation of the matrix valued function 
${\cal F}_3({A})$ along the same line as it is done for the target matrix function ${\cal F}_2({A})$
by using factorization in the eigen-basis of the univariate elliptic operators with variable 
coefficients, see also (\ref{eqn:DiagKron3D}).

\subsection{The PCG scheme with rank truncation}\label{ssec:LowRankPCG}

For operators $\mathtt{func}$ and $\mathtt{precond}$ given in a low-rank format, 
such as \eqref{eqn:Gen2DtimesX} (for $d=2$) or \eqref{eqn:Gen3DtimesX} (for $d = 3$) 
and \eqref{eqn:Lap2DtimesX} (for $d=2$) or \eqref{eqn:Lap3DtimesX} (for $d=3$), respectively. 
Krylov subspace methods
can be applied very efficiently, since they only require matrix-vector products.

In our applications, we use the formulation of the PCG method in Algorithm~1, \cite{HKKS:18},       %\ref{alg:pcg} 
which is independent of $d$, as long as an appropriate rank truncation procedure $\mathtt{trunc}$ is chosen. 
As adaptive rank truncation procedure we use the reduced singular value decomposition 
in the case $d = 2$ and the reduced higher order singular value decomposition 
based on the Tucker-to-canonical approximation is used for $d = 3$, 
see Appendix 1 and \cite{KhKh3:08} for details.  

For the sake of completeness we present this algorithm in Appendix 2.   %\ref{sec:low-rank_PCG}.

\section{Numerical tests} \label{sec:NumTests}

%\textcolor{red}{use $A$ or $\A$??}\\

%%%%%%%%%%%%%%%%%%%%%%%%%%%%%%%%%%%%

In this section, we present numerical results for both the 2D and 3D cases. 
In all tests, we choose $\alpha = 1, 1/2, 1/10$, rank truncation parameter $\varepsilon = 10^{-6}$, 
the preconditioner rank $R = 6$, and $\beta = 1$.  
Throughout this section, let $A=A_h$ be the FDM discretization of the target elliptic operator ${\cal A}$.
We investigate the numerical results and properties of the algorithm for solving 
equation (\ref{eqn:Lagrange_cont}), 
\begin{equation*}
(\beta A^{-\alpha} + \frac{\gamma}{\beta} A^\alpha) \uu = \yom,
\end{equation*} 
with respect to the optimal control $\uu$ and subsequently equation (\ref{eqn:State}), that is
 \begin{equation*}
\y = \beta A^{-\alpha} \uu,
\end{equation*} with respect to the state variable $\y$.  

%We stop the algorithm as soon 
The PCG iteration is stopped when the relative residual is small enough, that is whenever
$$ 
\left|\int_\Omega (\beta A^{-\alpha} + \frac{\gamma}{\beta} A^\alpha )\tilde{u} - y_\Omega \ \mathrm{d}x \right| \ \le \ 10^{-5}  
$$ 
holds for $\tilde{u}$ solution of (\ref{eqn:Lagrange_cont2}) discretized on the computational mesh. 
For the sake of simplicity, we define 
$$A_u := (\beta A^{-\alpha} + \frac{\gamma}{\beta} A^\alpha). 
$$

Throughout our numerical tests, we consider the following right hand sides $y_\Omega$: 
box-type and $H$-type shapes   as shown in Figures \ref{design1} and \ref{design2} for 2D and 3D cases,
respectively.
 \begin{figure}[H]
 \centering
 	\subfigure[box-type]{\includegraphics[width=0.32\textwidth]{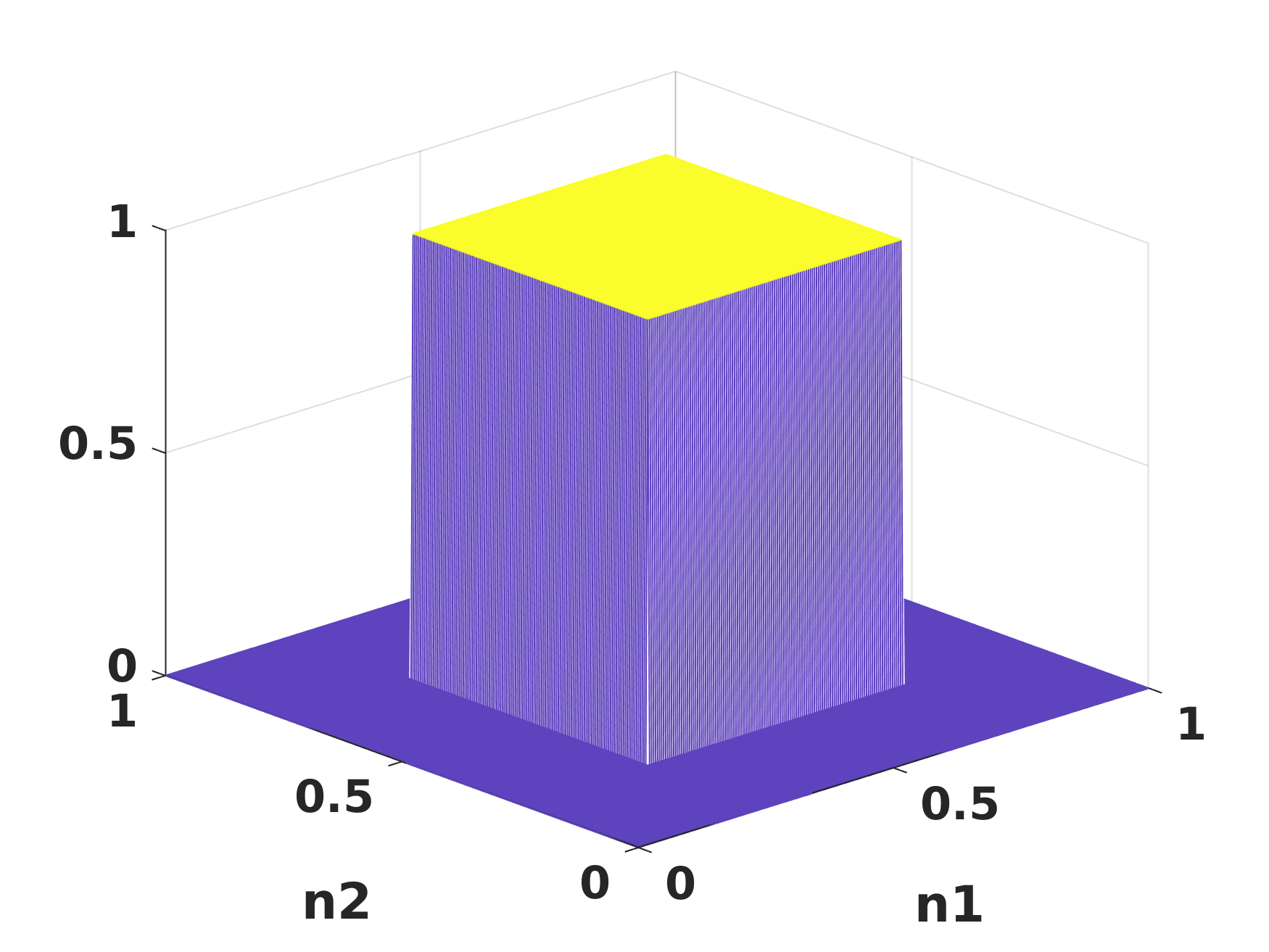}}
 	%\subfigure[U-shaped]{\includegraphics[width=0.32\textwidth]{Figs_Control/Figs_2D/ushaped/designU.png}}
 	\subfigure[$H$-type]{\includegraphics[width=0.32\textwidth]{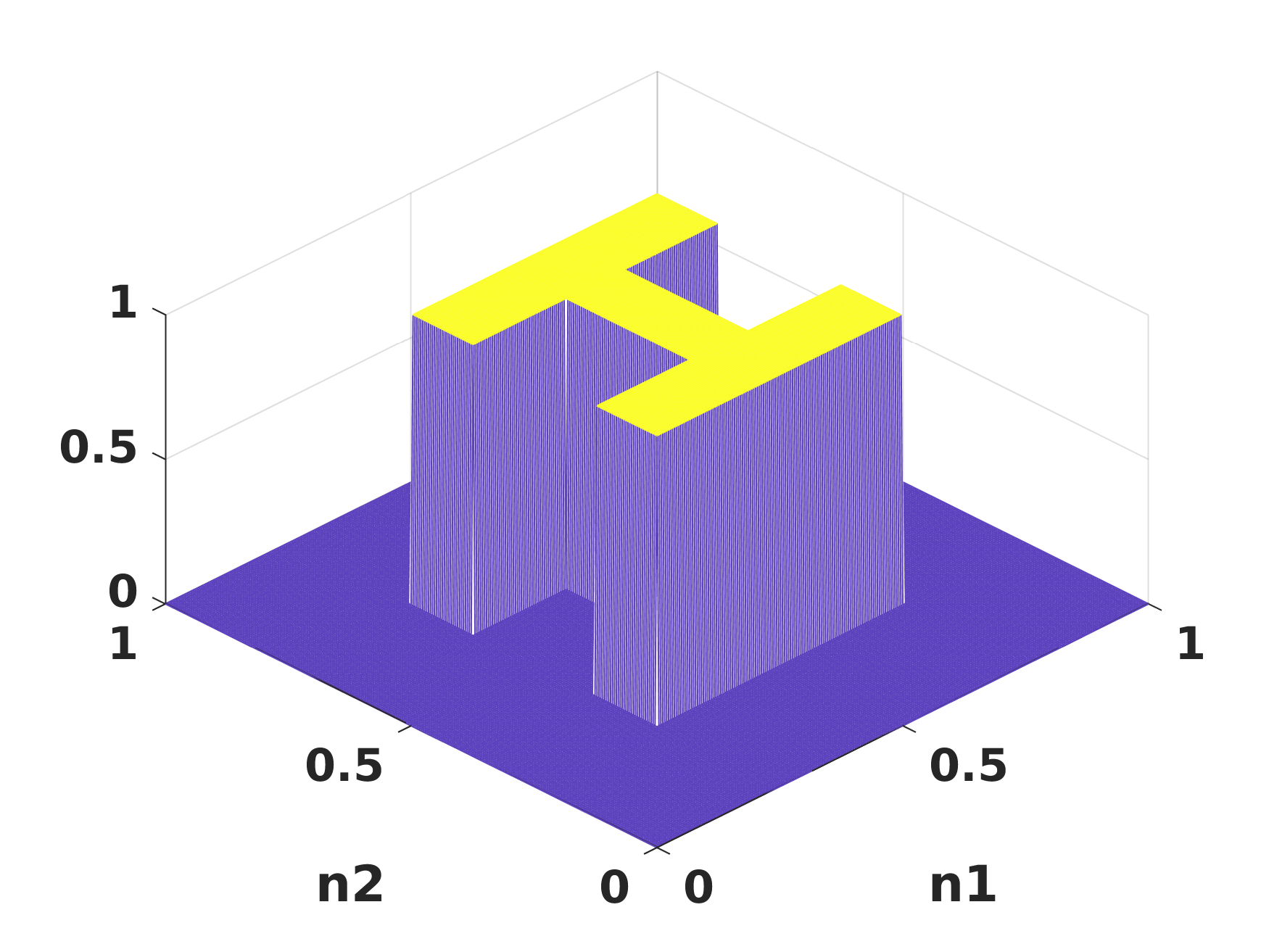}}
 	\caption[Test 2: Lösungen $u$ mit $\alpha = 1$]{Box-type and $H$-form-type 
 	right hand sides $y_\Omega$ for the 2D control problem.}
 	\label{design1}
 \end{figure}

  \begin{figure}[H]
 \centering
 	\subfigure[box-type]{\includegraphics[width=0.32\textwidth]{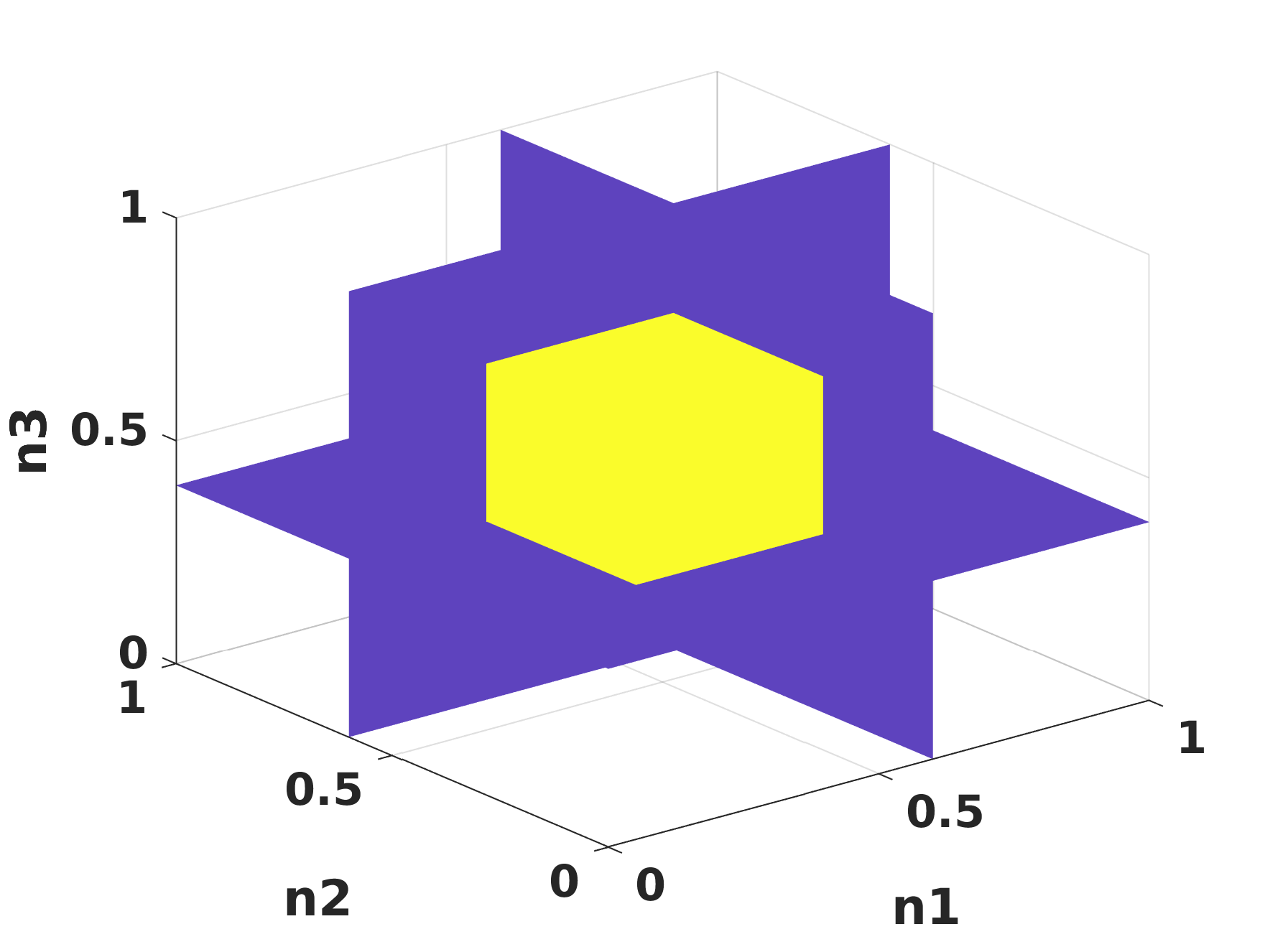}}
 	%\subfigure[U-shaped]{\includegraphics[width=0.32\textwidth]{3D_Figs_140220/RightHandSides/RHS_L_7_ex3.png}}
 	\subfigure[$H$-type]{\includegraphics[width=0.32\textwidth]{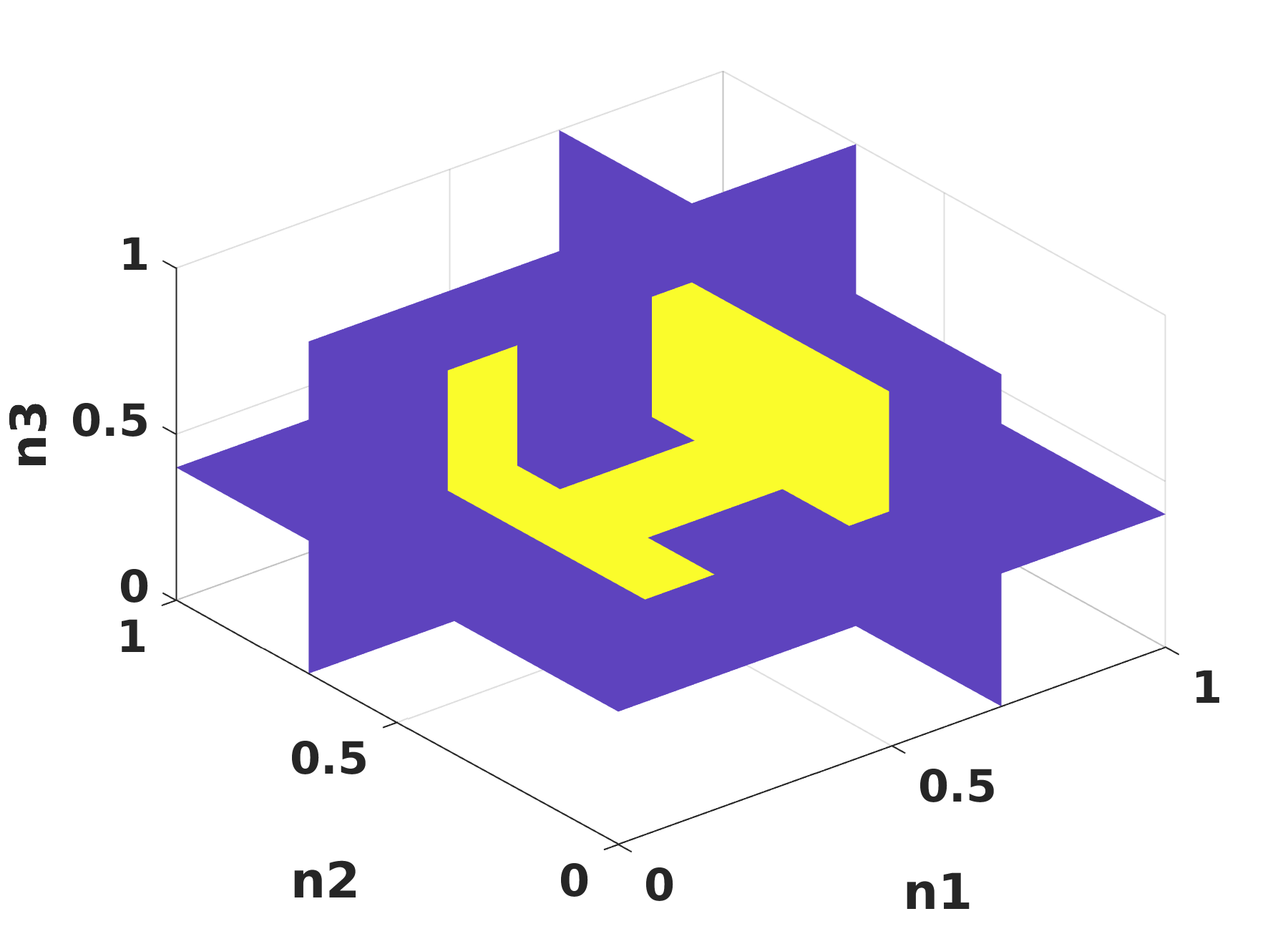}}
 	\caption[Test 2: Lösungen $u$ mit $\alpha = 1$]{Box-type and $H$-form-type 
 	right hand sides $y_\Omega$ for the 3D control problem.}
 	\label{design2}
 \end{figure}
 For solution of the control problem
%2D case 
we use the partly high oscillating equation coefficients 
\begin{align}
a_1(x_1) &= \sin(100\pi x_1) + 1.1, \label{fun:a1}\\
a_2(x_2) &= \sin(x_2)\cos(x_2), \label{fun:a2}\\
a_3(x_3) &= \cos(5\pi x_3)+2 \label{fun:a3}
\end{align}
for the construction of the diagonal equation coefficient matrices in a 
form (\ref{eqn:DiagCoef2}) and (\ref{eqn:DiagCoef3}) for the 2D and 3D cases, respectively.\\

% $$\mathbb{A}(x) = \diag \{a_1(x_1), a_2(x_2)\}$$  and 
% a similar coefficient for the 3D case,  
% \[
%  \mathbb{A}(x) = \diag \{a_1(x_1), a_2(x_2), a_3(x_3)\}.
% \]
In what follows, different preconditioning methods are used: in 2D case, we use an imprecise 
low $K$-rank approximation of the inverse matrix $A_u^{-1}$ as a preconditioner, 
whereas in 3D case, the classical anisotropic Laplace operator 
in the low-rank Kronecker form is used as a preconditioner, see chapter \ref{ssec:Precond},
Theorem \ref{lem:cond_bound_var} and \cite{HKKS:18}.

All simulations are performed in Matlab 2019b on a laptop with 16GB RAM and Intel(R) Core(TM) i7-8650U, 
using Ubuntu 18.04.

We use the low-rank canonical representation for the solution vector ${\bf u}$ and for the right-hand side,
and a short-term Kronecker product decomposition of all matrices involved. We maintain the 
quasi-optimal rank bound for the solution vector adapted to the given accuracy threshold. 
For the decomposition of the governing operators in 3D case we apply the multigrid Tucker 
tensor approximation \cite{KhKh3:08} and the subsequent Tucker-to-canonical transform. 
The adaptive rank reduction for the rank-structured tensors representing the current iterant
for the vector ${\bf u}$ in the course of PCG iteration is calculated 
by the canonical-to-Tucker algorithm \cite{KhKh:06} combined with the Tucker-to-canonical transform. 
The basic tensor operations are performed by using
the programs from the Matlab TESC package on tensor numerical methods developed in the recent years 
by the second and third authors,  
see \cite{Khor2Book2:18} for short descriptions and related references.

 \subsection{Numerical tests for 2D case}
\label{ssec:numerics_tensor_2D}

First, we validate the usage of the tensor structured PCG algorithm 
%for the laplace operator (which refers to $\alpha = 1$) 
 by comparison with the backslash Matlab solver that is applied to the direct finite difference method 
 discretization of equations (\ref{eqn:Lagrange_cont}) and (\ref{eqn:State}) as well as by investigating 
 the singular values of the involved operator $A_u$. 
 
We also present the solutions for the optimal control $\uu$ and for the state variable $\y$ in the respective
equations (\ref{eqn:Lagrange_cont}) and (\ref{eqn:State}) and investigate the impact of 
different regularization parameter values $\gamma = 1, \, 0.01$, and different fractional parameters 
$\alpha$. Note that the preconditioning in 2D case is done by making use of the direct low-rank approximation to the reciprocal matrix valued function $\mathcal{F}_3(A)$, see section \ref{ssec:Precond}.

Table \ref{Tab:times_2d} shows the times needed by both the rank structured pcg solver and the 
backslash Matlab solver for solving equation 
\begin{equation} \label{eqn:4.4}
A^\alpha \uu = \yom 
\end{equation} 
for different 
grid sizes, the box-type design function and with a fixed parameter $\alpha=1$. 
One can clearly observe that the rank structured pcg solver 
outperforms the full Matlab solver for a grid size $n \geq 128 $. 

Table \ref{Tab:times_2d_full} shows the times needed by the different solvers when considering
equation (\ref{eqn:Lagrange_cont}) for a fixed grid size of $n = 64$ grid points in each dimension and different values for $\alpha$ as well as 
the time needed to set up the operator $A_u$ as preliminary work for both solvers. 
Due to a lack of memory capacity, 
it is not possible to store the operator $A_u$ in full size format for grid sizes with $n > 64$ grid points in each dimension, 
see section \ref{ssec:numerics_tensor_3D} for more details. 
%\textcolor{red}{set times in relation to number of iteration --> no, bc we investigate full solver} 
\begin{table}[htb]
\begin{center}%
{\footnotesize
\begin{tabular}
[c]{|c|c|c| }%
\hline
  grid points   &    time pcg &  time full solver \\
 \hline
 64    & 0,0079    & 0,0035      \\
  \hline
 128    & 0,0145    & 0,0183       \\
     \hline 
 256	& 0,0306		& 0,0796 		\\
 	\hline
 512	& 0,0454	& 0,4122		\\
 	\hline
 1024	& 0,0937	& 1,8072		\\
 	\hline
 2048	& 0,5994	& 8,5634		\\
 	\hline
  \end{tabular}
\caption{\small Times in seconds needed by the rank-structured pcg scheme and the full-size scheme 
to solve (\ref{eqn:4.4}) with different numbers of grid points and $\alpha = 1$.}
\label{Tab:times_2d}
}
\end{center}
\end{table}

%\vspace*{-3mm}
\begin{table}[htb]
\begin{center}%
{\footnotesize
\begin{tabular}
[c]{|c|r|r|r| }%
\hline
     &    $\alpha=1$ &  $\alpha=1/2$ & $\alpha=1/10$ \\
 \hline
 time low-rank pcg     & 0.0092    & 0.0089    & 0.0051  \\
  \hline
 time full solver   & %1.27 
 0.3583    & %8.13 
 0.4697     & %8.57 
 0.4847  \\
     \hhline{|=|=|=|=|}

 time setting up operator (low-rank) &	0.1097 & 0.0264 & 0.0422\\
 \hline
 time setting up operator (full) &	1.1808 	& 21.2779 & 22.5202 \\
 \hline
  \end{tabular}
\caption{\small Times for the rank-structured pcg solver and the full-size solver 
for solution of (\ref{eqn:Lagrange_cont}) with $n=64$ grid points in each dimension and different $\alpha$ 
and times needed to set up $A_u$ as preliminary work.}
\label{Tab:times_2d_full}
}
\end{center}
\end{table}
Figure \ref{fig:G3_err} shows the errors that 
occur for different $\alpha$ when using the rank structured solver compared to the full size solver 
for equation %$A^\alpha u = y_\Omega$ and 
(\ref{eqn:Lagrange_cont}), which do not exceed an favourable error bound of $10^{-6}$.
%, respectively. 
 \begin{figure}[H]
 \subfigure[$\alpha = 1$]{
 \includegraphics[width=5.0cm]{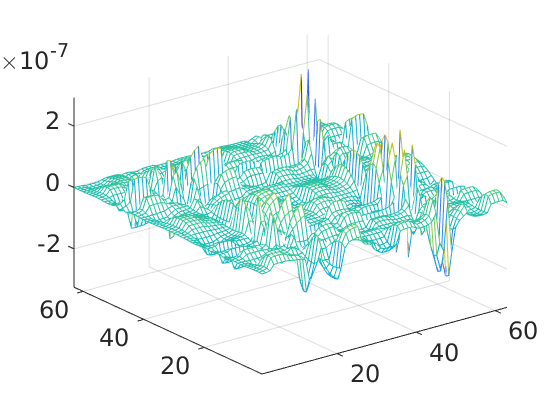}}
 \subfigure[$\alpha = 1/2$]{
 \includegraphics[width=5.0cm]{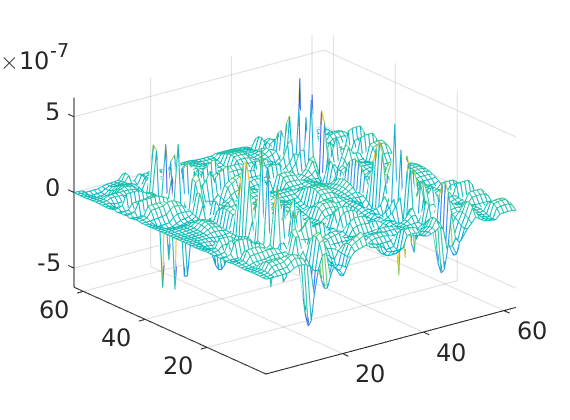}}
 \subfigure[$\alpha = 1/10$]{
 \includegraphics[width=5.0cm]{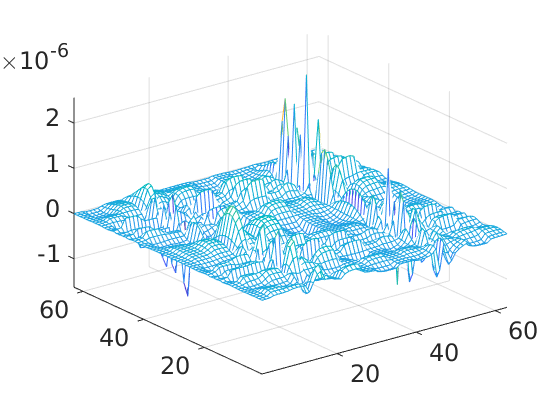}}
  \caption{\small Error for the rank-structured solver solution compared to the 
  full-size solver solution of (\ref{eqn:Lagrange_cont}), $n = 64$. } 
 \label{fig:G3_err}
 \end{figure}
  
In order to validate the existence of the accurate low Kronecker-rank approximation 
to the operator $A_u$ and its inverse, 
\[
A_{inv} = (A^\alpha + A^{-\alpha})^{-1},
\] 
which is required for preconditioning needs, we investigate the singular values of the corresponding matrices. 
Figure \ref{singvalues} demonstrates that for both operator cases $A_u$ and $A_{inv}$ and an exemplary grid size of $511$ grid points in each dimension, there is an exponential decay of 
the corresponding singular values, which justifies 
the existence of an accurate low-rank representation.
\begin{figure}[h]
 	\subfigure[$A_u$]{\includegraphics[width=0.48\textwidth]{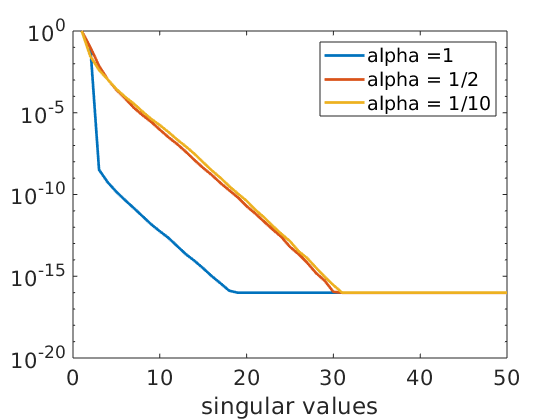}}
 	\subfigure[$A_{inv}$]{\includegraphics[width=0.49\textwidth]{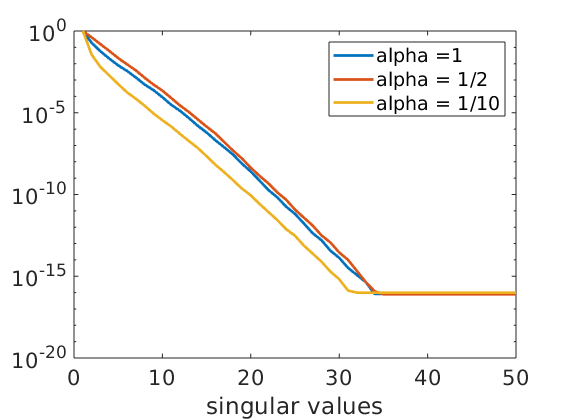}}
 	\caption[Test 2: Lösungen $u$ mit $\alpha = 1/10$]{Decay of singular values for 
 	operators $A_u$ and $A_{inv}$ with different values of $\alpha$ and a univariate grid size with $n = 511$ grid points in each dimension.}
 	\label{singvalues}
 \end{figure}  
%  
   
% \textbf{Time Complexity}\\

In what follows, we investigate the time the pcg algorithm needs  to solve \eqref{eqn:Lagrange_cont},
\[
 A_u \uu  = \yom,
\]
in our test setting. The corresponding tables and figures display the results for the box-type design function.
Table \ref{Tab:iter_2d} shows the number  of iterations needed by the solver when considering different grid 
sizes and different values for $\alpha$. The results validate a grid independence of the used pcg solver 
concerning the needed numbers of iterations. 
 \begin{table}[H]
\begin{center}%
{\footnotesize
\begin{tabular}
[c]{|c|c|c|c| }%
\hline 
  grid points   &    $\alpha=1$ &  $\alpha=1/2$ & $\alpha=1/10$ \\
  \hline
 64     & 2 & 2 & 1\\ 
  \hline 
 128  & 2 & 2 & 2 \\ 
     \hline 
 256    & 2 & 2 & 2\\ 
 \hline
512 &  2 &  2& 2\\
 \hline
1024 & 2  & 2 & 2\\
 \hline
2048 & 3  &2  &2 \\
 \hline
4098 & 4  &2  &3 \\
 \hline
8196 & 4  &3  & 3\\
     \hline 
  \end{tabular}
\caption{\small Number of iterations for the low-rank solver for solution of $(A^\alpha+A^{-\alpha})\uu=\yom$ 
with different numbers of grid points and different values of $\alpha$. The box-type design function is considered.}
\label{Tab:iter_2d}
}
\end{center}
\end{table}
\vspace*{-8mm} Figure \ref{Fig2D:Time_per_iter} represents the time that the pcg algorithm needs for one iteration. 
The presented data validates the theoretical findings from section \ref{ssec:StiffnessMatrixinLR}, 
that is the numerical cost for the algorithm of the order of $O(RSn^2)$. 
Therefore, the results demonstrate that the 
low-rank pcg scheme circumvents the curse of dimensionality.
\vspace*{-4mm}
\begin{figure}[H]
\centering
 	\includegraphics[width=0.4\textwidth]{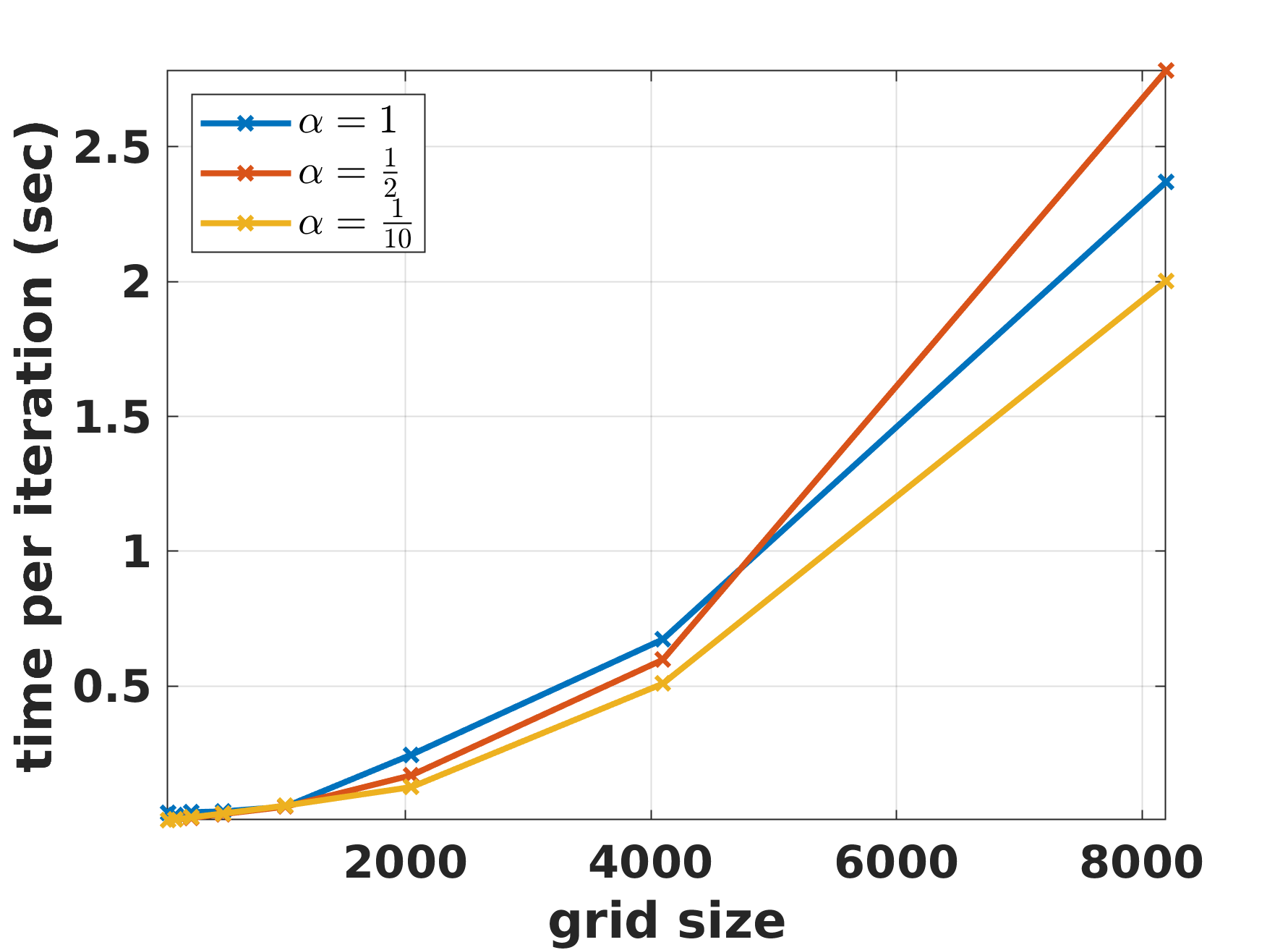}
 	\caption[Test 2: Lösungen $u$ mit $\alpha = 1/10$]{Times per iterations of PCG Algorithm for different 
 	grid sizes and different values of $\alpha$ considering the box-type design function.}
 	\label{Fig2D:Time_per_iter}
 \end{figure}

% %    
% %Figures \ref{Abb. 20a}-\ref{Abb. 21} show the solutions computed by the rank-structured 
% pcg scheme for test functions $a_1(x_1) = \sin(6\pi x_1)+1,1$ and $a_2(x_2) = 0,1 + x_2^2$ 
% and different design functions $y_\Omega$. 

\subsubsection{Solution for optimal control}
\label{subsubsec:u}

In this section, we present the solution for the optimal control $\uu$, which means 
that with the help of the low-rank PCG algorithm, we solve equation (\ref{eqn:Lagrange_cont}), 
$$A_u \uu = (A^\alpha + A^{-\alpha})\uu = \yom.$$

Figures \ref{Fig2D:u_square} - \ref{Fig2D:u_h_gam2} show the solutions computed by the 
rank-structured PCG scheme using the coefficients functions (\ref{fun:a1}) and (\ref{fun:a2}), 
different right hand sides $\yom$, fractional exponents $\alpha$, and a grid of size $n = 255$ 
grid points in each dimension. 
In figures \ref{Fig2D:u_square} and \ref{Fig2D:u_h} we consider $\gamma = 1$, whereas in 
figure \ref{Fig2D:u_h_gam2}, we investigate the impact of a small regularization parameter $\gamma = 0.01$.\\ 
The effect of the highly 
oscillating coefficient test function (\ref{fun:a1}) on the structure of the optimal control $\uu$ 
can be recognized in all figures.
 \begin{figure}[H]
 	\subfigure[$\alpha = 1$]{\includegraphics[width=0.32\textwidth]{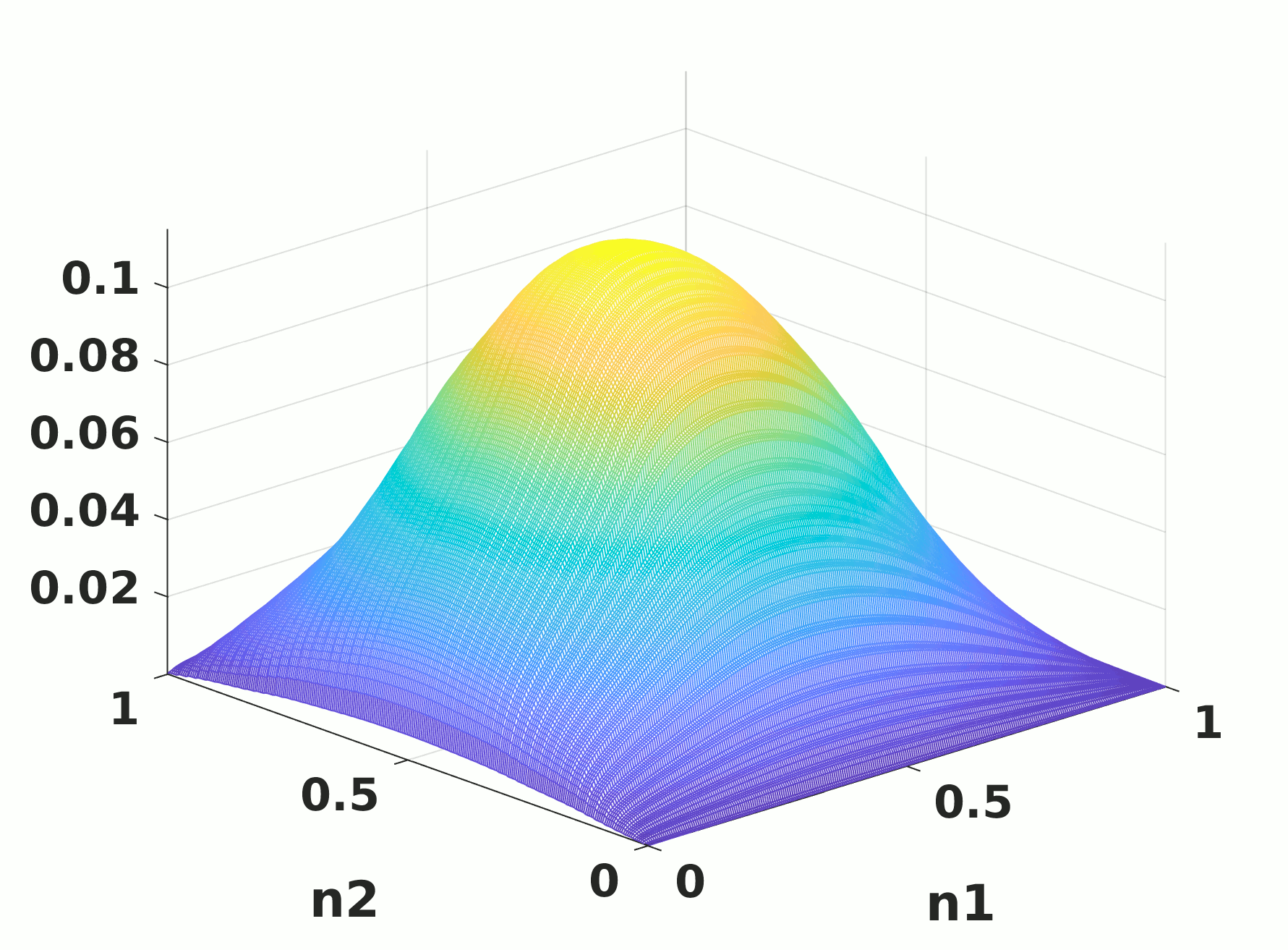}}
 	\subfigure[$\alpha = 1/2$]{\includegraphics[width=0.32\textwidth]{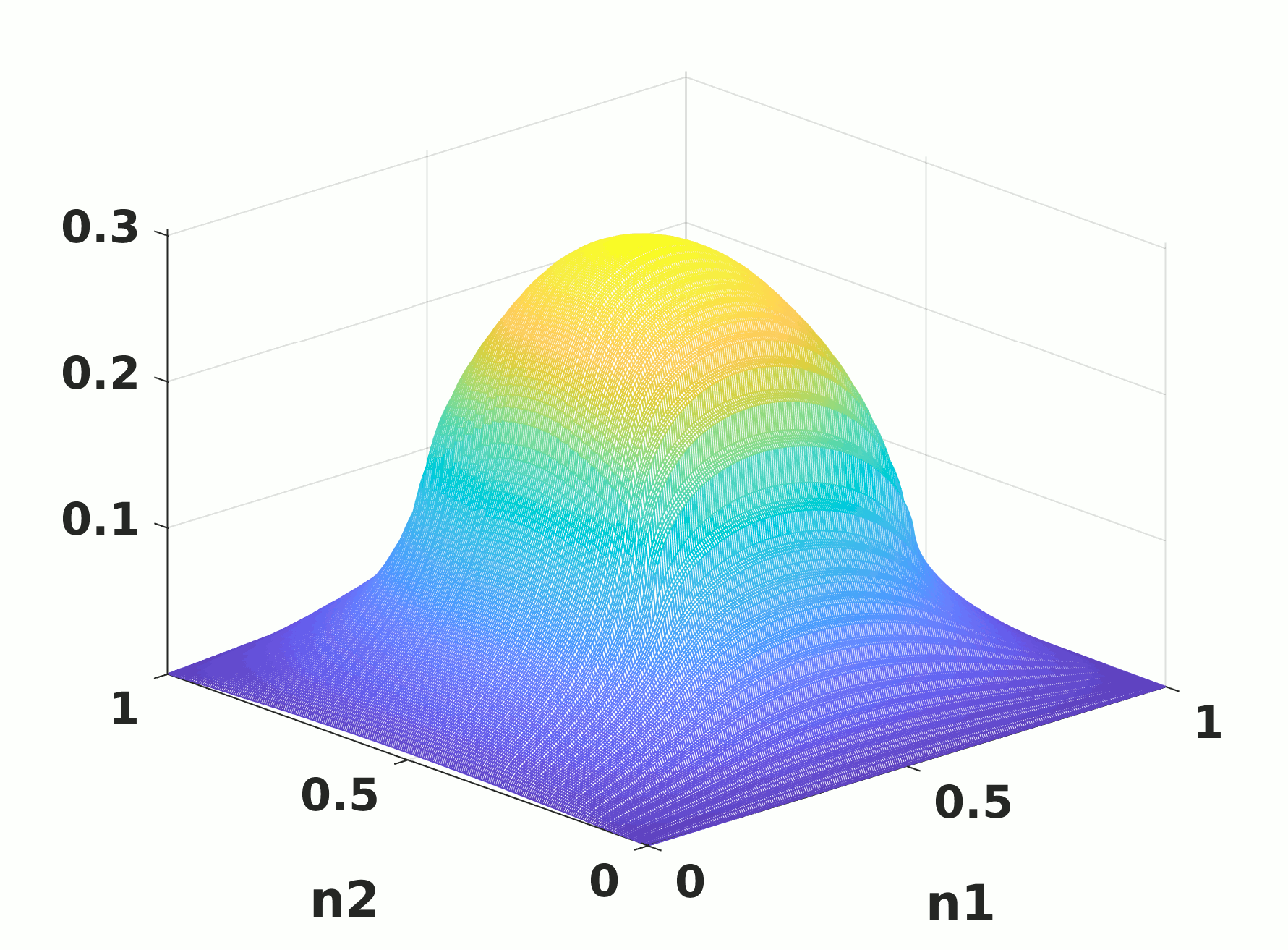}}
 	 \subfigure[$\alpha = 1/10$]{\includegraphics[width=0.32\textwidth]{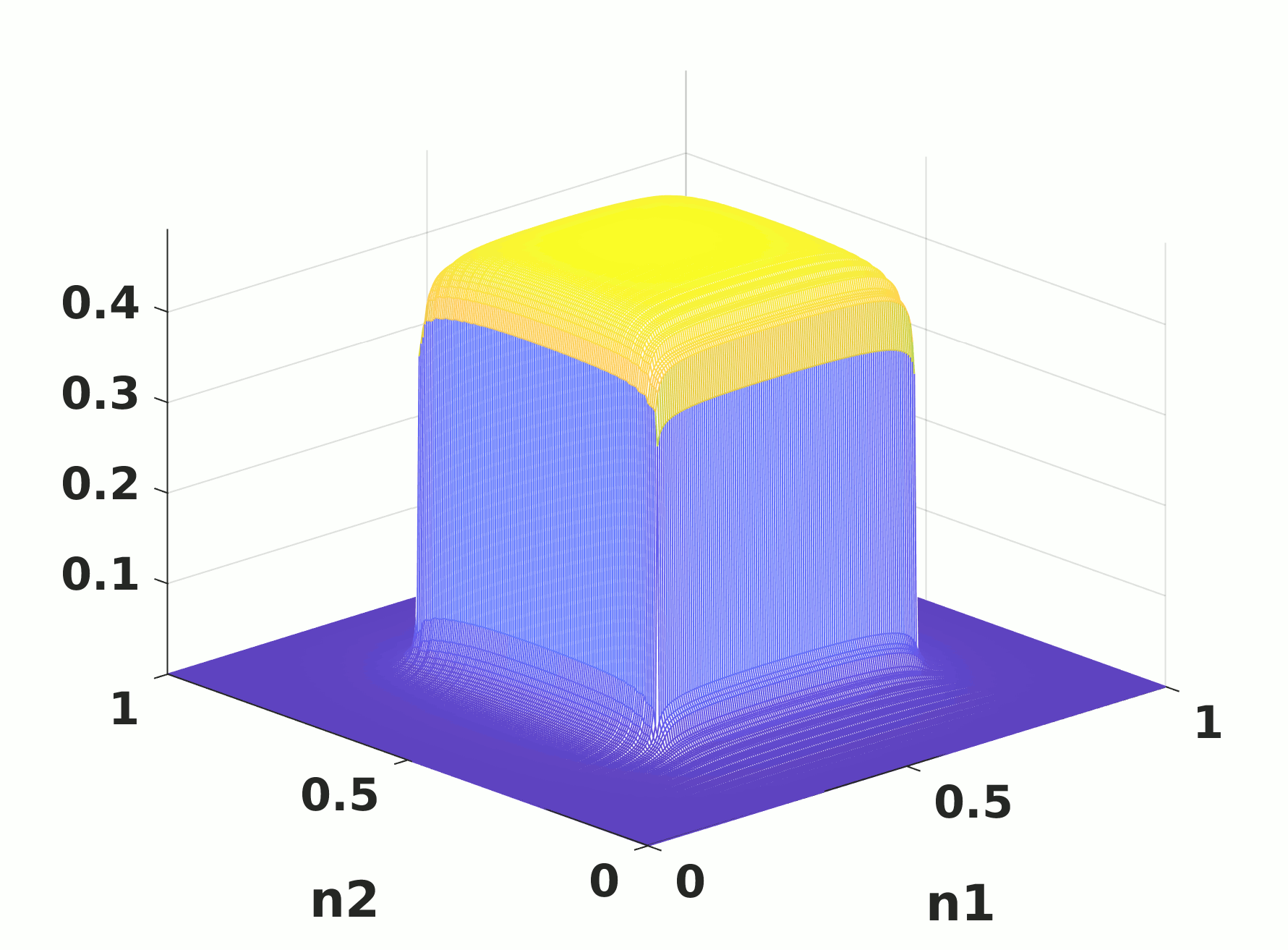}}
 	\caption[Test 2: Lösungen $u$ mit $\alpha = 1$]{Impact of fractional exponent $\alpha$ on solutions $\uu$ 
 	of (\ref{eqn:Lagrange_cont}) for box-type right hand side $\yom$, $\gamma = 1$.}
 	\label{Fig2D:u_square}
 \end{figure}
 
  \begin{figure}[H]
 	\subfigure[$\alpha = 1$]{\includegraphics[width=0.32\textwidth]{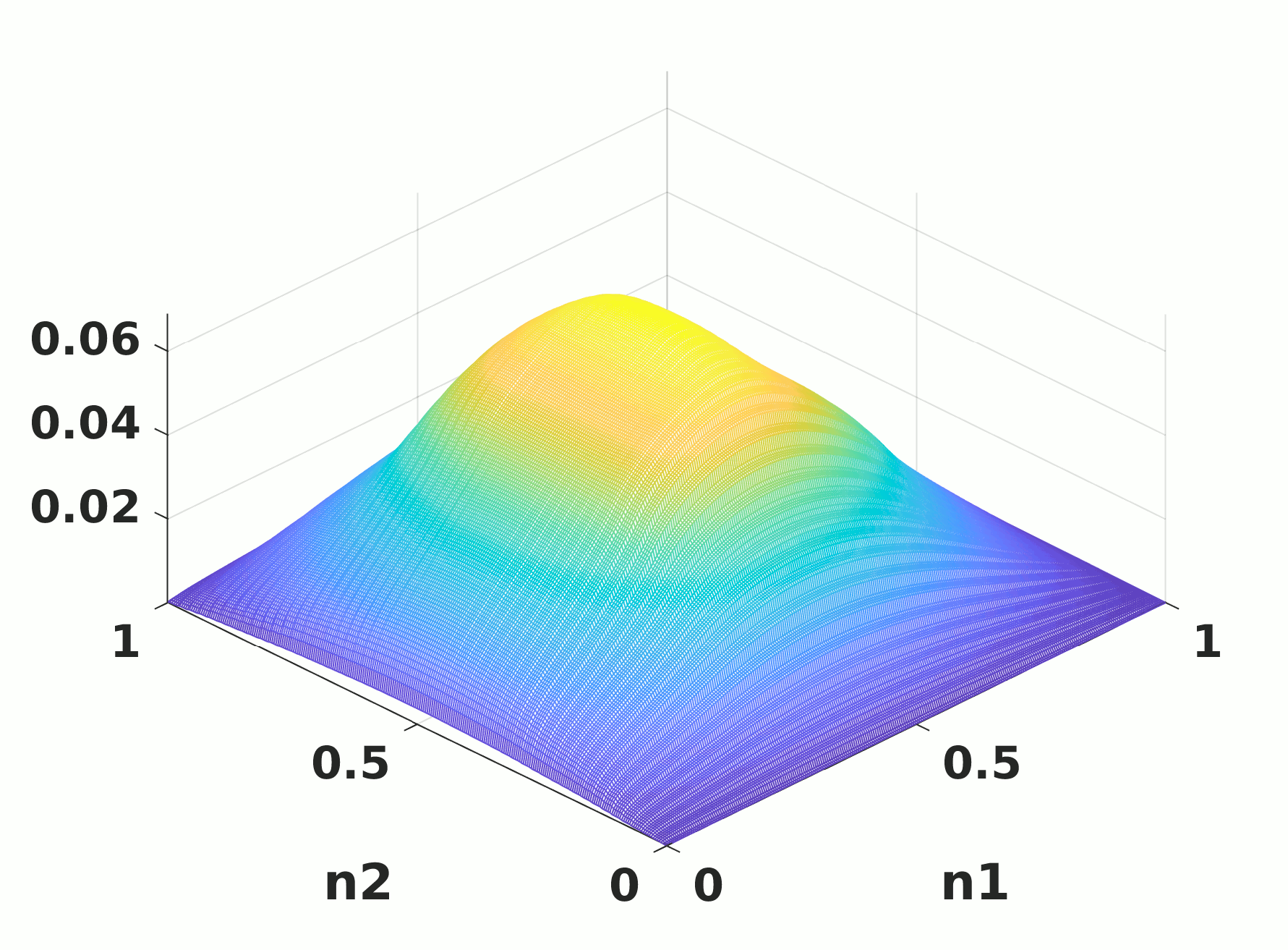}}
 \subfigure[$\alpha = 1/2$]{\includegraphics[width=0.32\textwidth]{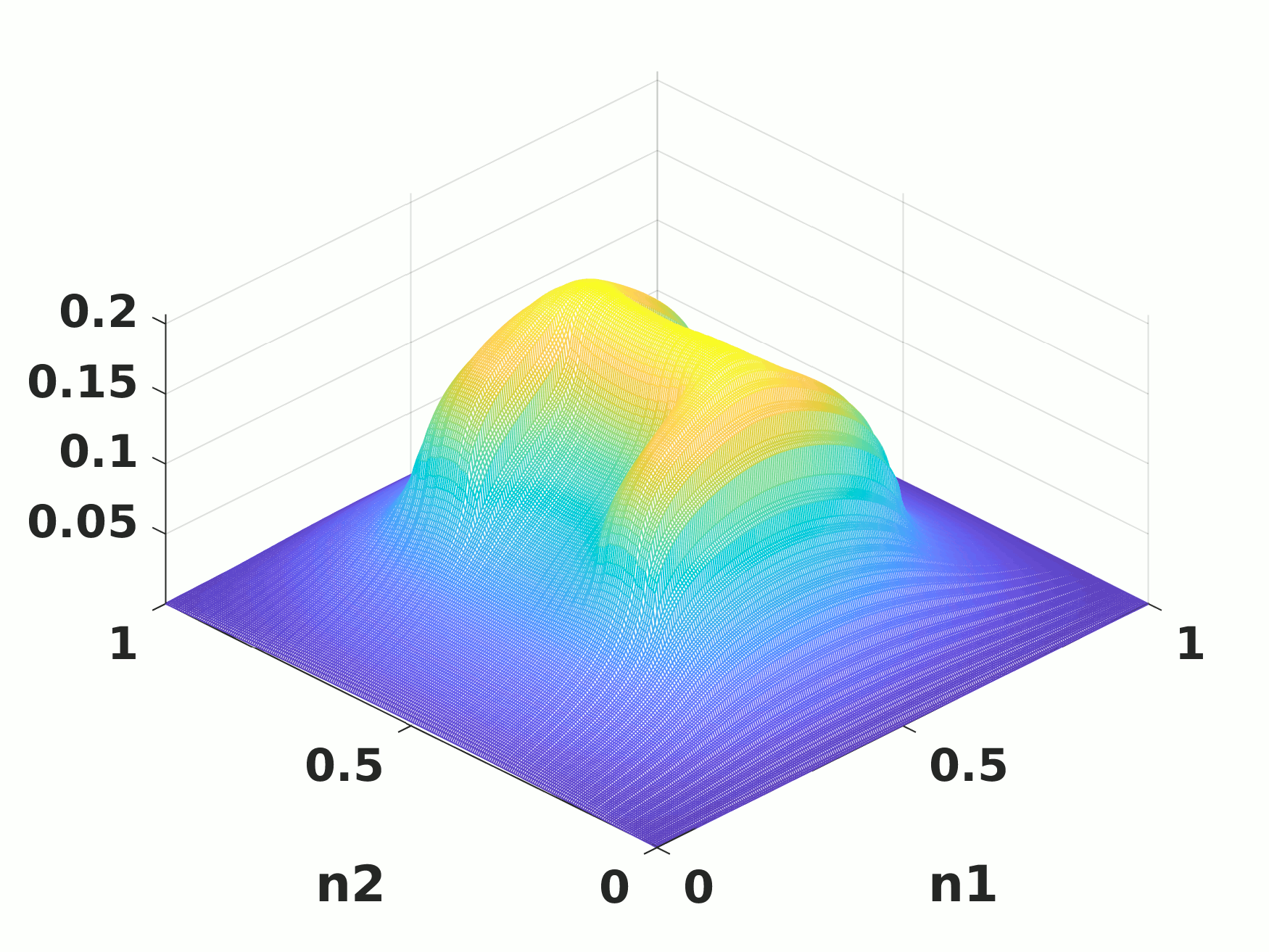}}
 	\subfigure[$\alpha = 1/10$]{\includegraphics[width=0.32\textwidth]{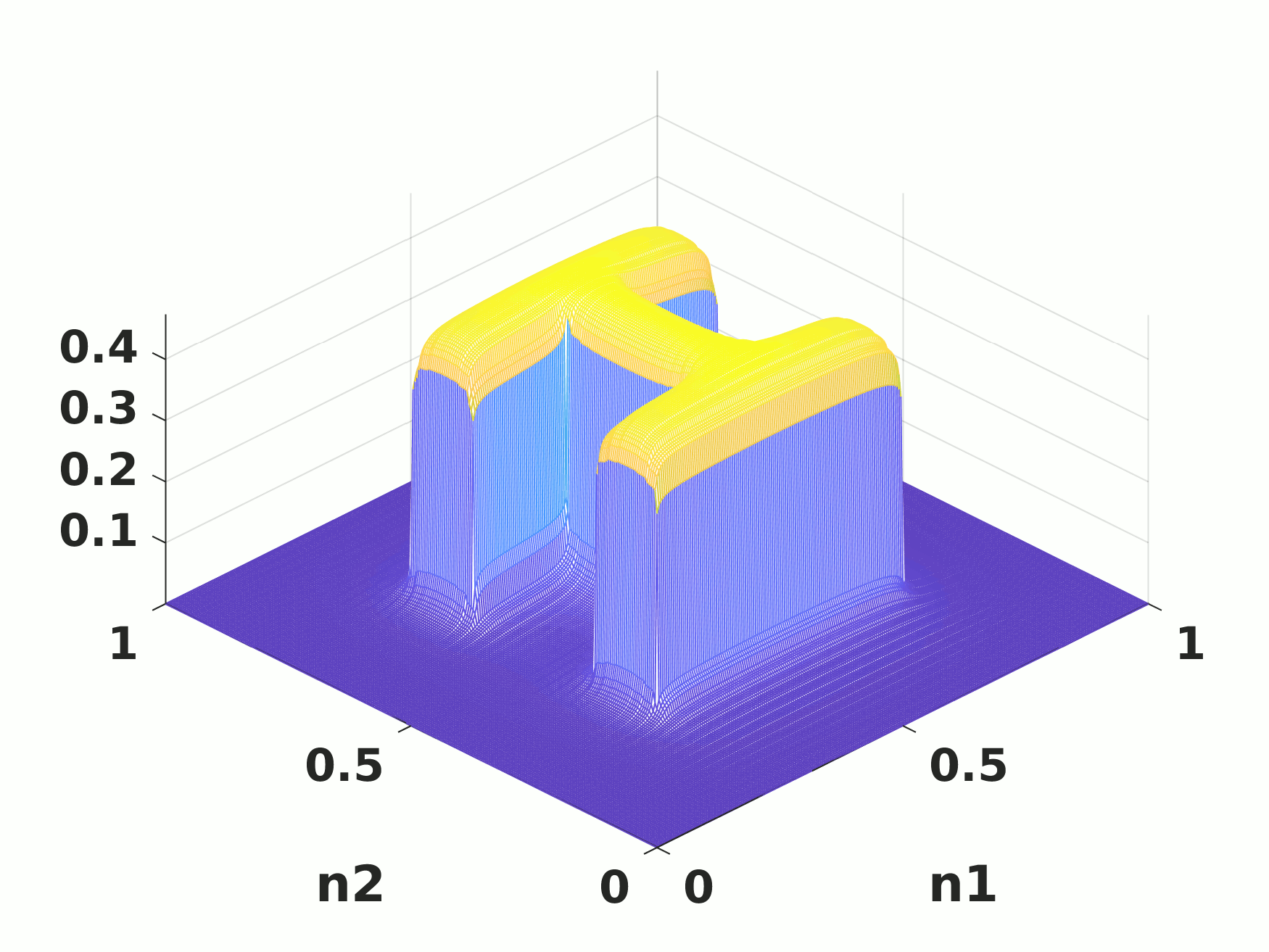}}
 	\caption[Test 2: Lösungen $u$ mit $\alpha = 1/10$]{Impact of fractional exponent $\alpha$ on solutions $\uu$ 
 	of (\ref{eqn:Lagrange_cont}) for $H$-type right hand side $\yom$, $\gamma = 1$.}
 	\label{Fig2D:u_h}
 \end{figure}   
 
   \begin{figure}[H]
 	\subfigure[$\alpha = 1$]{\includegraphics[width=0.32\textwidth]{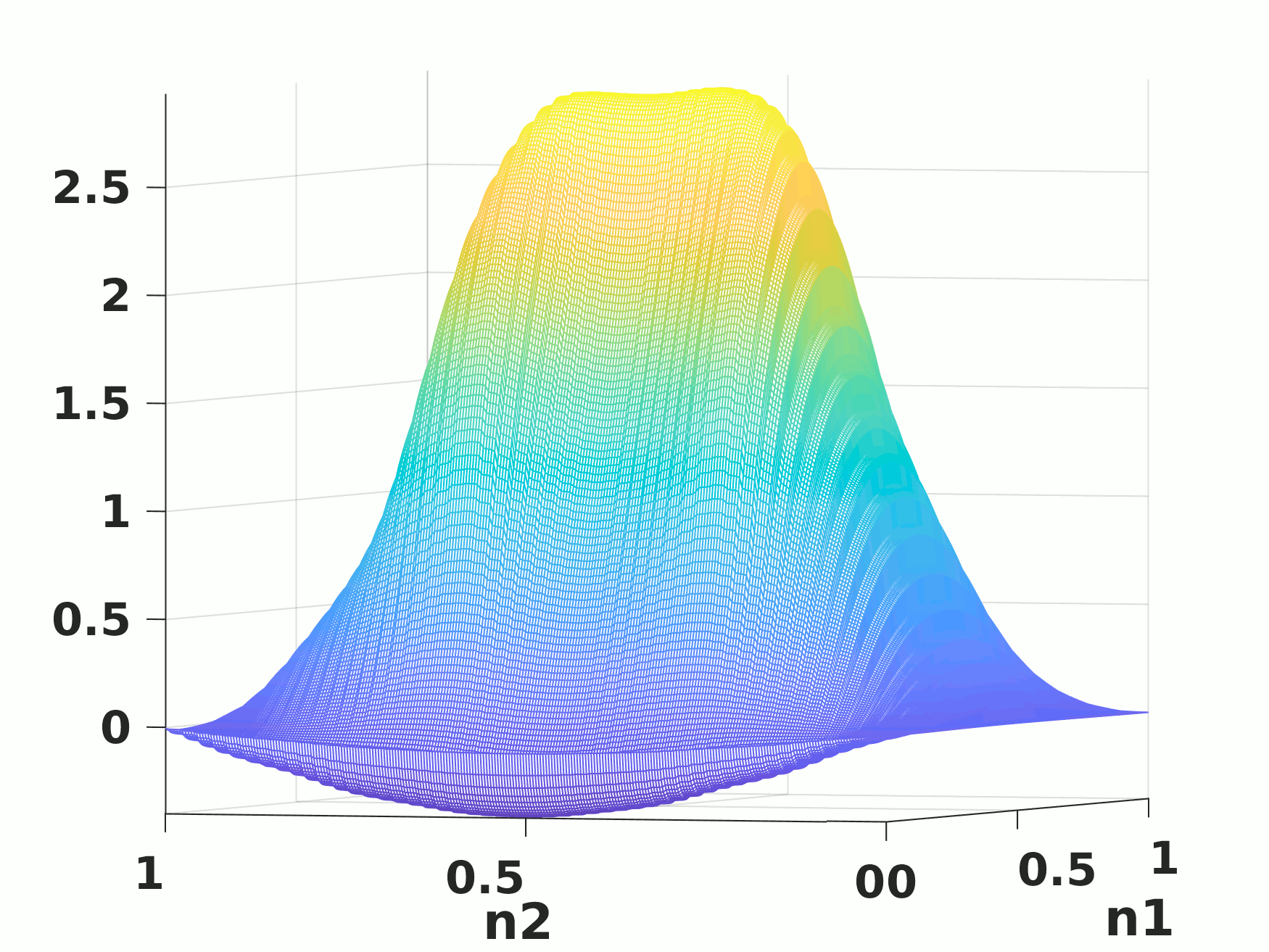}}
 	\subfigure[$\alpha = 1/2$]{\includegraphics[width=0.32\textwidth]{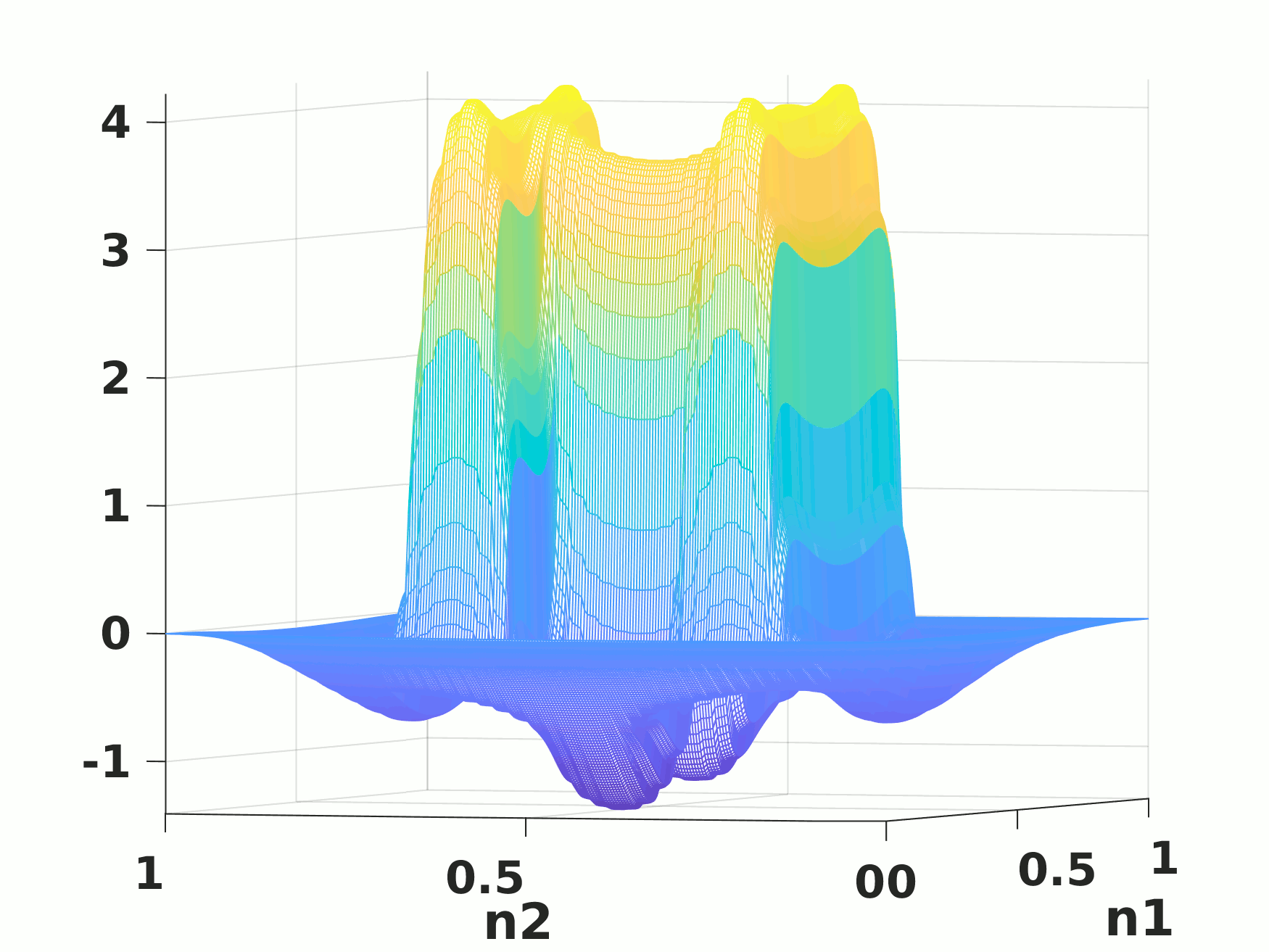}}
 	\subfigure[$\alpha = 1/10$]{\includegraphics[width=0.32\textwidth]{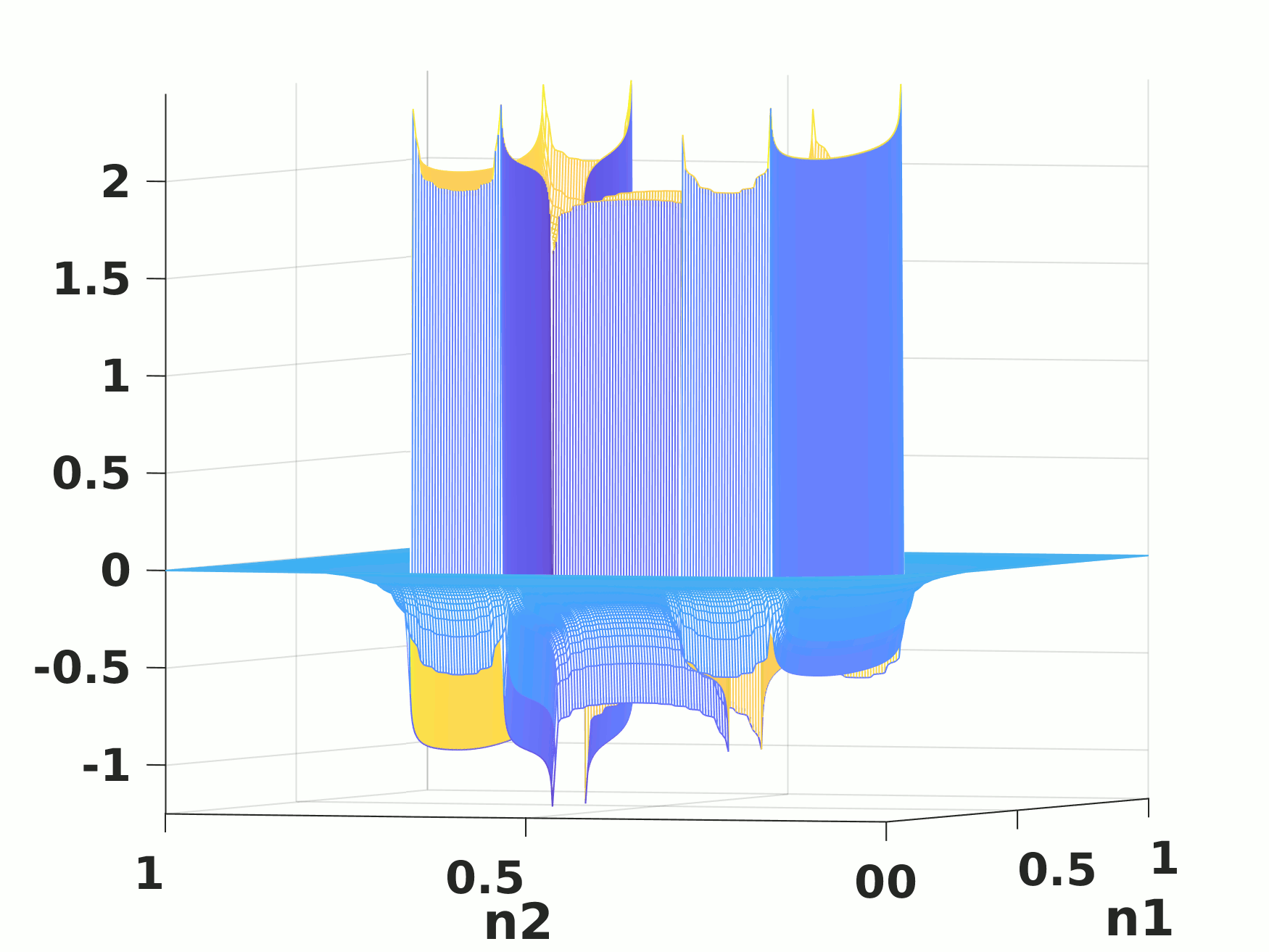}}
 	\caption[Test 2: Lösungen $u$ mit $\alpha = 1/10$]{Impact of small regularization parameter 
 	$\gamma = 0.01$ on solutions $\uu$ for $H$-type design function $\yom$ and different values for $\alpha$. }
 	\label{Fig2D:u_h_gam2}
 \end{figure}

 \subsubsection{Solution for State Variable}
 
   In this section, the solution for the state variable $\y$ is presented, 
   that is we solve equation (\ref{eqn:State}), 
   $$ \y = A^{-\alpha} \uu, 
   $$ 
   where $\uu$ is the solution of (\ref{eqn:Lagrange_cont}) presented in section \ref{subsubsec:u}. 
  Analogously, we consider the grid size $n = 255$ and compare the effects of different fractional exponents 
  $\alpha$ and regularization parameter values $\gamma = 1$ (figures \ref{Fig2D:y_square} and \ref{Fig2D:y_h}) 
  and $\gamma = 0.01$ (figure \ref{Fig2D:y_h_gam}).
   \vspace*{-3mm}
 \begin{figure}[H]
 	\subfigure[$\alpha = 1$]{\includegraphics[width=0.32\textwidth]{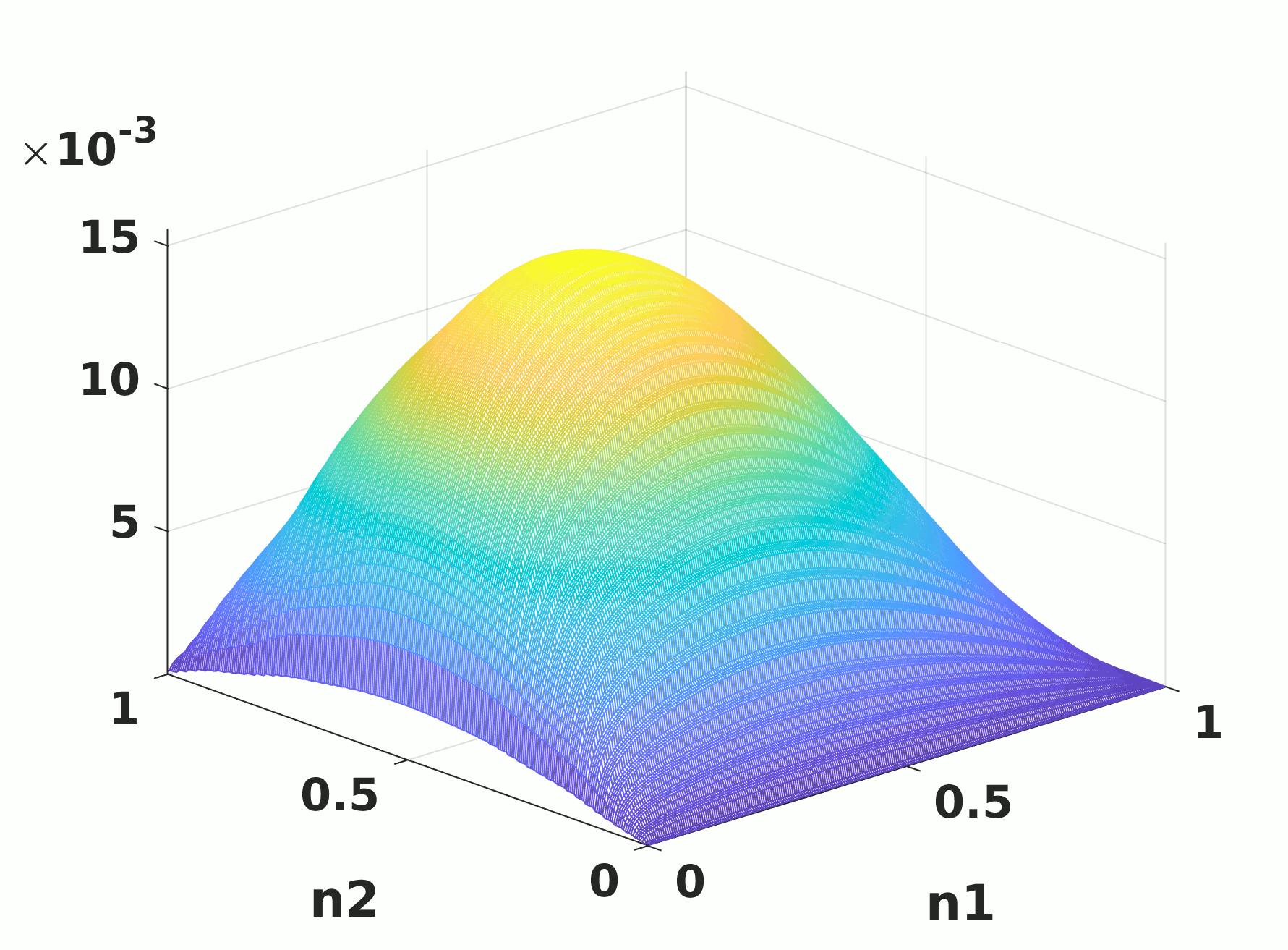}}
 	\subfigure[$\alpha = 1/2$]{\includegraphics[width=0.32\textwidth]{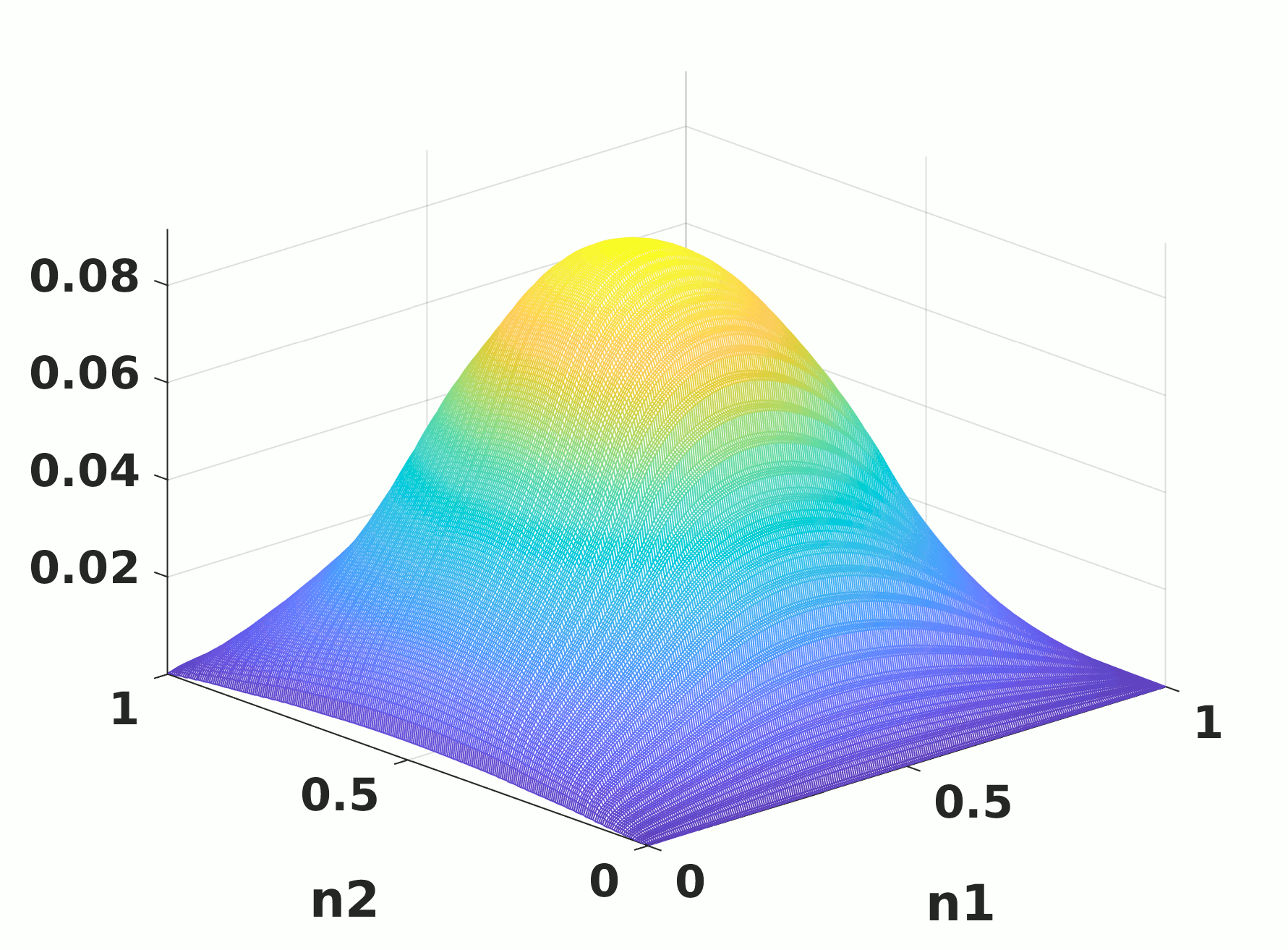}}
 	\subfigure[$\alpha = 1/10$]{\includegraphics[width=0.32\textwidth]{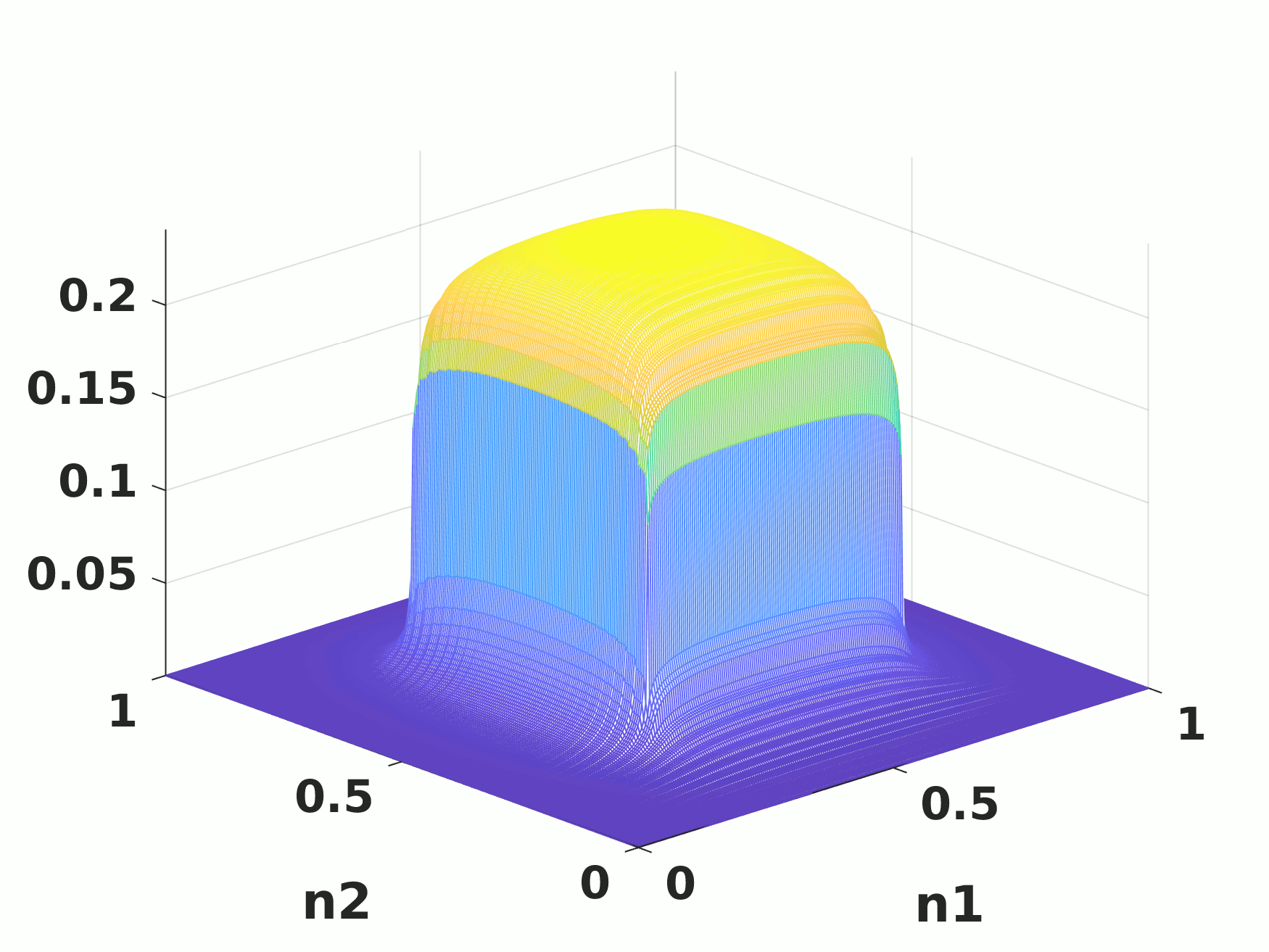}}
 	\caption[Test 2: Lösungen $u$ mit $\alpha = 1/10$]{Impact of fractional exponent $\alpha$ on solutions $\y$ 
 	of (\ref{eqn:State}) for box-type right hand side $\yom$, $\gamma = 1$. }
 	\label{Fig2D:y_square}
 \end{figure}     
 \vspace*{-9mm}
 \begin{figure}[H]
 	\subfigure[$\alpha = 1$]{\includegraphics[width=0.32\textwidth]{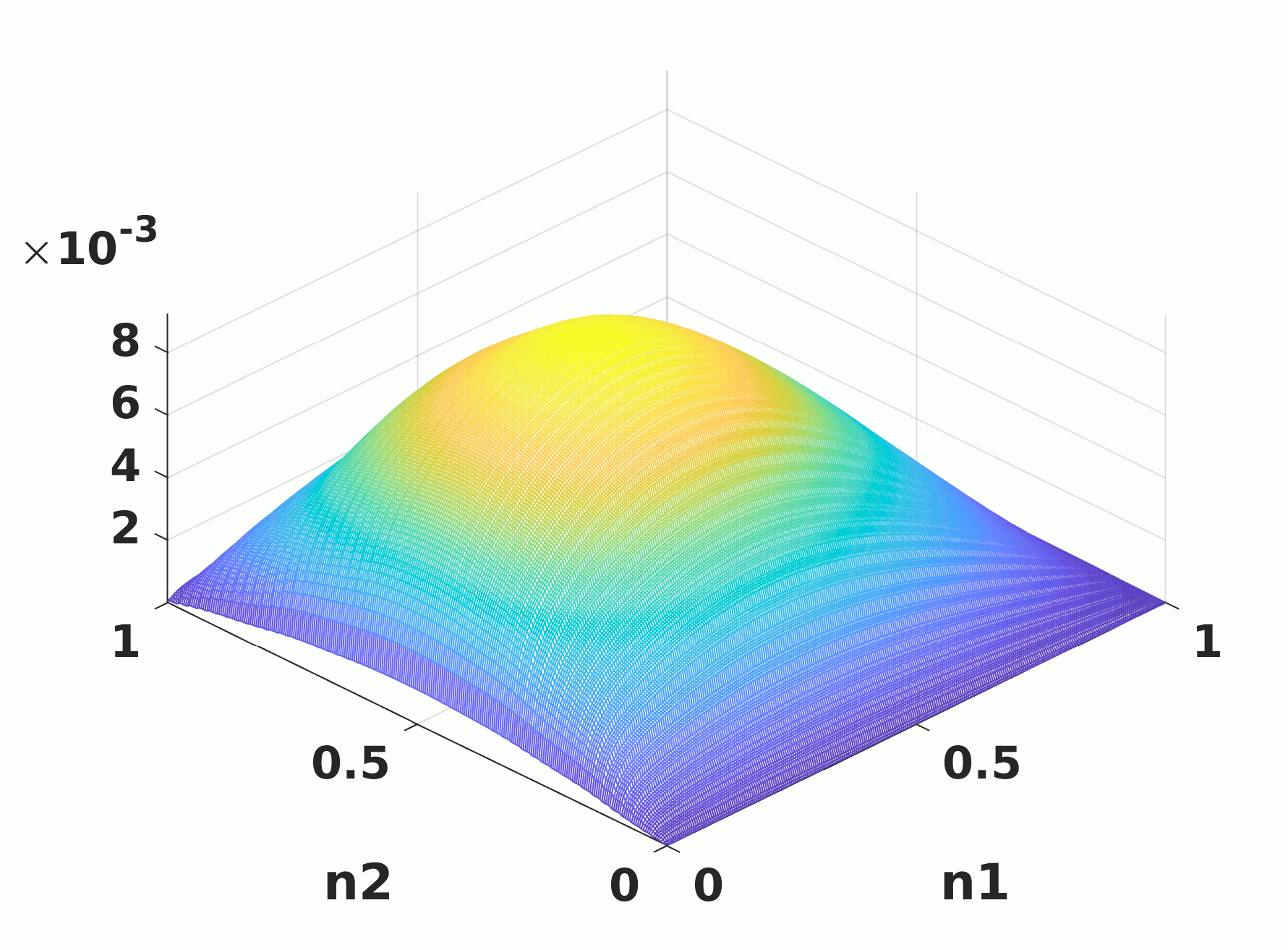}}
 	\subfigure[$\alpha = 1/2$]{\includegraphics[width=0.32\textwidth]{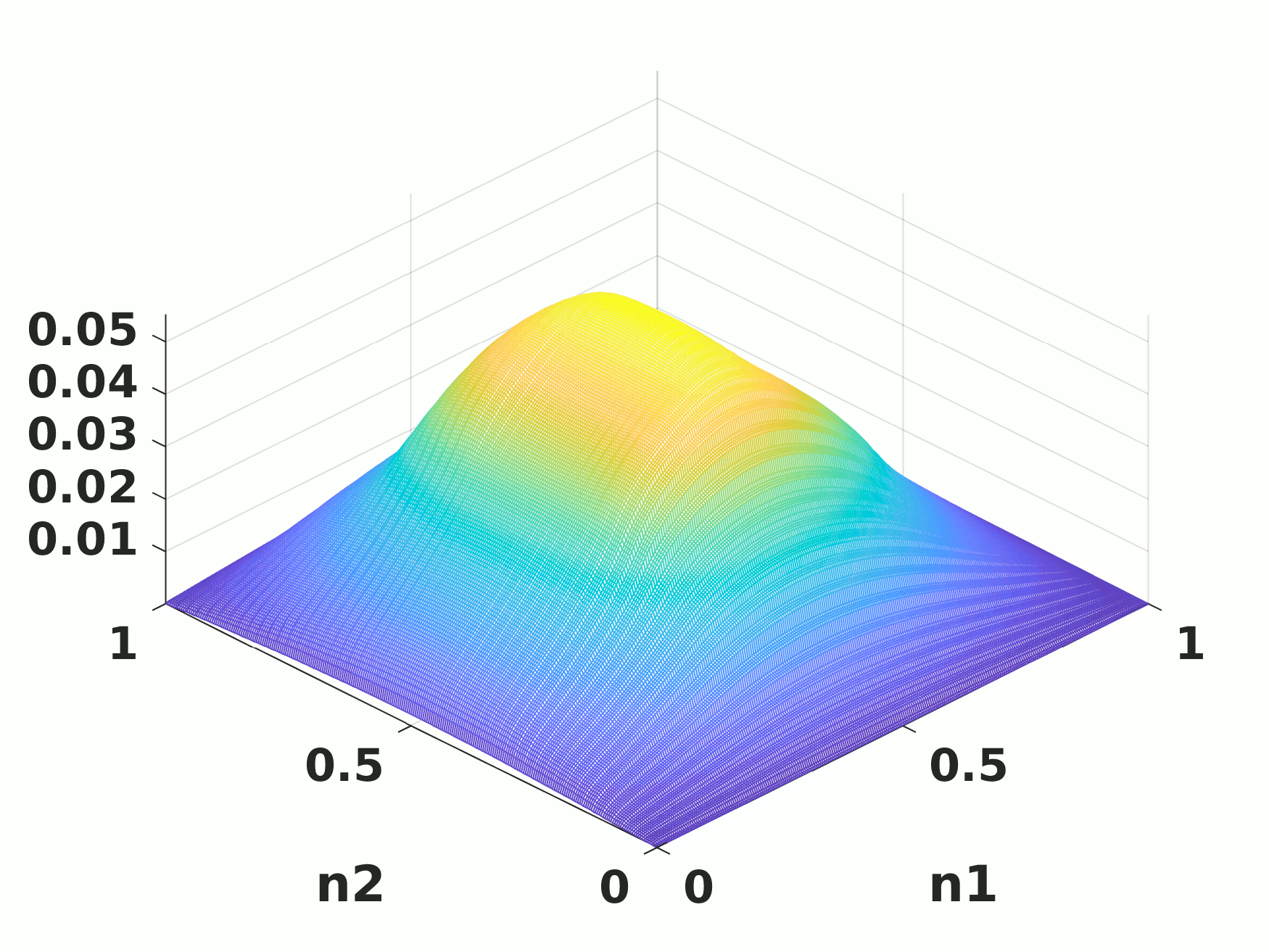}}
 	\subfigure[$\alpha = 1/10$]{\includegraphics[width=0.32\textwidth]{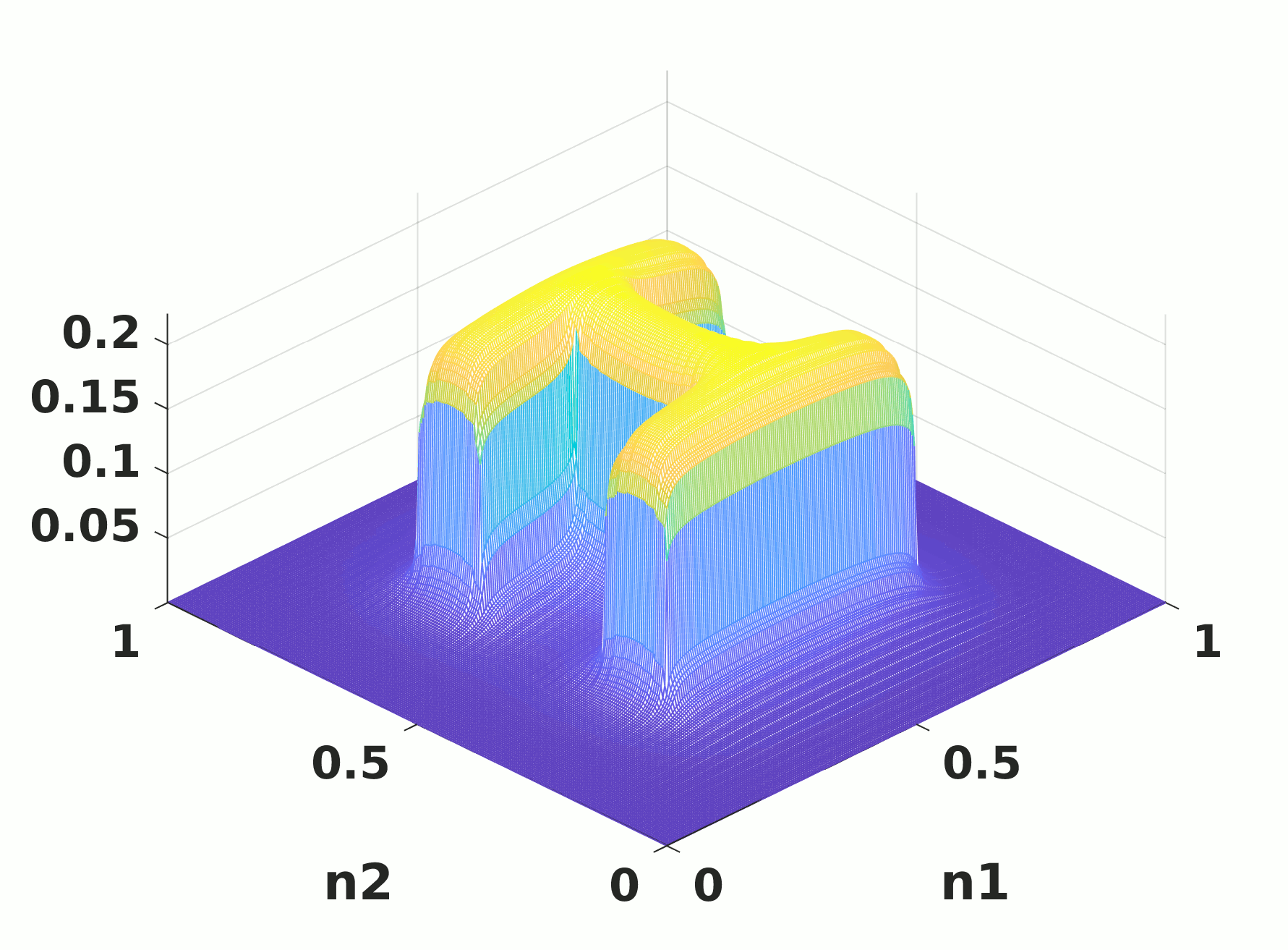}}
 	\caption[Test 2: Lösungen $u$ mit $\alpha = 1/2$]{Impact of fractional exponent $\alpha$ on solutions $\y$ 
 	of (\ref{eqn:State}) for $H$-type right hand side $\yom$, $\gamma = 1$. }
 	\label{Fig2D:y_h}
 \end{figure}

 \begin{figure}[H]
 	\subfigure[$\alpha = 1$]{\includegraphics[width=0.32\textwidth]{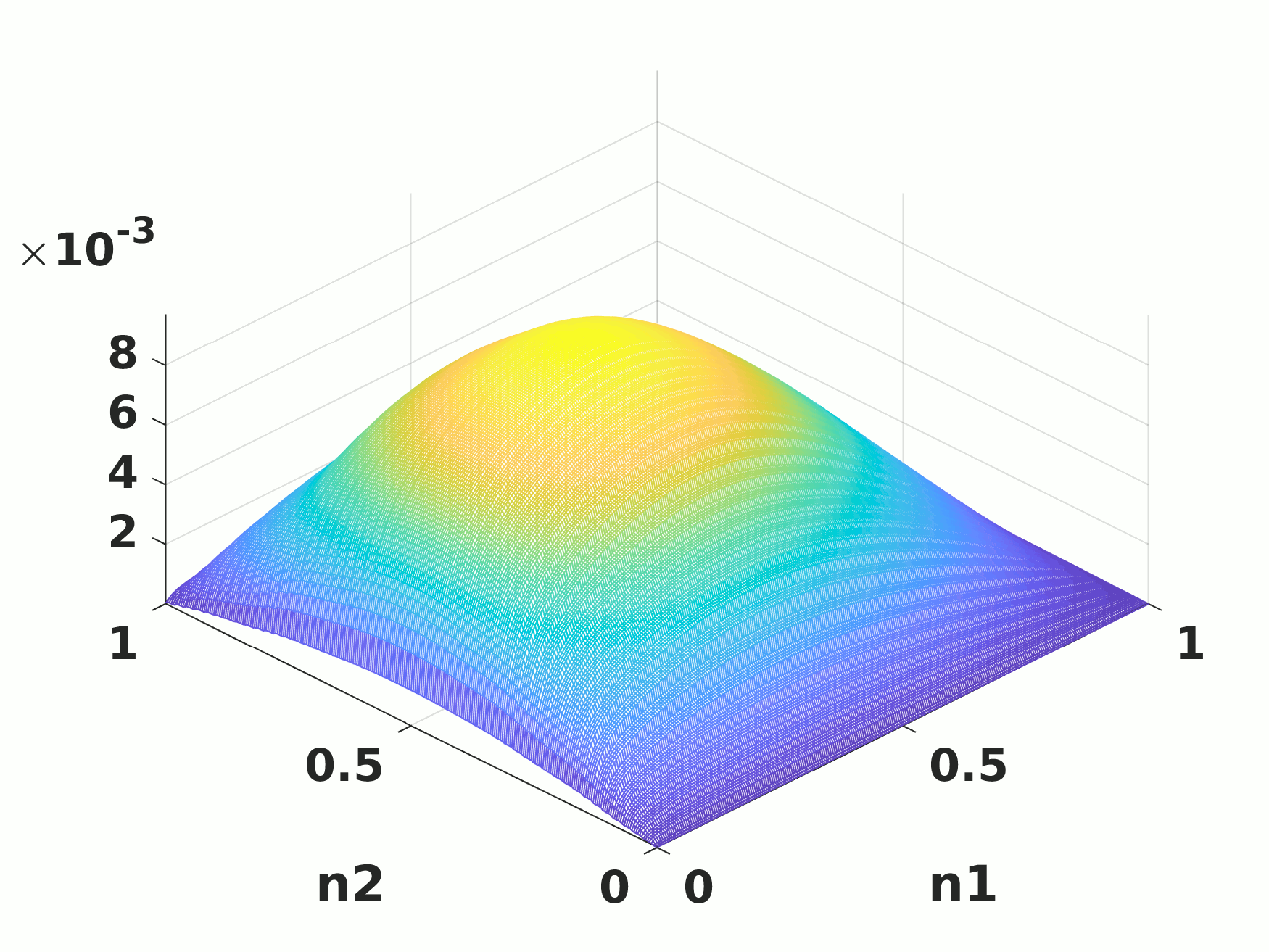}}
 	\subfigure[$\alpha = 1/2$]{\includegraphics[width=0.32\textwidth]{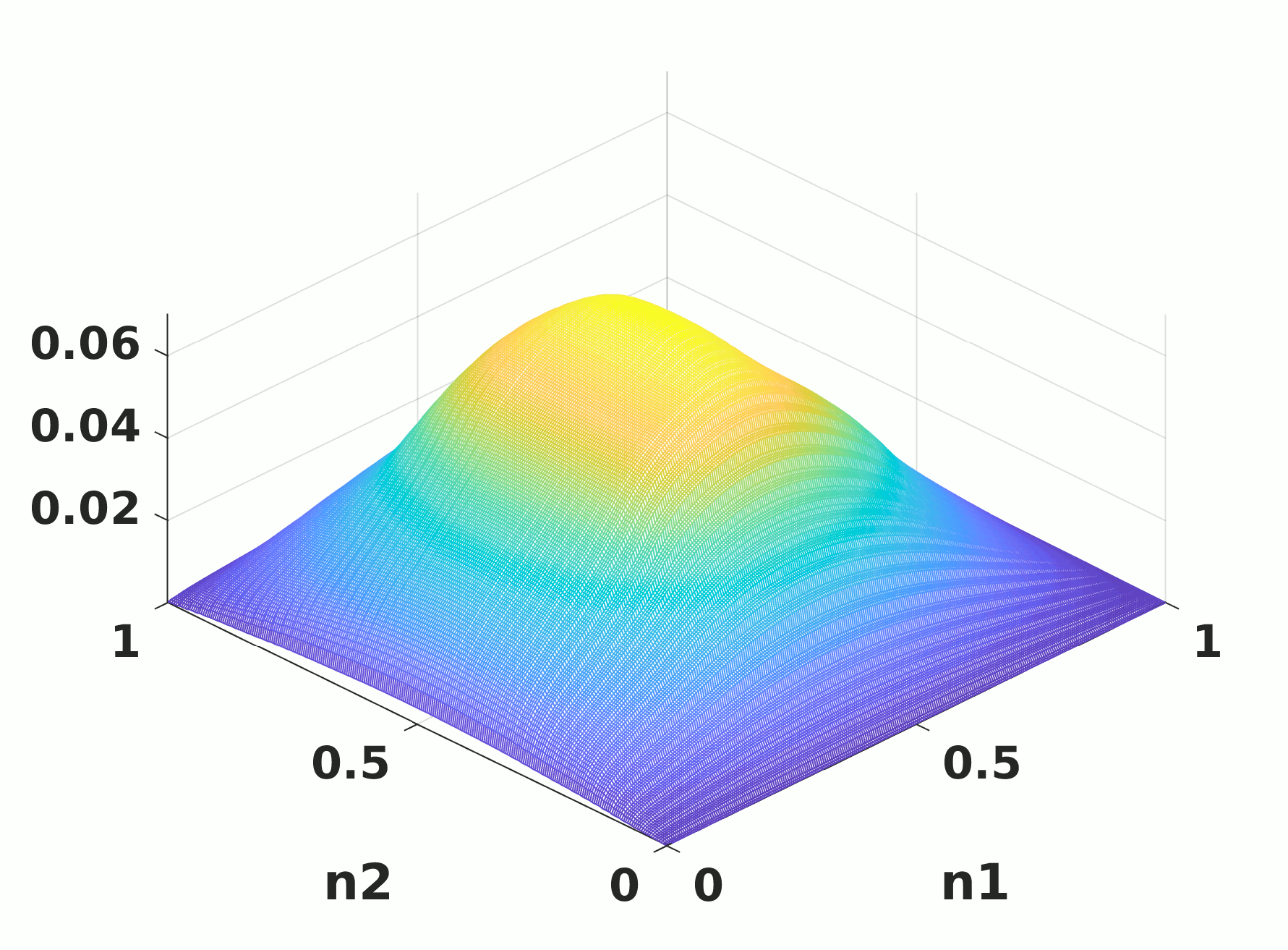}}
 	\subfigure[$\alpha = 1/10$]{\includegraphics[width=0.32\textwidth]{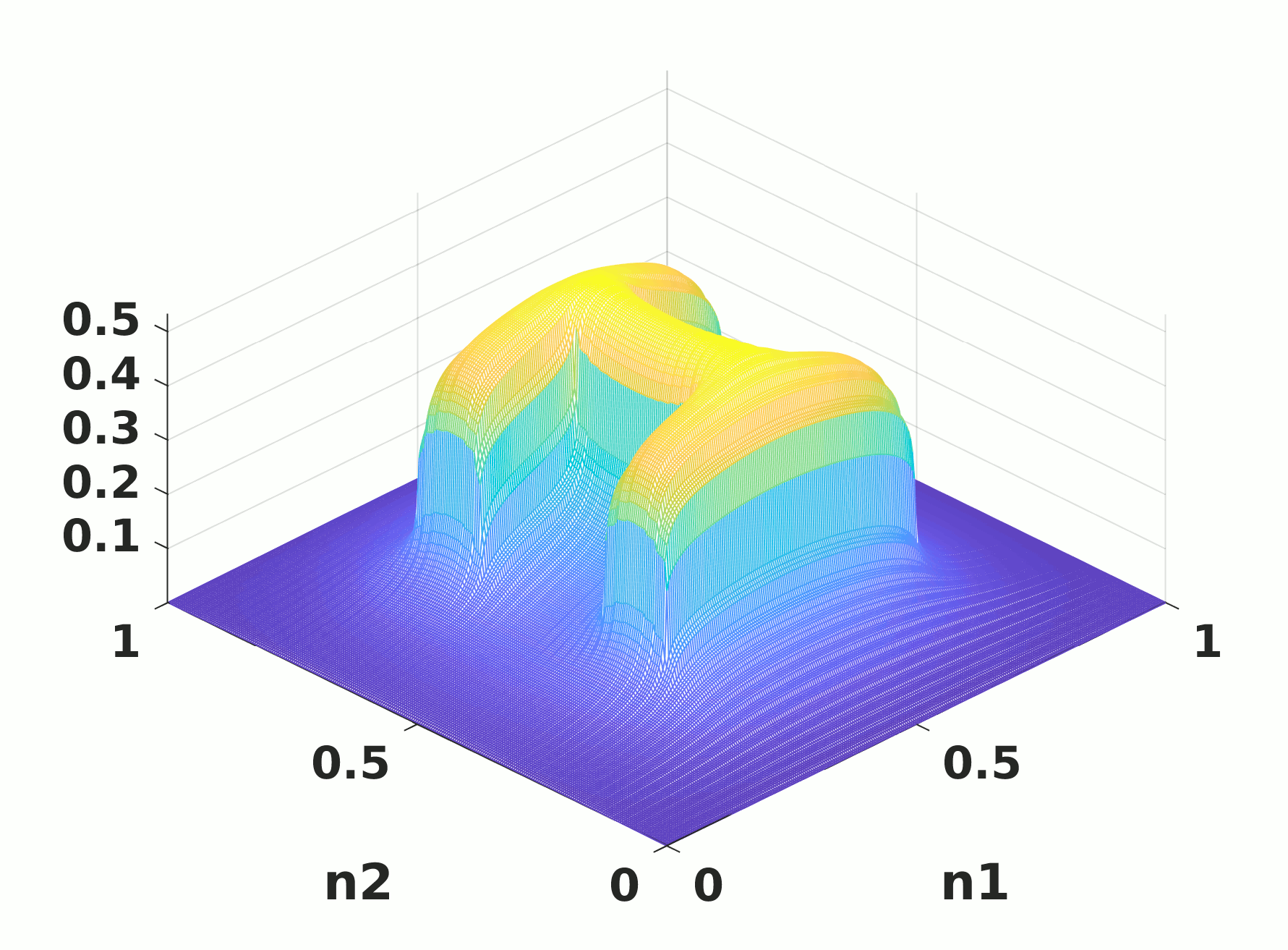}}
 	\caption[Test 2: Lösungen $u$ mit $\alpha = 1/10$]{Impact of small regularization parameter $\gamma = 0.01$ 
 	on solutions $\y$ for $H$-type design function $\yom$ and different values for $\alpha$. }
 	\label{Fig2D:y_h_gam}
 \end{figure}

 \subsection{Numerical tests for 3D case}
\label{ssec:numerics_tensor_3D}
 
In this section the numerical results for the 3D case are presented. 
We first  consider the 
computational time and then discuss the storage complexity of the algorithm.  Further, 
we present 3D solutions for equations (\ref{eqn:Lagrange_cont}) and (\ref{eqn:State}), 
that is the results of the
numerical simulations for the optimal control $\uu$ and the state $\y$, respectively. 
In these calculations, we use a grid size   of $n = 127$ grid points in 
each dimension, the regularization parameter $\gamma = 1$, and different values for $\alpha = 1, 1/2, 1/10$. 
As stated previously, we use the operator 
$
 (\Delta^\alpha + \Delta^{-\alpha})^{-1}
$
in a low-rank format as preconditioning operator, where $\Delta$ is the classic negative Laplace 
operator (see section \ref{ssec:Precond} for details). 

\subsubsection{Complexity results}
\label{ssec:complexity_3d}
% 
% \vspace*{4mm}
% \hspace*{4mm} \textbf{Time}\\
First, we investigate the time the low-rank pcg scheme needs to solve equation (\ref{eqn:Lagrange_cont}), 
$$(A^\alpha + A^{-\alpha})\uu = \yom $$ for the $H$-type right hand side $\yom$.

Table \ref{Tab:times_3d} shows the resulting computational times (in seconds) the pcg algorithm needs in total 
to solve (\ref{eqn:Lagrange_cont}), considering different numbers of grid points, different values for  $\alpha$ 
and aforementioned coefficient functions (\ref{fun:a1}) -- (\ref{fun:a3}).

Figure \ref{Fig:IterTimes} represents the corresponding time per iteration in the presented test, resulting 
from tables \ref{Tab:times_3d} and \ref{Tab:iter_3d}. The presented data confirms the computational 
complexity of $O(RSn^2)$ (see section \ref{ssec:StiffnessMatrixinLR}) and therefore proves that the used 
low-rank structures within the pcg scheme contribute to circumventing the course of dimensionality efficiently.
 \begin{table}[H]
\begin{center}%
{\footnotesize
\begin{tabular}
[c]{|c|c|c|c| }%
\hline 
  grid points   &    $\alpha=1$ &  $\alpha=1/2$ & $\alpha=1/10$ \\
  \hline
 64    & 28.1 & 16.0 & 2.53   \\
  \hline 
 128  & 92.8 & 43.7 & 7.15 \\
     \hline 
 256  &  318.0 & 125.0 & 22.5 \\
 \hline
512 &1180.0 & 512.0 & 66.8 \\
     \hline 
  \end{tabular}
\caption{\small Times in seconds needed by the rank-structured solver to 
solve $(A^\alpha+A^{-\alpha})\uu=\yom$ 
with different numbers of grid points and different values $\alpha$. 
The $H$-type design function is considered.}
\label{Tab:times_3d}
}
\end{center}
\end{table}
\vspace*{-10mm}
\begin{figure}[H]
\centering
 	\includegraphics[width=0.4\textwidth]{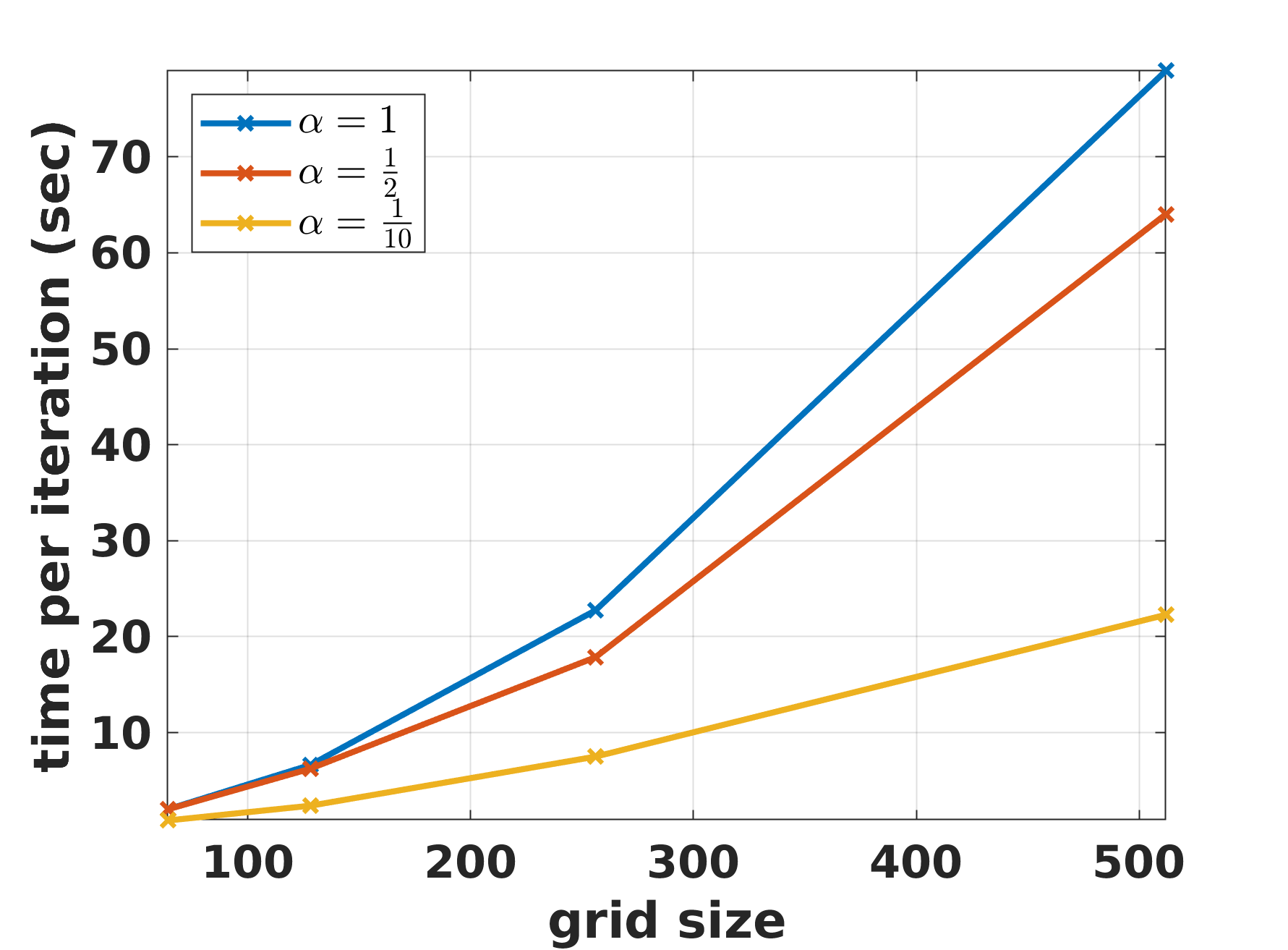}
 	\caption{Times per iterations of the pcg algorithm for different grid sizes and different 
 	values of $\alpha$ 
 	considering the $H$-type design function. }
 	\label{Fig:IterTimes}
 \end{figure}  

Table \ref{Tab:iter_3d} shows the number of iterations the algorithm needs to solve 
equation (\ref{eqn:Lagrange_cont}) 
for different numbers of grid points as well as different values for $\alpha$. 
Analogously to the 2D case, 
the data varifies that the algorithm provides a solution scheme for the investigated problem 
class whose number of iterations is independent of the number of grid points.\\
\vspace*{-5mm}
\begin{table}[H]
\begin{center}%
{\footnotesize
\begin{tabular}
[c]{|c|c|c|c| }%
\hline 
  grid points   &    $\alpha=1$ &  $\alpha=1/2$ & $\alpha=1/10$ \\
  \hline
 64     & 14 & 8 & 3\\ 
  \hline 
 128  & 14 & 7 & 3 \\ 
     \hline 
 256    & 14 & 7 & 3\\ 
 \hline
512 &  15 & 8 & 3\\
     \hline 
  \end{tabular}
\caption{\small Numbers of iterations for the low-rank solver for solution of $(A^\alpha+A^{-\alpha})\uu=\yom$ 
with different numbers of grid points and different values of $\alpha$. 
The $H$-type design function is considered.}
\label{Tab:iter_3d}
}
\end{center}
\end{table}

\subsubsection{Effects of anisotropic preconditioning}

In our tests, we use coefficient functions $a_1(x_1), a_2(x_2)$, and $a_3(x_3)$ as defined 
in \eqref{fun:a1} -- \eqref{fun:a3}. 
As stated in Theorem \ref{lem:cond_bound_var}, the condition number of the preconditioned operator is estimated 
with the help of majorants and minorants of the coefficient functions involved in the problem. In our numerical test, 
the minorant corresponding to the coefficient function $a_2(x_2)$ tends to be zero, which then might lead 
to a unfavorable condition tending to infinity (see condition estimate in section \ref{ssec:Precond}). 
However, our numerical tests do not suffer from this fact, as the inverse term in operator $A_u$ regularizes 
the spectrum. \\
Nevertheless, we define a modified coefficient $$\tilde{a}_2 := a_2 + 0.1$$
in order to change our numerical test slightly (which also circumvents any problems that might arise due to 
a bad condition) and to see the effects on the time complexity for solving equation 
\eqref{eqn:Lagrange_cont}. 
% \begin{equation}
%  (A^\alpha + A^{-\alpha})\uu = \yom.
%   \end{equation}
We also test the effect of using the anisotropic Laplacian \eqref{eqn:PrecClassA}, 
\[
{B}_0= \beta {B}^{-\alpha} + \tfrac{\gamma}{\beta}{B}^{\alpha}, \]
where $B$ is the discretization of $
{\cal B}  :=  -\sum_{\ell=1}^d \frac{\partial}{d x_\ell} b_0^\ell \frac{\partial}{d x_\ell} $,
used as preconditioning operator instead of the classic Laplacian.  During the test, we use coefficient 
functions $a_1,\tilde{a}_2, \text{and} \ a_3$ and choose 
\begin{equation} \label{eqn:modcoeff} 
{b}_0^{\ell}=
\frac{1}{2}(\max a_\ell(x) + \min a_\ell(x)  ) ,\quad \ell=1,2,3, 
\end{equation} 
as scaling constant coefficients. \\
Table \ref{Tab:iternums_precond_3d} shows the results of both tests: 
The left table represents the number of iterations needed by the pcg scheme for 
solving \eqref{eqn:Lagrange_cont} when using the same setting as in section 
\ref{ssec:complexity_3d} but replacing $a_2$ by $\tilde{a}_2$. Note that the number of iterations 
can be reduced for $\alpha = 1$ and $\alpha = 1/2$. \\
The right table shows the number of iterations for the same setting as in section \ref{ssec:complexity_3d} 
but replacing $a_2$ by $\tilde{a}_2$ \textit{and} using the anisotropic Laplacian \eqref{eqn:PrecClassA} 
with coefficients defined in \eqref{eqn:modcoeff} as preconditioner. Again the number of iterations 
needed to solve \eqref{eqn:Lagrange_cont} decreases for $\alpha = 1$ and $\alpha = 1/2$. \\
In both cases, we notice an improvement concerning the time complexity and still observe a grid 
independent number of iterations in order to solve the problem.  
\begin{table}[H]
\begin{minipage}{0.49\linewidth}
\centering 
{\footnotesize
\begin{tabular}
[c]{|c|c|c|c| }%
\hline 
  grid points   &    $\alpha=1$ &  $\alpha=1/2$ & $\alpha=1/10$ \\
  \hline
 64     & 11 & 7 & 3\\ 
  \hline 
 128  & 11 & 6 & 3 \\ 
     \hline 
 256    & 11 & 6 & 3\\ 
 \hline
512 &  11 & 6 & 3\\
     \hline 
  \end{tabular}
  }
\end{minipage}
\begin{minipage}{0.49\linewidth}
\centering
{\footnotesize
\begin{tabular}
[c]{|c|c|c|c| }%
\hline 
  grid points   &    $\alpha=1$ &  $\alpha=1/2$ & $\alpha=1/10$ \\
  \hline
 64     & 9 & 6 & 3\\ 
  \hline 
 128  & 8 & 6 & 3 \\ 
     \hline 
 256    & 8 & 6 & 3\\ 
 \hline
512 &  9 & 5 & 3\\
     \hline 
  \end{tabular}
 }
\end{minipage}
\caption{\small Numbers of iterations for the low-rank solver for soultion of $(A^\alpha + A^{-\alpha})\uu = \yom$. 
\textit{Left:} coefficient functions $a_1,\tilde{a}_2, \text{and} \ a_3$, classic Laplacian as preconditioner. 
\textit{Right:} coefficient functions $a_1, \tilde{a}_2, \text{and} \ a_3$, 
anisotropic Laplacian as preconditioner. The $H$-type design function is considered}
\label{Tab:iternums_precond_3d}
\end{table}

 \begin{remark}
From above tables concerning the time complexity, we observe that equation 
(\ref{eqn:Lagrange_cont}) can be solved fastest for small $\alpha \rightarrow 0$. 
This effect is due to the fact that for small values for $\alpha$, both the operators 
$A^\alpha$ and $A^{-\alpha}$  approach the Identity matrix \textbf{\MakeUppercase{\romannumeral 1}}, 
so that ultimately only an equation similar to \textbf{\MakeUppercase{\romannumeral 1}}$\uu = \y_\Omega$ 
has to be solved. 
\end{remark}

 %\vspace*{1mm} 
 %\textbf{Storage}\\
 %In order to further 

 It is worth to note that in 3D case the main bottleneck for the numerical treatment of nonlocal 
 operators is the dramatic storage growth for the corresponding fully populated 
 $n^3\times n^3$ stiffness matrices  as the volume size $n^3$ of the discretization grid increases. 
 To point out further advantages of our low-rank scheme compared to a full format algorithm, 
 we investigate the  costs for storing the operator $A$ of type (\ref{eqn_low_rank_lap})
 %$$A = A_1 \otimes I_2 \otimes I_3 + I_1 \otimes A_2 \otimes I_3 + I_1 \otimes I_2 \otimes A_3$$ 
 as well as for computing and storing the fractional operators $A^\alpha$ and $A^{-\alpha}$ in both 
 a low-rank format and a full format.  
 Figure \ref{Fig:StorageComplexity} shows the respective required storage to store the discrete 
 operator $A$ for both tests.  
 
   \begin{figure}[h]
\centering
 	\includegraphics[width=0.4\textwidth]{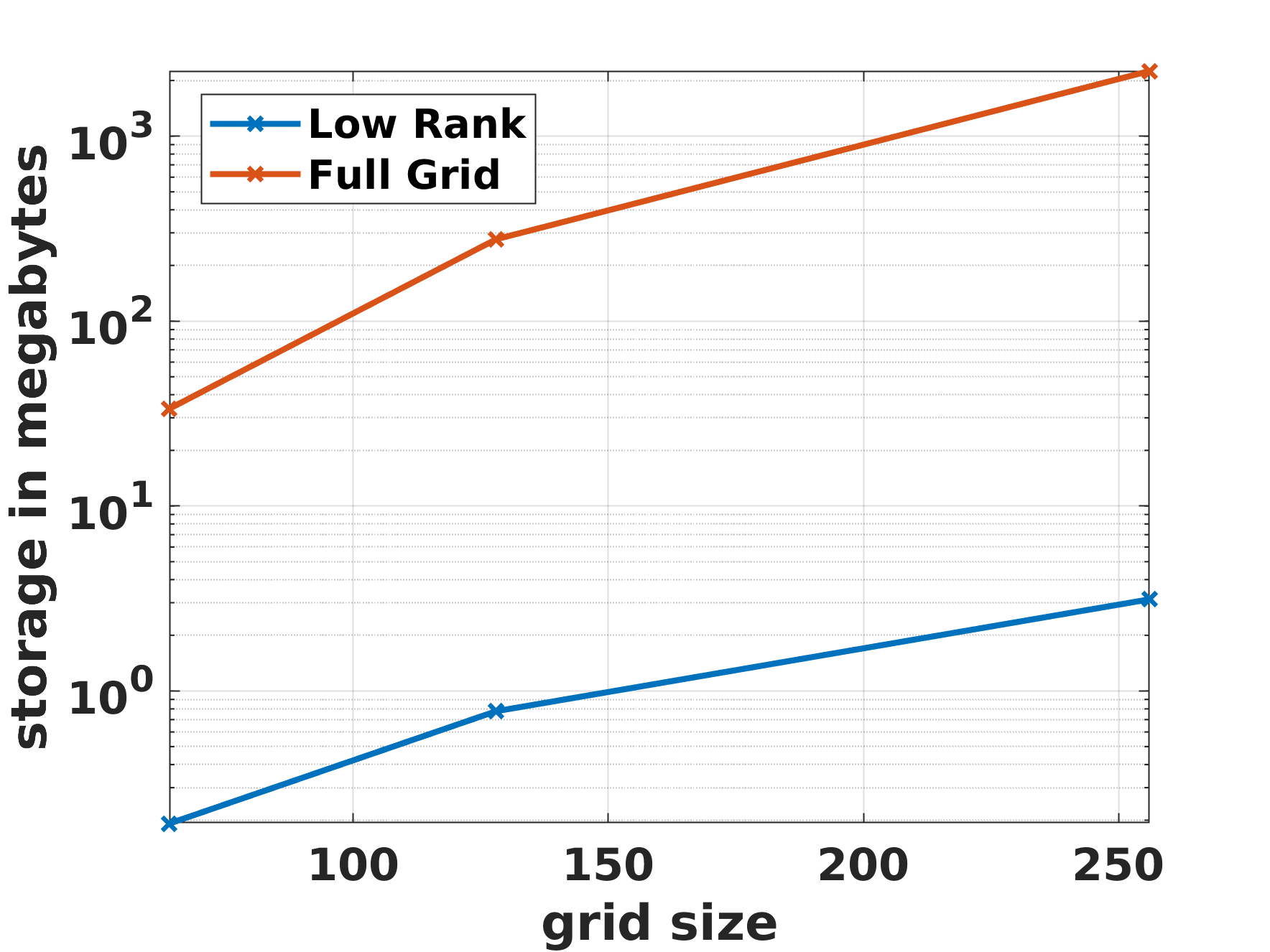}
 	\caption{Storage complexity needed to store operator $A$ when using a full format scheme and when 
 	using the low-rank scheme.}
 	\label{Fig:StorageComplexity}
 \end{figure}  
 If we worked with the full format tensor in the given problem, we would have to store a 
 tensor (matrix) of size $n^d \times n^d$ which means we would end up with a an polynomial 
 storage scaling in the number of grid points, which can only be handled by saving the operator 
 in a sparse format. However, in order to compute $A^\alpha$ and its inverse $A^{-\alpha}$, 
 the full format tensor is needed. As a consequence, their computations in our test exceed 
 the storage capacity of the used laptop (16GB RAM) for $n \ge 32$ grid points in each dimension. 
 As a result, the problem cannot be solved in a full format scheme with large grids.
  
However, in low-rank format we follow the computation and storage scheme presented in chapter 
\ref{ssec:StiffnessMatrixinLR}, that is we compute a factorized form (\ref{eqn:DiagFGen2D}) 
of $A^\alpha$ and $A^{-\alpha}$ with the help of the eigenvalue decomposition. 
%This has a computational cost of order $O(dn^2)$. \textit{Hier bild explizite zeiten} 
Thereafter, we compute $\mathcal{F}(\Lambda)$ in a canonical format (\ref{eqn:CanFormat}) that has 
a computational cost of order $O(n^d)$, but only has to be computed once before the algorithm starts. 
Finally, we end up with a storage cost for the canonical format of order $O(dRn)$ which is only linear 
in the number of grid points.

All in all we stress that the given problem can only be solved with controllable precision 
by efficiently using low-rank structures of the involved operators.

\subsubsection{Solutions for optimal control}
\label{sssec:3DU}

Figures \ref{Fig3D:u_square} and \ref{Fig3D:u_h} show the solution for the control $\uu$ of equation (\ref{eqn:Lagrange_cont}), 
$$ A_u \uu = \yom,$$ that is computed by the pcg scheme
%$$(A^\alpha + A^{-\alpha})\uu = \yom$$ 
for two different right hand sides $\yom$. 
We consider a grid size of $n = 127$ grid points in each dimension, regularization parameter $\gamma = 1$, 
and different values for $\alpha = 1, 1/2, 1/10$. \\

Each figure shows slice planes for the volumetric (tensor) data $\uu$, where the values in $\uu$ determine 
the contour colors. We choose three slice planes, where each of them is orthogonal to one dimension. 
Every slice plane could also be depicted analogously to the 2D figures in section \ref{ssec:numerics_tensor_2D}. 

% square: ex. 1
% hshape: ex. 2
% cshape: ex. 3
\vspace*{-8mm}
\begin{figure}[H]
 	\subfigure[$\alpha = 1$]{\includegraphics[width=0.32\textwidth]{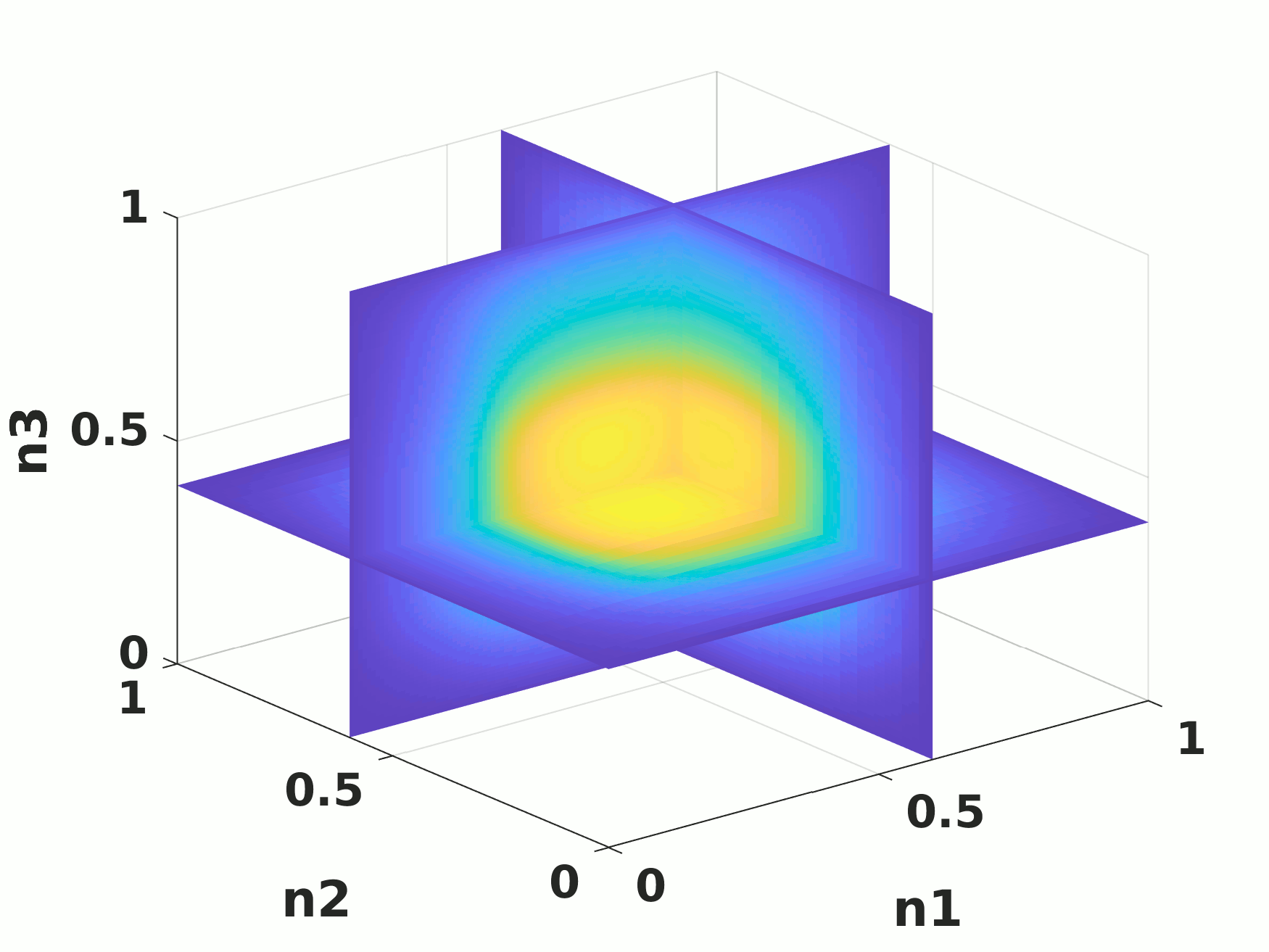}}
 	\subfigure[$\alpha = 1/2$]{\includegraphics[width=0.32\textwidth]{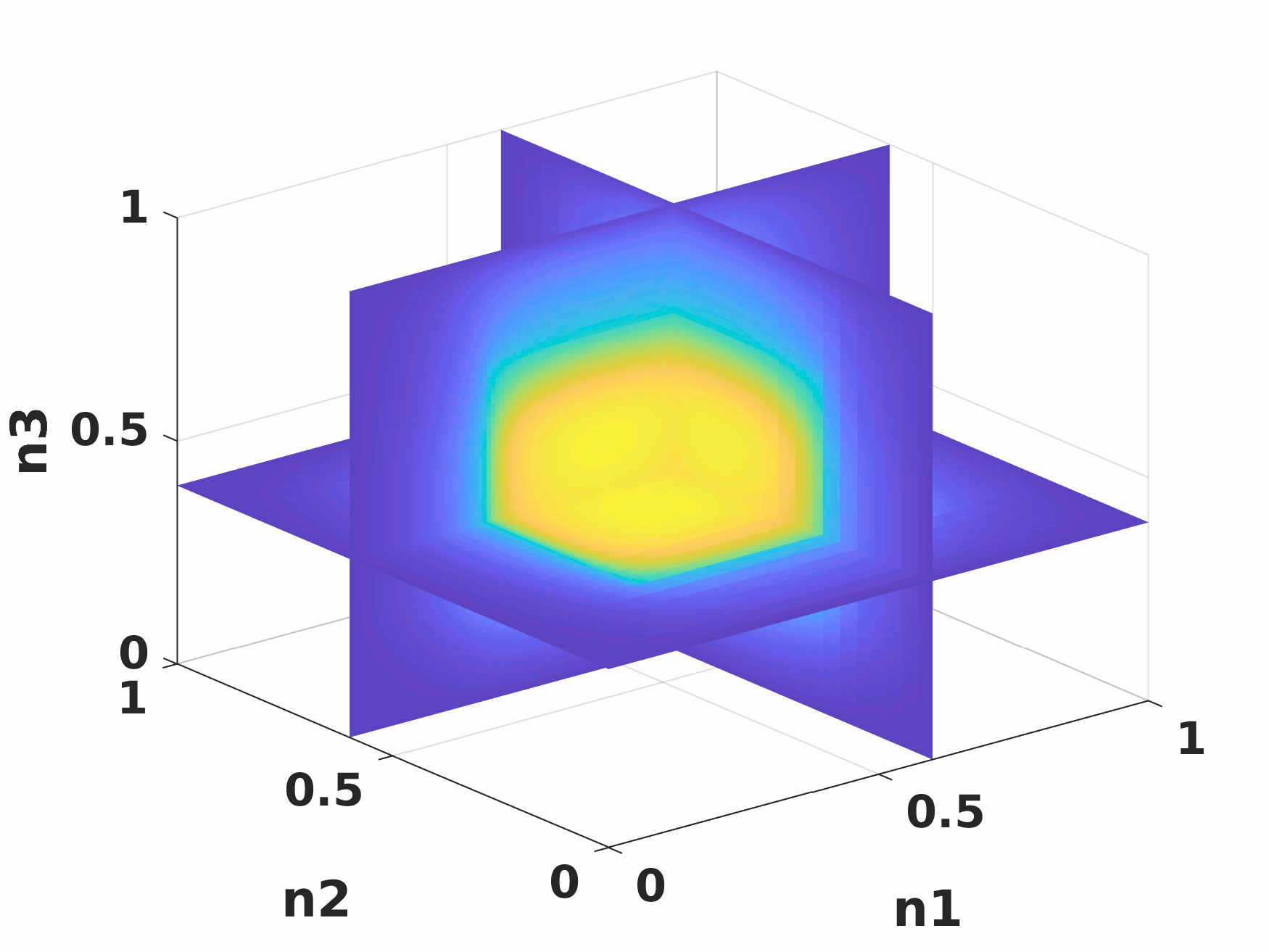}}
 	\subfigure[$\alpha = 1/10$]{\includegraphics[width=0.32\textwidth]{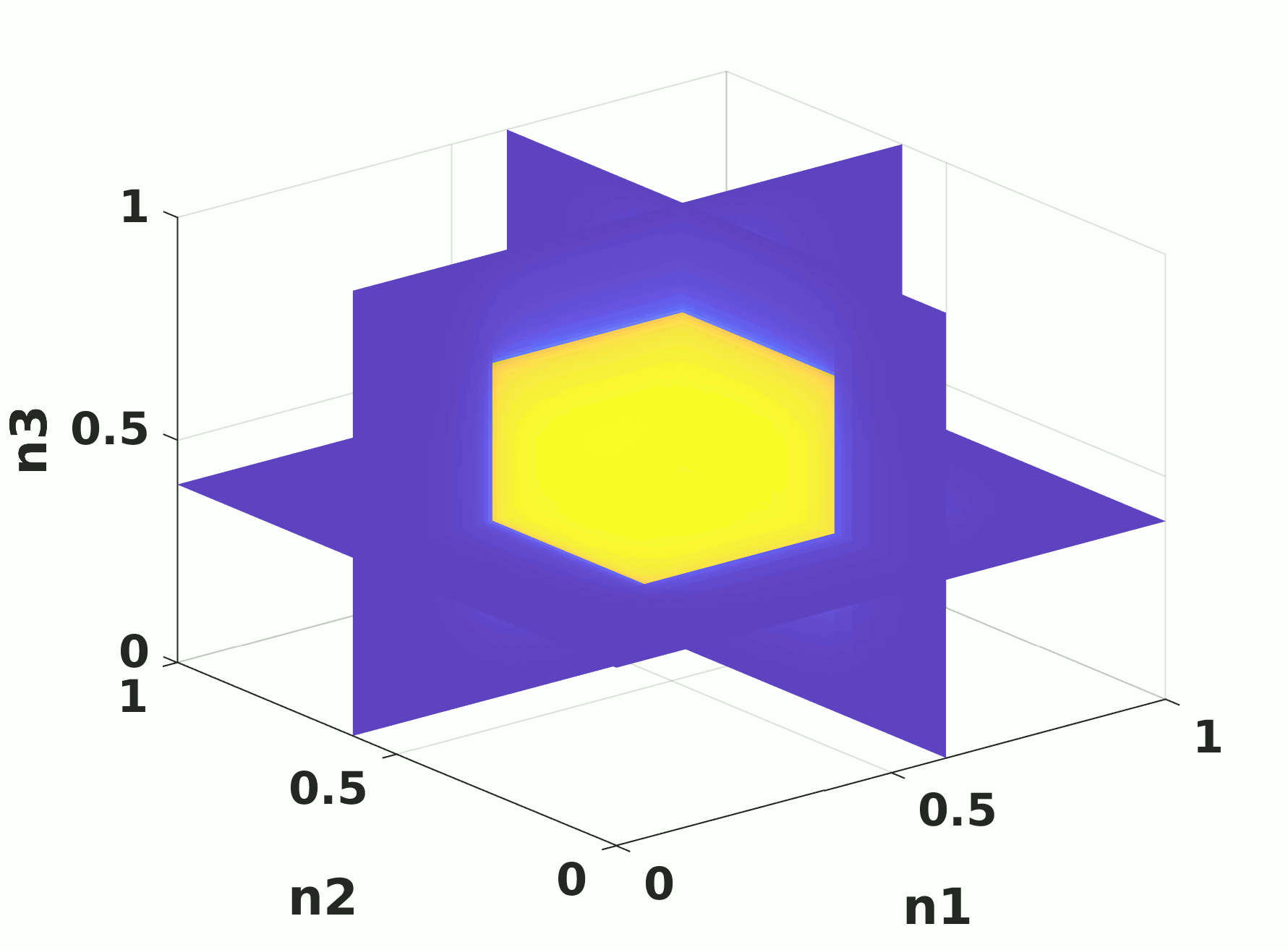}}
 	\caption[Test 2: Lösungen $u$ mit $\alpha = 1/10$]{\small Impact of fractional exponent 
 	$\alpha$ on solutions $\uu$ of (\ref{eqn:Lagrange_cont}) for box-type right hand side $\yom$ and $n = 127$ grid points in each dimension. }
 	\label{Fig3D:u_square}
 \end{figure} 
 
 \begin{figure}[H]
 	\subfigure[$\alpha = 1$]{\includegraphics[width=0.32\textwidth]{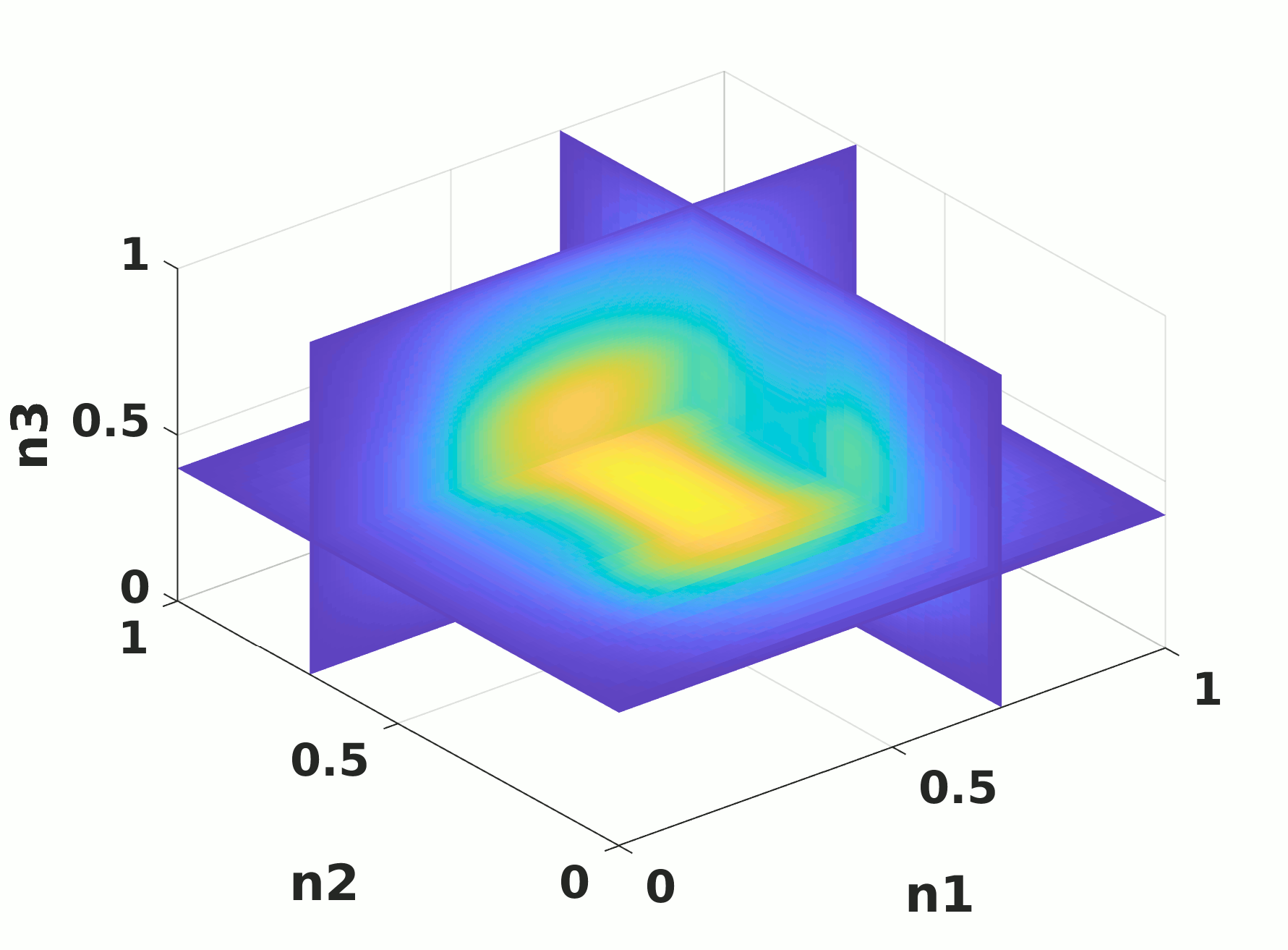}}
 	\subfigure[$\alpha = 1/2$]{\includegraphics[width=0.32\textwidth]{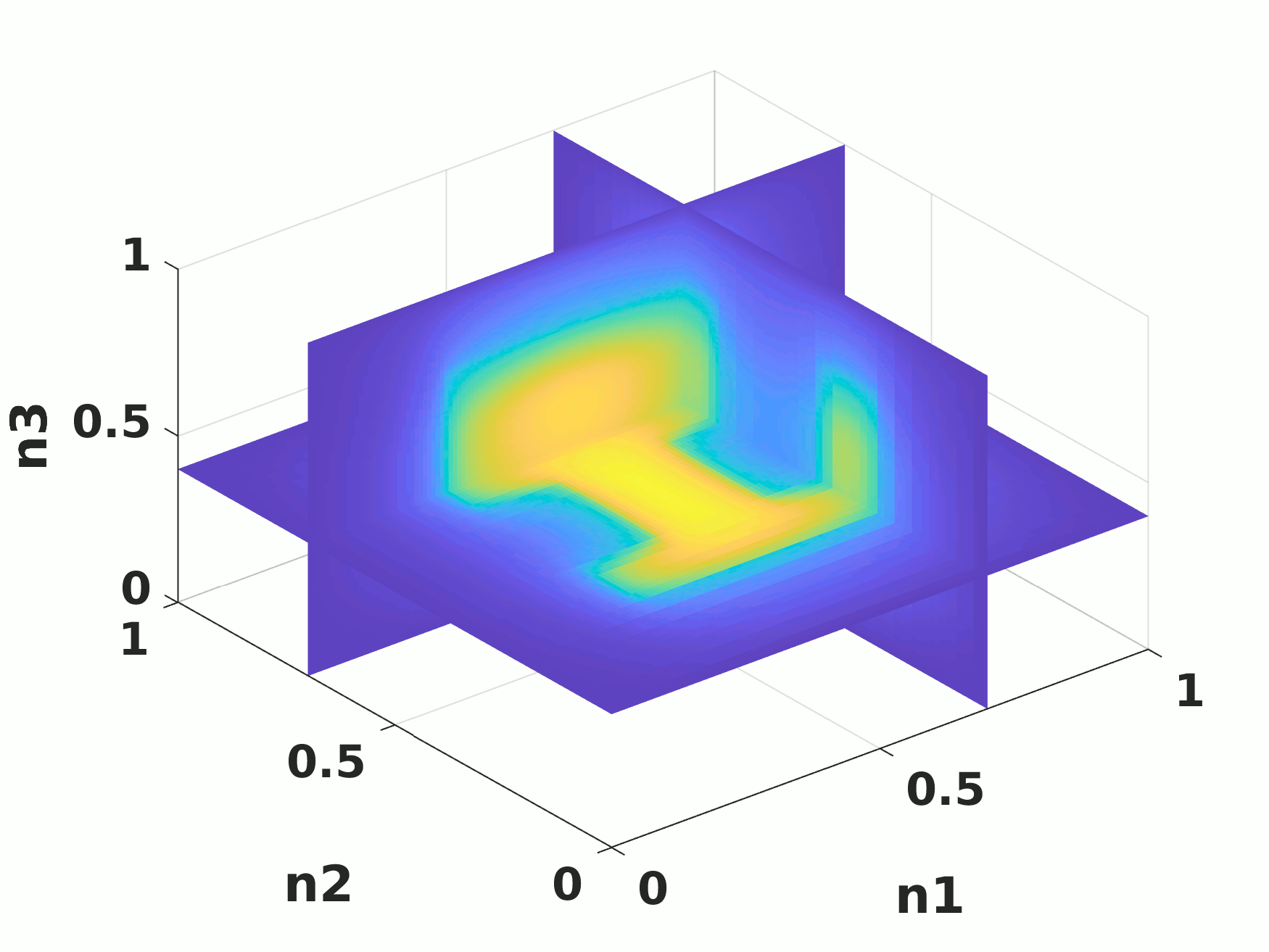}}
 	\subfigure[$\alpha = 1/10$]{\includegraphics[width=0.32\textwidth]{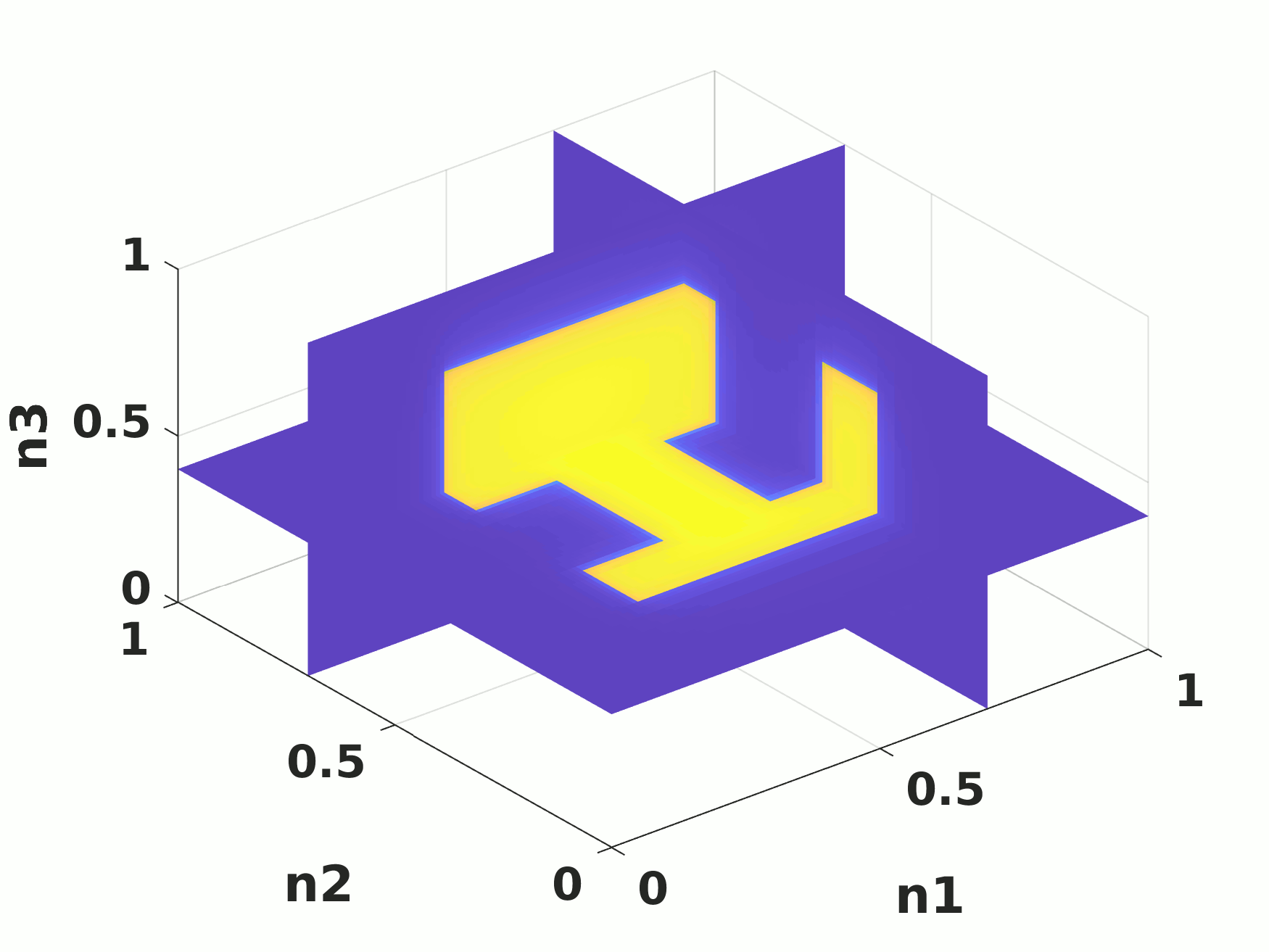}}
 	\caption[Test 2: Lösungen $u$ mit $\alpha = 1/10$]{\small Impact of fractional exponent $\alpha$ 
 	on solutions $\uu$ of (\ref{eqn:Lagrange_cont}) for $H$-type right hand side $\yom$ and $n = 127$ grid points in each dimension.}
 	\label{Fig3D:u_h}
 \end{figure} 

 Figure \ref{Fig3D:Volume_H} visualizes in volumetric style the same data as in figure \ref{Fig3D:u_h}
 for the case of $H$-typed right hand side $\yom$.
 
\begin{figure}[H]
\includegraphics[width=5.4cm]{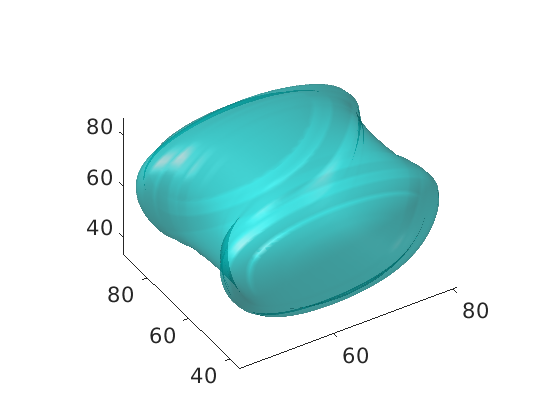}
 \includegraphics[width=5.4cm]{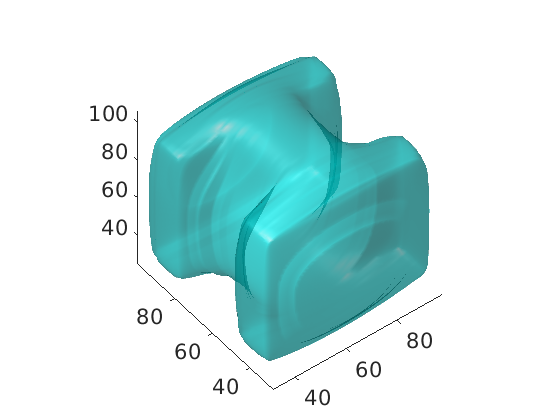}
 \includegraphics[width=5.4cm]{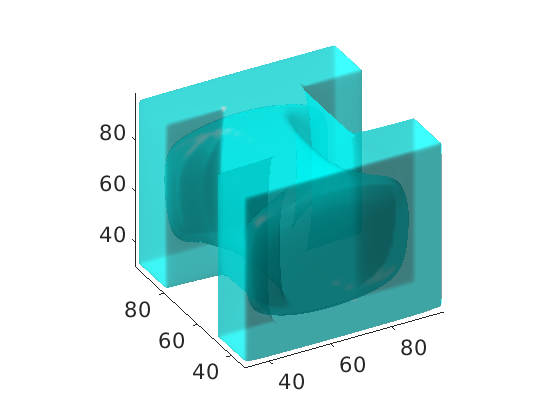}
  \caption{\small Volumetric visualization for the case of $H$-type: Impact of fractional exponent 
  $\alpha$ on solutions $\uu$ 
  of (\ref{eqn:Lagrange_cont})   for $H$-typed right hand side $\yom$, see Fig. \ref{Fig3D:u_h}.}
 \label{Fig3D:Volume_H}
\end{figure}

\subsubsection{Solutions for state variable}

Figures \ref{Fig3D:y_square} and \ref{Fig3D:y_h} represent the solution for the state $\y$ 
of (\ref{eqn:State}), 
 \[
  A^{-\alpha}\uu = \y,
 \]
where $\uu$ is the solution of (\ref{eqn:Lagrange_cont}) presented in section \ref{sssec:3DU}. 
Again we consider the grid size $n = 127$ and compare the effects of different fractional exponents $\alpha$.

\begin{figure}[H]
 	\subfigure[$\alpha = 1$]{\includegraphics[width=0.32\textwidth]{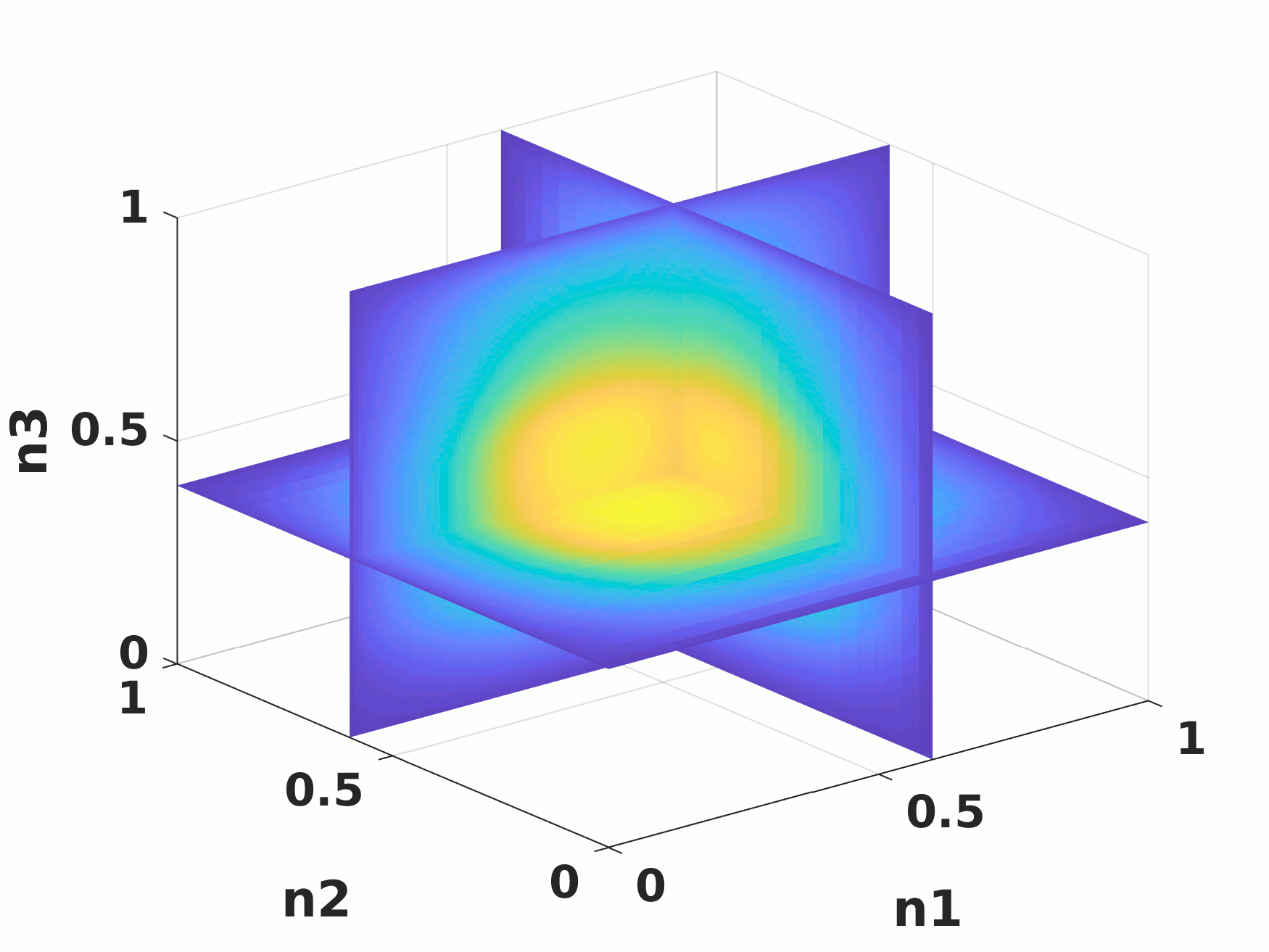}}
 	\subfigure[$\alpha = 1/2$]{\includegraphics[width=0.32\textwidth]{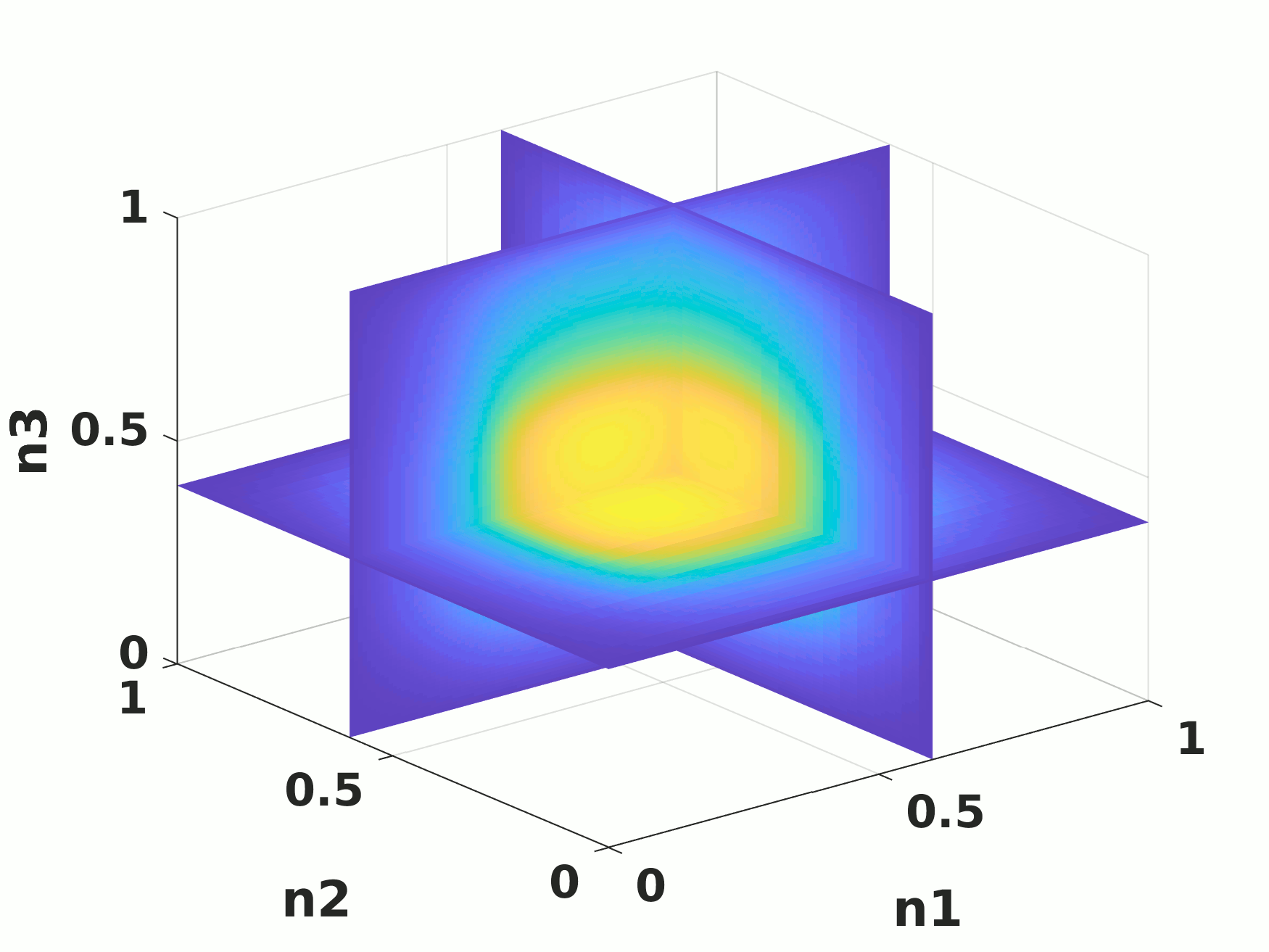}}
 	\subfigure[$\alpha = 1/10$]{\includegraphics[width=0.32\textwidth]{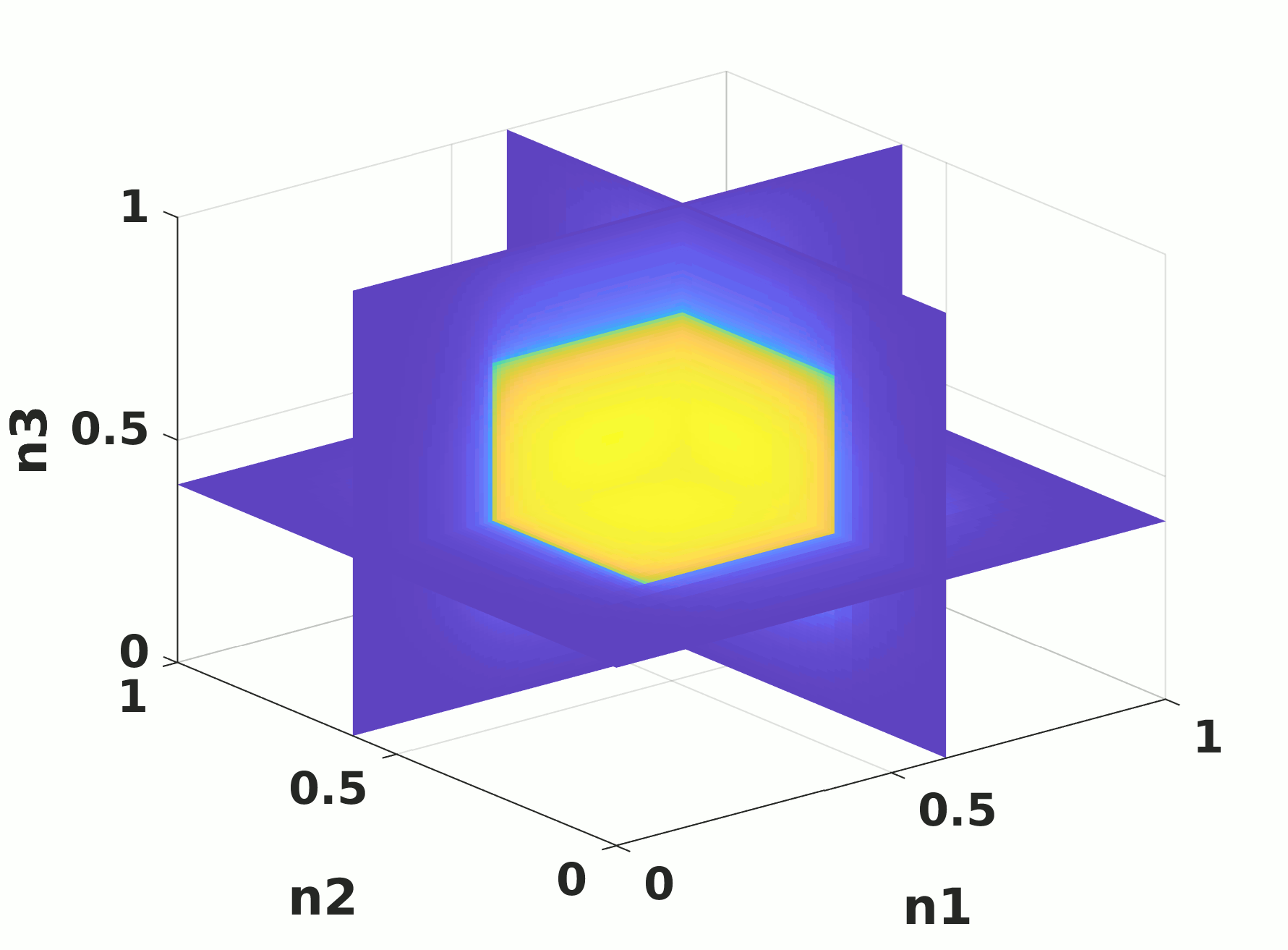}}
 	\caption[Test 2: Lösungen $\u$ mit $\alpha = 1/10$]{Impact of fractional exponent $\alpha$ on 
 	solutions $\y$ of (\ref{eqn:State}) for box-type right hand side $\yom$, $\gamma = 1$ and $n = 127$ grid points in each dimension. }
 	\label{Fig3D:y_square}
 \end{figure} 
 
 \begin{figure}[H]
 	\subfigure[Square]{\includegraphics[width=0.32\textwidth]{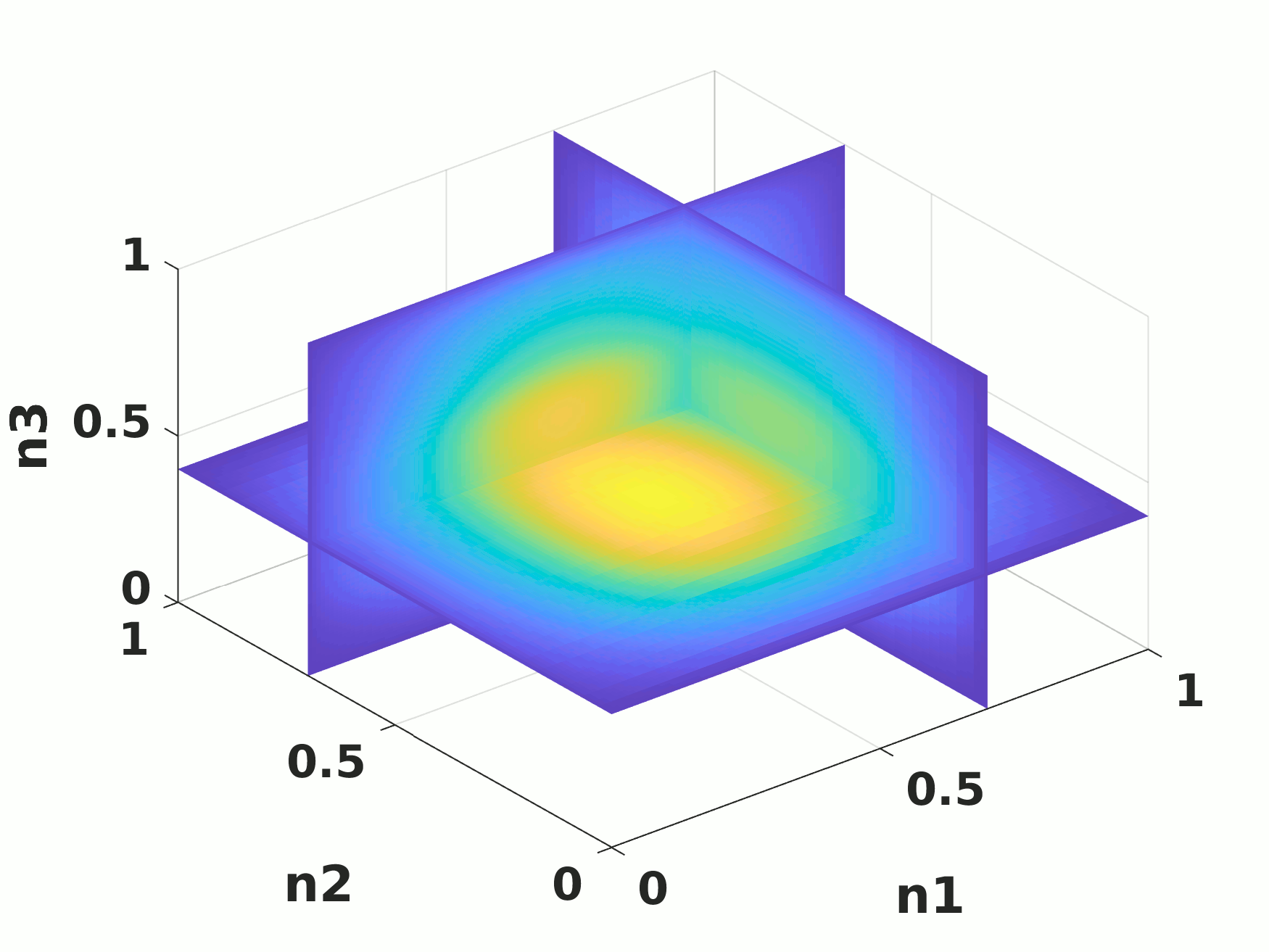}}
 	\subfigure[U-shaped]{\includegraphics[width=0.32\textwidth]{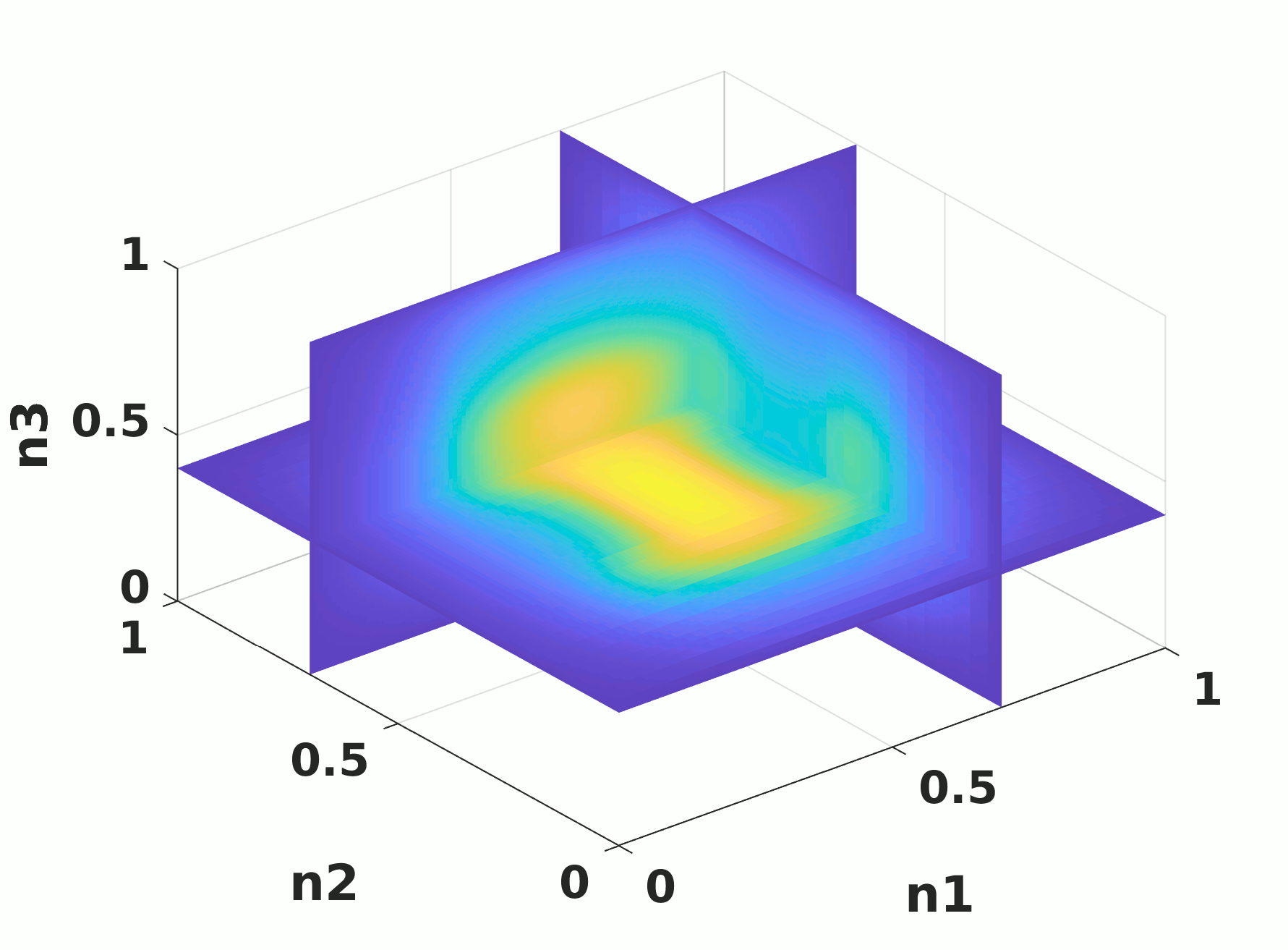}}
 	\subfigure[$H$-typed]{\includegraphics[width=0.32\textwidth]{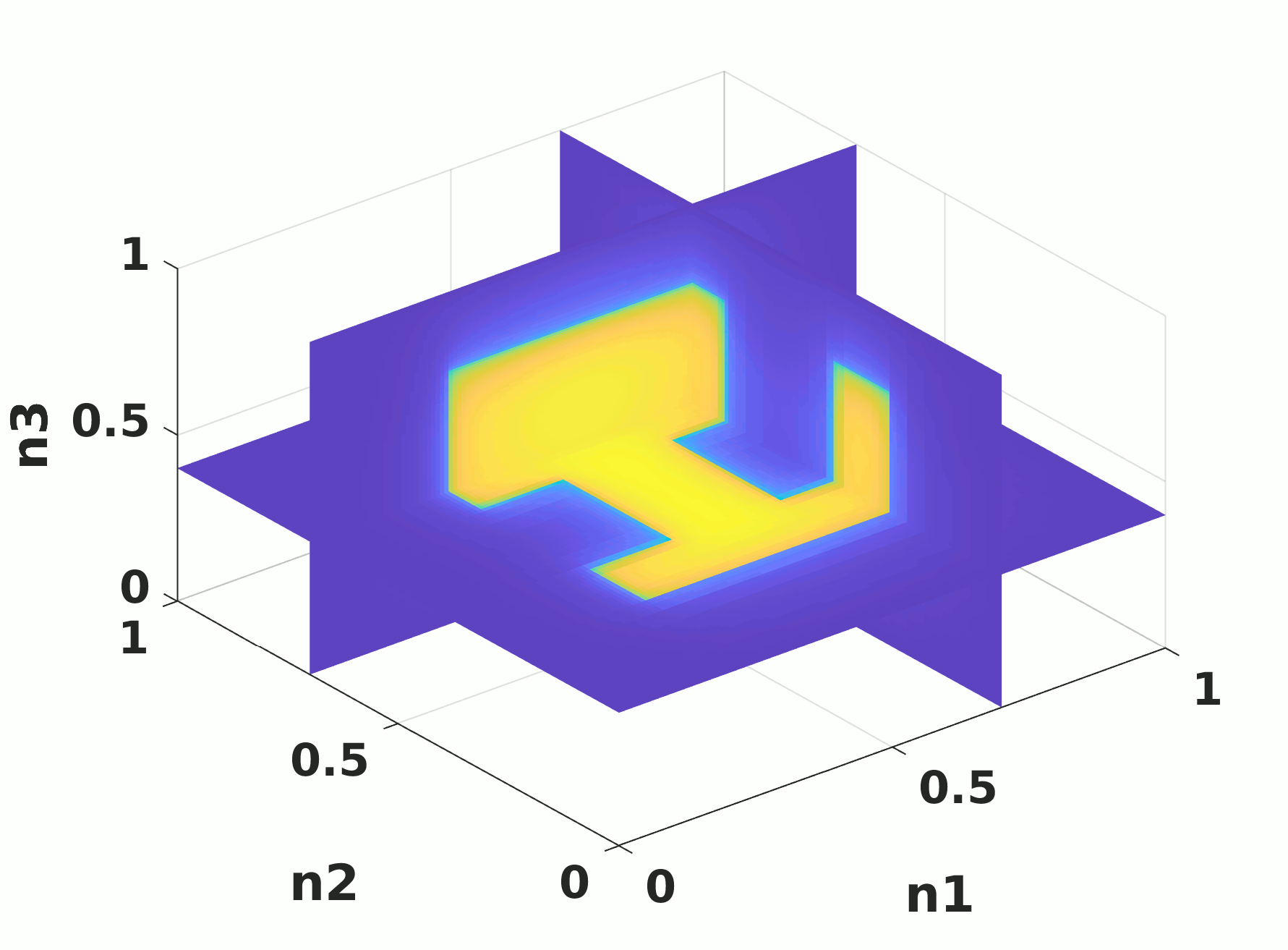}}
 	\caption[Test 2: Lösungen $u$ mit $\alpha = 1/10$]{Impact of fractional exponent $\alpha$ on 
 	solutions $\y$ of (\ref{eqn:State}) for $H$-typed right hand side $\yom$, $\gamma = 1$ and $n = 127$ grid points in each dimension. }
 	\label{Fig3D:y_h}
 \end{figure}

%%%%%%%%%%%%%%%%%%%%%%%%%%%%%%%%%%%%%%%%%%%
%%%%%%%%%%%%%%%%%%%%%%%%%%%%%%%%%%%%%%%%%%%

\section{Conclusions}
We have introduced and analyzed a tensor numerical pcg scheme  with adaptive rank truncation 
for the solution of optimal control problems 
constrained by  fractional 2D and 3D elliptic Laplacian-type operators with variable 
separable coefficients,  which essentially generalizes the results in \cite{HKKS:18}. 
To that end, we first have provided the  theoretical solution setting by 
exploiting the separable structure of the involved functions of an elliptic operator 
and by deriving the corresponding finite difference discretization scheme on  
$n\times n \times n$ Cartesian grids. 
With the help of the eigenvalue decomposition of the 
one-dimensional differential operators and using the efficient Tucker-to-canonical 
tensor approximation techniques, we have managed to present the  fully populated stiffness  
matrix in factorized low-rank format. 
For preconditioning, we have extended  the low-rank structures for the classic Laplace operator 
developed in \cite{HKKS:18} to the case of anisotropic Laplacian in 3D. 

As a result, we achieved a dimensionally independent quadratic complexity scaling in the number of 
univariate grid points, $O(RSn^2)$, for the (factorized) matrix-vector multiplication that also
represents an upper bound for the complexity of the PCG solution scheme with rank truncation as a whole.
 Here $R$ and $S$ are, respectively, the (low) Kronecker ranks of the nonlocal solution operator (matrix) and 
the iterated unknown solution vector. 

In our numerical study, we have verified the efficiency of our solution scheme over 
rank structured ``data manifold'' by comparison 
to the Matlab intern pcg routine and justified the existence of low-rank representations of 
the involved operators by investigating the decay of the respective singular values. 
Furthermore, in our numerical tests we have verified the computational cost of $O(RSn^2)$ for $d=2, 3$ 
and have pointed out that the required memory capacity for solving the problem for large grids 
can only be handled by using the proposed low-rank method. 

The presented approach can also be applied to the simpler case of control problems  constraint
via both classic and fractional Laplace operator. 
In order to reduce the numerical complexity further, the application of quantized tensor train 
formats to the involved operators and vectors may be examined. What is more, the effects of considering 
fractional operators with general rank-$R$ separable coefficients may be investigated and analyzed.

 \section*{Appendix 1: Short sketch on tensor numerical methods }
\label{sec:low-rank_tensor}

Recent tensor numerical methods emerged   as bridging  of the basic tensor decompositions 
and algorithms of the multilinear algebra with the rigorous results in approximation theory on the low-rank
representation of the multivariate functions an operators \cite{GHK:05,HaKhtens:04I}. 
The latter results on tensor-product approximation to multi-dimensional nonlocal operators have been   
first originated in the framework of the low-rank ${\cal H}$-matrix techniques  \cite{Hack_Book:09}.

The basic tensor decompositions used in multilinear algebra for the low-rank representation 
of the multidimensional tensors 
are the canonical \cite{Hitch:1927} and Tucker \cite{Tuck:1966} tensor formats. 

For a tensor of order $d$ given in a full size format  
\begin{equation*}
 \label{Tensor_def}
{\bf T}=[t_{i_1,\ldots,i_d}]  \in 
\mathbb{R}^{n_1 \times \ldots \times n_d}\;
\mbox{  with } \quad i_\ell\in I_\ell:=\{1,\ldots,n_\ell\}.
\end{equation*}
all operations scale exponentially with the dimension size, as $O(n^d)$ (assuming $n_\ell =n$).
 The so-called ``curse of dimensionality'' can be reduced or eliminated when the tensor is 
 given in a rank-structured representation. 
 A tensor in the $R$-term canonical format is defined by a sum of rank-$1$ tensors 
 \begin{equation}\label{eqn:CP_form}
   {\bf T} = {\sum}_{k =1}^{R} \xi_k
   {\bf u}_k^{(1)}  \otimes \ldots \otimes {\bf u}_k^{(d)},  \quad  \xi_k \in \mathbb{R},
\end{equation}
where ${\bf u}_k^{(\ell)}\in \mathbb{R}^{n_\ell}$ are normalized vectors, 
and $R$ is the canonical rank. The storage cost of this
parametrization is bounded by  $d R n$. However, for $d\geq 3$, there is lack of stable 
algorithms to compute the 
canonical low-rank representation of a general tensor ${\bf T}$, that is, with
the minimal number $R$ in representation (\ref{eqn:CP_form}), and the respective
decomposition with the polynomial cost in $d$, i.e., the computation of the
canonical decomposition is in general an $N$-$P$ hard problem.

The Tucker tensor format is suitable for stable numerical decompositions with a fixed
truncation threshold.
We say that the tensor ${\bf T} $ is represented in the rank-$\bf r$ orthogonal Tucker format 
with the rank parameter ${\bf r}=(r_1,\ldots,r_d)$ if 
\begin{equation*}
\label{eqn:Tucker_form}
  {\bf T}  =\sum\limits_{\nu_1 =1}^{r_1}\ldots
\sum\limits^{r_d}_{{\nu_d}=1}  \beta_{\nu_1, \ldots ,\nu_d}
\,  {\bf v}^{(1)}_{\nu_1} \otimes  
{\bf v}^{(2)}_{\nu_2} \ldots \otimes {\bf v}^{(d)}_{\nu_d}, 
\end{equation*}
where $\{{\bf v}^{(\ell)}_{\nu_\ell}\}_{\nu_\ell=1}^{r_\ell}\in \mathbb{R}^{n_\ell}$, $\ell=1,\ldots,d$
represents a set of orthonormal vectors 
and $\boldsymbol{\beta}=[\beta_{\nu_1,\ldots,\nu_d}] \in \mathbb{R}^{r_1\times \cdots \times r_d}$ is 
the Tucker core tensor. 
However, the complexity of the Tucker tensor decomposition algorithm \cite{DMV-SIAM2:00} is
$O(n^{d+1})$ and it requires the tensor in full size format. This step is called  
the higher order singular value decomposition (HOSVD).  

These tensor decompositions yield rank-structured representation of tensors
(for $d=2$ these are rank-structured matrices), which provide the reduction of the operations
with tensor to one-dimensional operations. Given two tensors in the canonical tensor format,
\[
{\bf A}_1 =\sum_{k =1}^{R_1} c_{k}
{\bf u}^{(1)}_{k}  \otimes \ldots \otimes {\bf u}^{(d)}_{k},\quad
{\bf A}_2 =\sum_{m =1}^{R_2} b_{m}
{\bf v}^{(1)}_{m}  \otimes \ldots \otimes {\bf v}^{(d)}_{m}.
\]
 the  Euclidean scalar product   
\[
\langle {\bf A}_1, {\bf A}_2 \rangle:=
\sum\limits_{k=1}^{R_1}
\sum\limits_{m=1}^{R_2}
c_k b_m  \prod\limits_{\ell=1}^d
\left\langle {\bf u}^{(\ell)}_k,  {\bf v}^{(\ell)}_m \right\rangle,
\] 
 and Hadamard product of tensors $A_1$, $A_2$, 
\[
{\bf A}_1 \odot {\bf A}_2    :=
\sum\limits_{k=1}^{R_1} \sum\limits_{m=1}^{R_2}
c_k b_m \left(  {\bf u}_k^{(1)} \odot  {\bf v}_m^{(1)} \right) \otimes \ldots \otimes
\left(  {\bf u}_k^{(d)} \odot  {\bf v}_m^{(d)} \right).
\]
are computed in $O(d R_1 R_2 n)$ complexity.
The rank of the resulting tensor  is a product of the original ranks of  tensors.

In multilinear algebra the Tucker tensor decomposition was used in chemometrics, 
psychometrics, and signal processing for the quantitative analysis of the experimental data,
without special demands on accuracy and data size.
These techniques could not be applied for the  usage in numerical analysis of PDEs,
with large data arrays and high accuracy requirements.
Also, for general type tensors given in the rank-$R$ canonical format, 
%${\bf T}_C\in \mbox{\boldmath{$\mathcal{C}$}}_{{ R},{\bf n}}$,
with large ranks and with large  mode size $n$, both construction of the full size tensor 
representation and HOSVD become  intractable. 

However, it was found  in \cite{Khor1:06} that for function related tensors the Tucker tensor 
decomposition  exhibits exceptional approximation properties. 
In particular, it was proven and demonstrated numerically  that 
for a class of higher order tensors arising from the discretization of linear operators and 
functions in $\mathbb{R}^d$ using $n\times n  \ldots \times n$ Cartesian grids
the approximation error of the Tucker decomposition decays exponentially in the Tucker 
rank \cite{Khor1:06,KhKh:06}. These findings motivated introducing the multigrid Tucker tensor  
decomposition \cite{KhKh3:08} which is used in this paper to transform the reshaped 
discretized elliptic operators, being the fully populated 3D tensors,
into the low-rank canonical tensors. 
The main advantage of the multigrid Tucker decomposition is the elimination of the 
HOSVD for large grids, thus reducing the required storage to the maximum size of the tensor,
$O(n^d)$, instead of $O(n^{d+1})$. 
The last step in the transformation of the full tensor to canonical format is performed 
by using the Tucker-to-canonical algorithm \cite{Khor2Book2:18}.

Main motivation for tensor numerical methods in scientific computing
was the invention of the canonical-to-Tucker (C2T) decomposition and the reduced HOSVD (RHOSVD)
(introduced in \cite{KhKh3:08}, see detailed description in \cite{Khor2Book2:18,KhorBook:18})
which does not require the construction of a full size tensor. 
% Another motivation for tensor numerical methods in scientific computing 
% was the invention of the canonical-to-Tucker (C2T) decomposition and of the RHOSVD 
% \cite{KhKh3:08,Khor2Book2:18,KhorBook:18},
% which does not require the construction of a full size tensor. 
The complexity of the RHOSVD is  $O(d n^2 R)$, where $R$ is the canonical rank, %(or $O(n  R^2)$ if $R\le n$), 
and it applies to any dimension size $d$.

Tensor operations with canonical tensors, provide an advantage of
one-dimensional complexity of $d$-dimensional operations.
However, these operations lead to ``curse of ranks'', since the ranks are multiplied and 
after several operations calculations become intractable.
The C2T transform provides a robust tool for the rank reduction
of the canonical tensors. The combination of C2T algorithm with the Tucker-to-canonical 
transform\footnote{This decomposition applies to a small size core tensor 
%see Figure \ref{fig:Tucker_can} 
for the mixed Tucker-canonical format of type  
\[
 {\bf T}_{({\bf r})}=\left(\sum\limits_{k=1}^R b_k {\bf u}_k^{(1)} 
\otimes\ldots\otimes {\bf u}_k^{(d)} \right)
\times_1 V^{(1)}\times_2 V^{(2)} \times_3 \ldots 
   \times_d V^{(d)},  
\]
providing the canonical tensor rank of the order of $R\leq r^2$ in 3D case \cite{Khor2Book2:18}.}    
is the main working horse in all rank-truncation procedures.
In this paper, these transforms are used to reduce the ranks of the involved
quantities in the course of PCG iteration for the numerical solution of the 3D control problems
with fractional elliptic operators in constraints.

 \section*{Appendix 2: Precoditioned CG iteration in low-rank tensor formats }
\label{sec:low-rank_PCG}

\begin{algorithm}
	\caption{Preconditioned CG method in low-rank format} \label{alg:pcg}
	\begin{algorithmic}[1]
		\Require{Rank truncation procedure $\mathtt{trunc}$, rank tolerance parameter $\varepsilon$,
		linear function in low-rank format $\mathtt{fun}$, preconditioner in low-rank format
		$\mathtt{precond}$, right-hand side tensor $\mathbf{B}$, initial guess $\mathbf{X}^{(0)}$,
		maximal iteration number $k_{\max}$}
			\State $\mathbf{R}^{(0)} \leftarrow \mathbf{B} - \mathtt{fun}(\mathbf{X}^{(0)})$
			\State $\mathbf{Z}^{(0)} \leftarrow \mathtt{precond}(\mathbf{R}^{(0)})$
			\State $\mathbf{Z}^{(0)} \leftarrow \mathtt{trunc}(\mathbf{Z}^{(0)},\varepsilon)$
			\State $\mathbf{P}^{(0)} \leftarrow (\mathbf{Z}^{(0)})$
			\State $k \leftarrow 0$
			\Repeat
				\State $\mathbf{S}^{(k)} \leftarrow
				\mathtt{fun}(\mathbf{P}^{(k)})$
				\State {\color{red}$\mathbf{S}^{(k)} \leftarrow
				\mathtt{trunc}(\mathbf{S}^{(k)},\varepsilon)$}
				\State $\alpha_k \leftarrow
				\frac{\dotprod{\mathbf{R}^{(k)}}{\mathbf{Z}^{(k)}}}
				{\dotprod{\mathbf{P}^{(k)}}{\mathbf{S}^{(k)}}}$
				\State $\mathbf{X}^{(k+1)} \leftarrow
				\mathbf{X}^{(k)} + \alpha_k \mathbf{P}^{(k)}$
				\State {\color{red}$\mathbf{X}^{(k+1)} \leftarrow
				\mathtt{trunc}(\mathbf{X}^{(k+1)},\varepsilon)$}
				\State $\mathbf{R}^{(k+1)} \leftarrow
				\mathbf{R}^{(k)} - \alpha_k \mathbf{S}^{(k)}$
				\State {\color{red}$\mathbf{R}^{(k+1)} \leftarrow
				\mathtt{trunc}(\mathbf{R}^{(k+1)},\varepsilon)$}
				\If {$\mathbf{R}^{(k+1)}$ is sufficiently small}
					\State \Return $\mathbf{X}^{(k+1)}$
					\State \textbf{break}
				\EndIf
				\State $\mathbf{Z}^{(k+1)} \leftarrow
				\mathtt{precond}(\mathbf{R}^{(k+1)})$
				\State {\color{red}$\mathbf{Z}^{(k+1)} \leftarrow
				\mathtt{trunc}(\mathbf{Z}^{(k+1)},\varepsilon)$}
				\State $\beta_k \leftarrow
				\frac{\dotprod{\mathbf{R}^{(k+1)}}{\mathbf{Z}^{(k+1)}}}
				{\dotprod{\mathbf{Z}^{(k)}}{\mathbf{R}^{(k)}}}$
				\State $\mathbf{P}^{(k+1)} \leftarrow
				\mathbf{Z}^{(k+1)} + \beta_k \mathbf{P}^{(k)}$
				\State {\color{red} $\mathbf{P}^{(k+1)} \leftarrow
				\mathtt{trunc}(\mathbf{P}^{(k+1)},\varepsilon)$}
				\State $k \leftarrow k+1$
			\Until{$k = k_{\max}$}
			\Ensure{Solution $\mathbf{X}$ of $\mathtt{fun}(\mathbf{X})=\mathbf{B}$}
	\end{algorithmic}
\end{algorithm}
As the rank truncation procedure,
in our implementation we apply the reduced SVD algorithm in 2D case and the RHOSVD based 
canonical-to-Tucker-to-canonical algorithm (see \cite{KhKh3:08}) 
as described in Section \ref{sec:low-rank_tensor}.

%      
%  \section{Part II: Application in phase-field approaches to structural topology optimization}
% \label{sec:phase-field}   
%     
% In part II of this paper we consider fast rank-structured solvers for 
% the equation (e.g. for $\alpha= 1$)
% \[
%  Au = F,
% \]
% with general rank-$R$ separable coefficients in the form
% \[
%  A= -\sum_{\ell=1}^d \frac{d}{d x_\ell} 
%  \left(\sum_{k=1}^R \prod_{m=1}^d a_\ell^{m,k}(x_m)\right) \frac{d}{d x_\ell}.
% \]
% This solver is served for applications in phase-field approaches to structural 
% topology optimization.
% The first tests can be performed with $R=1$.
 
 \section*{Acknowledgment}
This research has been supported by the German Research Foundation (DFG) within 
the \textit{Research Training Group 2126: Algorithmic Optimization}, Department of Mathematics, University of Trier, Germany.

\begin{footnotesize}

\end{footnotesize}

% \section{Part II: Application in phase-field approaches to structural topology optimization}
%\label{sec:phase-field}   
%    
%In part II of this paper we consider fast rank-structured solvers for 
%the equation (e.g. for $\alpha= 1$)
%\[
% Au = F,
%\]
%with general rank-$R$ separable coefficients in the form
%\[
% A= -\sum_{\ell=1}^d \frac{d}{d x_\ell} 
% \left(\sum_{k=1}^R \prod_{m=1}^d a_\ell^{m,k}(x_m)\right) \frac{d}{d x_\ell}.
%\]
%This solver is served for applications in phase-field approaches to structural 
%topology optimization.
%The first tests can be performed with $R=1$.

\end{document}